\documentclass[]{interact}
\usepackage{epstopdf}
\usepackage{subfigure}

\usepackage[numbers,sort&compress,merge]{natbib}
\usepackage{amsmath,amssymb}
\usepackage{hyperref}
\usepackage{mathrsfs}
\usepackage{amsfonts}
\usepackage{mathtools}
\usepackage{tikz}
\usepackage{lmodern}
\usepackage{bigints}
\usepackage{relsize}
\usepackage{bbm}
\usepackage{amsthm}
\bibpunct[, ]{(}{)}{,}{n}{,}{,}

\theoremstyle{plain}
\newtheorem{theorem}{Theorem}[section]
\newtheorem{lemma}[theorem]{Lemma}
\newtheorem{corollary}[theorem]{Corollary}

\theoremstyle{definition}
\newtheorem{definition}[theorem]{Definition}

\theoremstyle{remark}
\newtheorem{remark}{Remark}
\newtheorem{notation}{Notation}
\newcommand{\norm}[1]{\left\lVert#1\right\rVert}
\DeclareMathOperator{\sech}{sech}

\DeclareMathOperator{\hess}{Hess}

\begin{document}


\title{On the kink-kink collision problem for the $\phi^{6}$ model\\
with low speed}

\author{
\name{Abdon Moutinho \textsuperscript{a}\thanks{CONTACT Abdon Moutinho Email: moutinho@math.univ-paris13.fr}}
\affil{\textsuperscript{a} LAGA, Université Sorbonne Paris Nord, Villetaneuse, France}
}

\maketitle

\begin{abstract}
 We study the elasticity of the collision of two kinks with an incoming low speed $v\in (0,1)$ for the nonlinear wave equation in dimension $1+1$ known as the $\phi^{6}$ model. We prove for any $k\in\mathbb{N}$ that if the incoming speed $v$ is small enough, then, after the collision, the two kinks will move away with a velocity $v_{f}$ such that $\vert v_{f}-v\vert\leq v^{k}$ and the energy of the remainder will also be smaller than $v^{k}.$ This manuscript is the continuation of our previous paper where we constructed a sequence $\phi_{k}$ of approximate solutions for the $\phi^{6}$ model. The proof of our main result relies on the use of the set of approximate solutions from our previous work, modulation analysis, and a refined energy estimate method to evaluate the precision of our approximate solutions during a large time interval. 
\end{abstract}


\section{Introduction}
\subsection{Background}
Considering the potential function $U(\phi)=\phi^{2}(1-\phi^{2})^{2},$ the partial differential equation known as the $\phi^{6}$ model in domain $1+1$ is defined by:
\begin{equation}\label{nlww}
    \partial_{t}^{2}\phi(t,x)-\partial_{x}^{2}\phi(t,x)+ U^{'}(\phi(t,x))=0 \text{, $(t,x) \in \mathbb{R}\times \mathbb{R}$}.
\end{equation}
\par The solutions $\phi(t,x)$ of \eqref{nlww} 
preserve the energy given by
\begin{equation}\label{energy}\tag{Energy}
    E(\phi)(t)=\int_{\mathbb{R}}\frac{\left[\partial_{t}\phi(t,x)\right]^{2}+\left[\partial_{x}\phi(t,x)\right]^{2}}{2}+U(\phi(t,x))\,dx,
\end{equation}
and
the momentum 
\begin{equation}\label{momentum}\tag{Momentum}
    P(\phi)={-}\int_{\mathbb{R}}\partial_{t}\phi(t,x)\partial_{x}\phi(t,x)\,dx.
\end{equation}
The kinetic energy and potential energy are given, respectively, by
\begin{equation*}
    E_{kin}(\phi)(t)=\int_{\mathbb{R}}\frac{\left[\partial_{t}\phi(t,x)\right]^{2}}{2}\,dx,\quad
    E_{pot}(\phi)(t)=\int_{\mathbb{R}}\frac{\left[\partial_{x}\phi(t,x)\right]^{2}}{2}+U(\phi(t,x))\,dx.
\end{equation*}
\par The vacuum set $\mathcal{V}$ of the potential function $U$ is the set $U^{{-}1}\{0\}=\{{-}1,0,1\}.$ The unique constant solutions with finite energy of \eqref{nlww} are the functions of the form $\phi\equiv \eta,$ for any $\eta\in \mathcal{V}.$  
\par Furthermore, it is well known that if a solution $\phi(t,x)$ of the partial differential equation \eqref{nlww} is in the energy space, which is the set of strong solutions with finite energy, then the solution is global-in-time and there exist numbers $\eta_{1},\,\eta_{2}\in\mathcal{V}$ such that
\begin{equation*}
    \lim_{x\to{-}\infty}\phi(t,x)=\eta_{1},\,\, \lim_{x\to{+}\infty}\phi(t,x)=\eta_{2},
\end{equation*}
for all $t\in\mathbb{R}.$
The set of solutions of \eqref{nlww} with finite energy is invariant under space translation, time translation, space reflection, time reflection, and Lorentz transformations.
\par The unique non-constant stationary solutions of \eqref{nlww} with finite energy are the kinks which are the space translation of either $H_{0,1}(x)$ or $H_{-1,0}(x)$ that are denoted by
\begin{equation*}
    H_{0,1}(x)=\frac{e^{\sqrt{2}x}}{\sqrt{1+e^{2\sqrt{2}x}}},\quad H_{-1,0}(x)=-H_{0,1}({-}x)=\frac{-e^{{-}\sqrt{2}x}}{\sqrt{1+e^{{-}2\sqrt{2}x}}},
\end{equation*}
 and the anti-kinks which are the space translation of the following functions \begin{equation*}
    H_{1,0}(x)=H_{0,1}({-}x)=\frac{e^{-\sqrt{2}x}}{\sqrt{1+e^{-2\sqrt{2}x}}},\quad H_{0,-1}(x)=-H_{0,1}(x)=\frac{-e^{\sqrt{2}x}}{\sqrt{1+e^{2\sqrt{2}x}}}.
\end{equation*}
 Using the identity $ H^{'}_{0,1}(x)=\sqrt{2}\frac{e^{\sqrt{2}x}}{\left(1+e^{2\sqrt{2}x}\right)^{\frac{3}{2}}},$ it is not difficult to verify that
\begin{equation}\label{kinkmass}
    \norm{\frac{d}{dx} H_{0,1}(x)}_{L^{2}_{x}(\mathbb{R})}^{2}=\frac{1}{2\sqrt{2}}.
\end{equation}
The kink $H_{0,1}$ satisfies the Bogomolny identity, which is 
$
    H^{'}_{0,1}(x)=\sqrt{2U(H_{0,1}(x))},
$
and the following estimates for any $k\geq 1$
\begin{equation}\label{le2}
    \left\vert \frac{d^{k}}{dx^{k}}H_{0,1}(x) \right\vert\lesssim_{k} \min\left(e^{\sqrt{2}x},e^{-2\sqrt{2}x}\right),
\end{equation}
and clearly
\begin{equation}\label{le1}
    \left\vert H_{0,1}(x) \right\vert\leq e^{\sqrt{2}\min(x,0)}.
\end{equation}
\par For the $\phi^{6}$ model there are stability results for the kinks. In \cite{first}, the orbital stability of two kinks with energy close to the minimal was obtained, and also the dynamics of two interacting kinks, which is a kink-kink solution with low kinetic energy and potential energy slightly bigger than the minimum possible for two kinks, was described in function of the initial data and the energy of the solution. In \cite{asympt}, the asymptotic stability of a kink for the $\phi^{6}$ model was obtained, moreover, asymptotic stability of a single kink was also obtained for a certain class of nonlinear wave equations of dimension $1+1.$ There are also asymptotic stability results for a single kink in other models, for example see \cite{munoz} and \cite{kinkdelort} for the $\phi^{4}$ model.
\par This manuscript is the sequel of the work done in \cite{second}. In this paper, we study the traveling kink-kink solutions of \eqref{nlww} with speed $0<v<1$ small enough. More precisely, we consider the following definition.
\begin{definition}
The traveling kink-kink with speed $v\in(0,1)$ are the set of solutions $\phi(t,x)$ that satisfies for some positive constants $K,\,c$ and any $t\geq K$ the following decay estimate
\begin{equation}\label{expkinkink}
   \norm{ \left(\phi(t,x),\partial_{t}\phi(t,x)\right)-\overrightarrow{H_{0,1}}\left(\frac{x-vt}{\sqrt{1-v^{2}}}\right)-\overrightarrow{H_{-1,0}}\left(\frac{x+vt}{\sqrt{1-v^{2}}}\right)}_{H^{1}_{x}(\mathbb{R})\times L^{2}_{x}(\mathbb{R})}\leq e^{-ct},
\end{equation}
where, for any $-1<v<1$ and any $y\in\mathbb{R},$
\begin{align}\label{vetorh1}
    \overrightarrow{H_{0,1}}\left(\frac{x-vt+y}{\sqrt{1-v^{2}}}\right)=&
    \begin{bmatrix}
     H_{0,1}\left(\frac{x-vt+y}{\sqrt{1-v^{2}}}\right)\\
     \frac{-v}{\sqrt{1-v^{2}}} H^{'}_{0,1}\left(\frac{x-vt+y}{\sqrt{1-v^{2}}}\right)
    \end{bmatrix},\\ \label{vetorh2} 
    \overrightarrow{H_{-1,0}}\left(\frac{x+vt-y}{\sqrt{1-v^{2}}}\right)=&\begin{bmatrix}
     H_{-1,0}\left(\frac{x+vt-y}{\sqrt{1-v^{2}}}\right)\\
     \frac{v}{\sqrt{1-v^{2}}} H^{'}_{-1,0}\left(\frac{x+vt-y}{\sqrt{1-v^{2}}}\right)
     \end{bmatrix}.
\end{align}
\end{definition}
The existence and uniqueness for any $0<v<1$ of solutions $\phi(t,x)$ satisfying \eqref{expkinkink} was obtained in \cite{multison}, but the uniqueness of the solution of \eqref{nlww} satisfying for $0<v<1$
\begin{equation*}
   \lim_{t\to+\infty}\norm{ \overrightarrow{\phi}(t,x)-\overrightarrow{H_{0,1}}\left(\frac{x-vt}{\sqrt{1-v^{2}}}\right)+\overrightarrow{H_{-1,0}}\left(\frac{x+vt}{\sqrt{1-v^{2}}}\right)}_{H^{1}_{x}(\mathbb{R})\times L^{2}_{x}(\mathbb{R})}=0
\end{equation*}
is still an open problem. For references on the existence and uniqueness of multi-soliton solutions of other nonlinear dispersive partial differential equations, see for example \cite{multigkdv} and \cite{multisupergkdv}. 
\par For non-integrable dispersive models, there exist previous results about the inelasticity of the collision of two solitons. For example, in the article \cite{collision1}, Martel and Merle verified that the collision between two solitons with nearly equal speed is not elastic. More precisely, they obtained that the incoming speed of the two solitons is different of their outgoing speed after their collision. \par Since the $\phi^{6}$ model is a non-integrable system, 
the collision of two kinks with low speed $0<v<1$ is expected to be inelastic. More precisely, we were expecting the existence of a value $k>1$ such that if $0<v\ll 1$ and $\phi(t,x)$ is a solution \eqref{nlww} satisfying the condition \eqref{expkinkink}, 
then $\phi(t,x)$ should have inelasticity of order $v^{k},$ which means the existence of $t<0$ with $\vert t\vert \gg 1$ such that 
\begin{equation}\label{contr1}
\left(\phi(t,x),\partial_{t}\phi(t,x)\right)=\overrightarrow{H_{0,1}}\left(\frac{x+v_{f}t+y_{1}(t)}{\sqrt{1-v_{f}^{2}}}\right)+\overrightarrow{H_{{-}1,0}}\left(\frac{x-v_{f}t+y_{2}(t)}{\sqrt{1-v_{f}^{2}}}\right)+ro(t,x),
\end{equation}
with $v^{k}\ll\norm{ro(t)}_{H^{1}_{x}(\mathbb{R})\times L^{2}_{x}(\mathbb{R})}\ll v$ and $v_{f}(t),\,y_{1},\,y_{2}$ satisfying
\begin{equation}\label{contr2}
    v^{k}\ll\left\vert v_{f}(t)-v\right\vert+\max_{j\in\{1,2\}}\vert \dot y_{j}(t)\vert\ll v,
\end{equation}
  for all $t<0$ satisfying $\vert t\vert\gg 1.$
Actually, in the quartic $gKdV,$ the collision of the two solitons satisfies a similar property than our previous expectations in \eqref{contr1} and \eqref{contr2}, see Theorem $1$ in the article \cite{collision1} of Martel and Merle for more details.    
\par However, in this manuscript, we prove for the $\phi^{6}$ model and any $k>1$ that if $0<v\ll 1$ and $t$ is close to $-\infty,$ both estimates \eqref{contr1} and \eqref{contr2} are not possible. Indeed, we demonstrate that if $v\ll 1$ and $\phi(t,x)$ satisfies \eqref{expkinkink}, then there exists a number $e_{k,2v}\in\mathbb{R}$ satisfying, for all $t$ close to $-\infty,$ 
\begin{equation*}
   \left(\phi(t,x),\partial_{t}\phi(t,x)\right)=\overrightarrow{H_{0,1}}\left(\frac{x+v_{f}t-e_{k,2v}}{\sqrt{1-v_{f}^{2}}}\right)+\overrightarrow{H_{-1,0}}\left(\frac{x-v_{f}t+e_{k,2v}}{\sqrt{1-v_{f}^{2}}}\right)+r_{c,v}(t,x),
\end{equation*}
$\limsup_{t\to-\infty}\norm{r_{c,v}(t)}_{H^{1}_{x}\times L^{2}_{x}}\leq v^{2k}$ and
\begin{equation}\label{elasticity22}
   \limsup_{t\to-\infty} \vert v_{f}(v,t)-v\vert\leq v^{2k}.
\end{equation}
In conclusion, the inelasticity of the collision of two kinks cannot be of any order $v^{k}$ for any $1\ll k\in\mathbb{N},$ if the incoming speed $v$ of the kinks is small enough. The problem to verify the inelasticity of the collision of kinks for the $\phi^{6}$ model is still open. But, because of the conclusion obtained in this paper, the change $\vert v-v_{f}\vert$ in the speeds of each soliton is much smaller than any monomial function $v^{k},$ more precisely for all $k>0$
\begin{equation}\label{newnewnew}
    \lim_{v\to 0^{+}}\limsup_{t\to{-}\infty}\frac{\vert v_{f}(v,t)-v\vert}{v^{k}}=0,
\end{equation}
which is a new result. 
\par The study of collision of kinks for the $\phi^{6}$ model is important for high energy physics, see for example \cite{kinkcollision} and \cite{collision}. Actually, in the article \cite{kinkcollision}, it was obtained numerically the existence of a critical speed $v_{c}$ such that if each of the two kinks moves with speed $v$ with absolute value less than $v_{c}$ and they approach each other, then they will collide and the collision will be very elastic, which is exactly the result we obtained rigorously in this paper. The study of the dynamics of multi-soliton solutions of the $\phi^{6}$ model has also applications in condensed matter physics, see \cite{condensed}, and cosmology, see \cite{cosmic}. 
\par For other nonlinear dispersive equations, there exist rigorous results of inelasticity and stability of collision of solitons. For  $gKdV$ models, the inelasticity of collision of solitons was proved  for the quartic $gKdV$ in \cite{collision1}, and, for a certain class of generalized $gKdV,$ inelasticity of collision between solitons was also proved in \cite{munozkdv1} and  \cite{munozkdv2} by Muñoz, see also the article \cite{stabcol} of Martel and Merle. For nonlinear Schrödinger equation, in \cite{perelman}, Perelman studied the collision of two solitons of different sizes and obtained that after that the solution does not preserve the two solitons' structure after the collision. See also the work \cite{collisioncw} by Martel and Merle about inelasticity of the collision of two solitons for the fifth dimensional energy critical wave equation.
\subsection{Main Results}
The main theorem obtained in this manuscript is the following result:
\begin{theorem}\label{maintheo}
There exists a continuous function $v_{f}:(0,1)\times \mathbb{R}\to (0,1)$ and, for any $0<\theta<1$ and $k\in\mathbb{N}_{\geq 2},$ there exists $0<\delta(\theta,k)<1,$ such that if $0<v<\delta(\theta,k),$ and $\phi(t,x)$ is a travelling kink-kink solution of \eqref{nlww} with speed $v,$ then there exists a number $e_{v,k}$ such that $\vert e_{v,k}\vert <\ln\left(\frac{8}{v^{2}}\right)$ and if $t\leq -\frac{\ln{\left(\frac{1}{v}\right)}^{2-\theta}}{v},$ then $\vert v_{f}(v,t)-v\vert<v^{k}$ and 
     \begin{multline*}
    \norm{\phi(t,x)-H_{0,1}\left(\frac{x-e_{k,v}+v_{f}t}{\sqrt{1-v_{f}^{2}}}\right)-H_{-1,0}\left(\frac{x+e_{k,v}-v_{f}t}{\sqrt{1-v_{f}^{2}}}\right)}_{H^{1}_{x}(\mathbb{R})}\\ {+}\norm{\partial_{t}\phi(t,x)-\frac{v_{f}}{\sqrt{1-v_{f}^{2}}} H^{'}_{0,1}\left(\frac{x-e_{v,k}+v_{f}t}{\sqrt{1-v_{f}^{2}}}\right)+\frac{v_{f}}{\sqrt{1-v_{f}^{2}}} H^{'}_{-1,0}\left(\frac{x+e_{v,k}-v_{f}t}{\sqrt{1-v_{f}^{2}}}\right)}_{L^{2}_{x}(\mathbb{R})}\leq v^{k}.
    \end{multline*}
    If $\frac{-4\ln{\left(\frac{1}{v}\right)}^{2-\theta}}{v}\leq t\leq \frac{-\ln{\left(\frac{1}{v}\right)}^{2-\theta}}{v},$ then
    \begin{multline*}
        \norm{\phi(t,x)-H_{0,1}\left(\frac{x-e_{k,v}+vt}{\sqrt{1-v^{2}}}\right)-H_{-1,0}\left(\frac{x+e_{k,v}-vt}{\sqrt{1-v^{2}}}\right)}_{H^{1}_{x}(\mathbb{R})}\\{+}\norm{\partial_{t}\phi(t,x)-\frac{v}{\sqrt{1-v^{2}}} H^{'}_{0,1}\left(\frac{x-e_{v,k}+vt}{\sqrt{1-v^{2}}}\right)+\frac{v}{\sqrt{1-v^{2}}} H ^{'}_{-1,0}\left(\frac{x+e_{v,k}-vt}{\sqrt{1-v^{2}}}\right)}_{L^{2}_{x}(\mathbb{R})}\leq v^{k}.
    \end{multline*}
\end{theorem}
Clearly, Theorem \ref{maintheo} implies \eqref{newnewnew}.
Actually, the first item of Theorem \ref{maintheo} is a consequence of the second item of this theorem and the following result about orbital stability of two moving kinks.
\begin{theorem}\label{orbittheo}
There exists a constant $c>0$ and, for any $\theta\in(0,1),$ there exists $\delta(\theta)\in(0,1)$ such that if $0<v<\delta(\theta),$ and $(\psi_{0}(x),\psi_{1}(x))\in H^{1}_{x}(\mathbb{R})\times L^{2}_{x}(\mathbb{R})$ is an odd function satisfying
\begin{equation}
    \norm{(\psi_{0},\psi_{1})}_{H^{1}_{x}\times L^{2}_{x}}< v^{2+\theta},
\end{equation}
and $y_{0}\geq -4\ln{v},$ then the solution $(\phi(t,x),\partial_{t}\phi(t,x))$ of the Cauchy problem
\begin{equation}\label{nlwwsmooth}
    \begin{cases}
         \partial^{2}_{t}\phi(t,x)-\partial^{2}_{x}\phi(t,x)+\dot U(\phi(t,x))=0,\\
         \begin{bmatrix}
         \phi(0,x)\\
         \partial_{t}\phi(0,x)
         \end{bmatrix}=
         \begin{bmatrix}
         H_{0,1}\left(\frac{x-y_{0}}{\sqrt{1-v^{2}}}\right)+H_{{-}1,0}\left(\frac{x+y_{0}}{\sqrt{1-v^{2}}}\right)+\psi_{0}(x)\\
         \frac{{-}v}{\sqrt{1-v^{2}}} H^{'}_{0,1}\left(\frac{x-y_{0}}{\sqrt{1-v^{2}}}\right)+\frac{v}{\sqrt{1-v^{2}}} H^{'}_{{-}1,0}\left(\frac{x+y_{0}}{\sqrt{1-v^{2}}}\right)+\psi_{1}(x)
         \end{bmatrix}
    \end{cases}
\end{equation}
is given for all $t\geq 0$ by
\begin{equation}\label{modufor}
    \begin{bmatrix}
     \phi(t,x)\\
     \partial_{t}\phi(t,x)
    \end{bmatrix}
    =
    \begin{bmatrix}
         H_{0,1}\left(\frac{x-y(t)}{\sqrt{1-v^{2}}}\right)+H_{{-}1,0}\left(\frac{x+y(t)}{\sqrt{1-v^{2}}}\right)+\psi(t,x)\\
         \frac{{-}v}{\sqrt{1-v^{2}}} H^{'}_{0,1}\left(\frac{x-y(t)}{\sqrt{1-v^{2}}}\right)+\frac{v}{\sqrt{1-v^{2}}} H^{'}_{{-}1,0}\left(\frac{x+y(t)}{\sqrt{1-v^{2}}}\right)+\partial_{t}\psi(t,x)
    \end{bmatrix},
\end{equation}
such that
\begin{align}\label{globalorbit}
  \vert y(0)-y_{0} \vert + \norm{\overrightarrow{\psi}(t,x)}_{H^{1}_{x}\times L^{2}_{x}}\leq &c \norm{\overrightarrow{\psi_{0}}(x)}_{H^{1}_{x}\times L^{2}_{x}}^{\frac{1}{2}}+c(1+y_{0})^{\frac{1}{2}}e^{{-}\sqrt{2}y_{0}},\\ \nonumber
  \vert \dot y(t) -v\vert\leq &c\norm{\overrightarrow{\psi_{0}}(x)}_{H^{1}_{x}\times L^{2}_{x}},
\end{align}
for all $t\in\mathbb{R}_{\geq 0}.$
\end{theorem}

\subsection{Notation}
In this subsection, we explain the notation that we are going to use in the next sections. 
\begin{notation}
\par First, for any real function $f:\mathbb{R}^{2}\to\mathbb{R}$ satisfying the conditions $f(t,\cdot)\in L^{\infty}_{x}(\mathbb{R}),$ and $\partial_{t}f(t,\cdot)\in L^{2}_{x}(\mathbb{R}),$ we denote the function $\overrightarrow{f}:\mathbb{R}^{2}\to\mathbb{R}^{2}$ by
\begin{equation*}
    \overrightarrow{f}(t,x)=\left(f(t,x),\partial_{t}f(t,x)\right) \text{, for every $(t,x)\in\mathbb{R}^{2}.$}
\end{equation*}
\par For any $k\in\mathbb{N}$ and any smooth function $f:\mathbb{R}\to\mathbb{R},$ we use the following notation
\begin{equation*}
    f^{(k)}(x)=\frac{d^{x}}{dx^{k}}f(x) \text{, for all $x\in\mathbb{R}.$}
\end{equation*}
\par For any $z\in \mathbb{R},$ we use the following notation $H^{z}_{0,1}(x)=H_{0,1}(x-z),$ $H^{z}_{-1,0}(x)=H_{-1,0}(x-z).$ 
For any subset $\mathcal{D}\subset\mathbb{R},$ any $v\in (0,1)$ and any function $y:\mathcal{D}\to\mathbb{R},$ we define the functions $\overrightarrow{H_{0,1,v,y}}:\mathcal{D}\times\mathbb{R}\to\mathbb{R}^{2},\, \overrightarrow{H_{{-}1,0,v,y}}: \mathcal{D}\times\mathbb{R}\to\mathbb{R}^{2}$ by
\begin{align*}
\overrightarrow{H_{0,1,v,y}}(t,x)=&\begin{bmatrix}
H_{0,1}\left(\frac{x-v t+y(t)}{\sqrt{1-v^{2}}}\right)\\
     \frac{{-}v}{\sqrt{1-v^{2}}} H^{'}_{0,1}\left(\frac{x-vt+y(t)}{\sqrt{1-v^{2}}}\right)
     \end{bmatrix},\\
\overrightarrow{H_{{-}1,0,v,y}}(t,x)=&\begin{bmatrix}
H_{{-}1,0}\left(\frac{x+vt-y(t)}{\sqrt{1-v^{2}}}\right)\\
     \frac{v}{\sqrt{1-v^{2}}} H^{'}_{-1,0}\left(\frac{x+vt-y(t)}{\sqrt{1-v^{2}}}\right)
     \end{bmatrix}.
\end{align*}
\par For any set $\mathcal{D}\subset\mathbb{R}$ and any non-negative function $k:\mathcal{D}\to \mathbb{R}_{\geq 0},$ we say that $f(x)=O(k(x)),$ if $f$ has the same domain $\mathcal{D}$ as $k$ and there is a constant $C>0$ such that
$\vert f(x) \vert \leq C k(x) \text{ for any $x \in D.$}$
For any two non-negative real functions $f_{1}(x)$ and $f_{2}(x),$ we have that the condition $f_{1}\lesssim f_{2}$ is true if there is a constant $C>0$ such that $f_{1}(x)\leq C f_{2}(x) \text{ for any $x\in\mathbb{R}.$} $
Furthermore, for a finite number of real variables $\alpha_{1},\,...,\,\alpha_{n}$ and  two non-negative functions $f_{1}(\alpha_{1},...,\alpha_{n},x)$ and $f_{2}(\alpha_{1},...,\alpha_{n},x)$ both with domain $\mathcal{D}\times\mathbb{R}\subset \mathbb{R}^{n+1},$ we say that $f_{1}\lesssim_{\alpha_{1},..., \alpha_{n}}f_{2}$ if there is a positive function $L:\mathcal{D}\to \mathbb{R}_{+}$ such that
\begin{equation*}
    f_{1}(\alpha_{1},...,\alpha_{n},x)\leq L(\alpha_{1},\, ... , \alpha_{n})f_{2}(\alpha_{1},...,\alpha_{n},x) \text{ for all $(\alpha_{1},...,\alpha_{m},x)\in\mathcal{D}\times\mathbb{R}.$} 
\end{equation*}
We denote $f_{1}\cong f_{2}$ if $f_{1}\lesssim f_{2}$ and $f_{2}\lesssim f_{1}.$ 
\par We consider for any $f\in H^{1}_{x}(\mathbb{R})$ and any $g\in L^{2}_{x}(\mathbb{R})$ the following norms
\begin{equation*}
    \norm{f}_{H^{1}_{x}}=\norm{f}_{H^{1}_{x}(\mathbb{R})}=\left(\norm{f}_{L^{2}_{x}(\mathbb{R})}^{2}+\norm{\frac{df}{dx}}_{L^{2}_{x}(\mathbb{R})}^{2}\right)^{\frac{1}{2}},\quad
    \norm{g}_{L^{2}_{x}}=\norm{g}_{L^{2}_{x}(\mathbb{R})}.
\end{equation*}
In this manuscript, we consider the norm $\norm{\cdot}_{H^{1}_{x}\times L^{2}_{x}}$ given by
\begin{equation*}
    \norm{(f_{1}(x),f_{2}(x))}_{H^{1}_{x}\times L^{2}_{x}}=\left(\norm{f_{1}}_{H^{1}_{x}(\mathbb{R})}^{2}+\norm{f_{2}(x)}_{L^{2}_{x}(\mathbb{R})}^{2}\right)^{\frac{1}{2}},
\end{equation*}
for any $(f_{1},f_{2})\in H^{1}_{x}(\mathbb{R})\times L^{2}_{x}(\mathbb{R}).$ For any $(f_{1},f_{2}) \in L^{2}_{x}(\mathbb{R})\times L^{2}_{x}(\mathbb{R})$ and any $(g_{1},g_{2})\in L^{2}_{x}(\mathbb{R})\times L^{2}_{x}(\mathbb{R}),$ we denote
\begin{equation*}
    \left\langle (f_{1},f_{2}),\,(g_{1},g_{2}) \right\rangle=\int_{\mathbb{R}} f_{1}(x)g_{1}(x)+f_{2}(x)g_{2}(x)\,dx.
\end{equation*}
For any functions $f_{1}(x),\,g_{1}(x)\in L^{2}_{x}(\mathbb{R}),$ we denote
\begin{equation*}
    \left\langle f_{1},g_{1} \right\rangle=\int_{\mathbb{R}}f_{1}(x)g_{1}(x)\,dx.
\end{equation*}
\par In this manuscript, we consider the set $\mathbb{N}$ as the set of all positive integers.
For any for any $n\in\mathbb{N},$ and any $a,\,b\in\mathbb{R}^{n},$  we denote the scalar product in the Euclidean space $\mathbb{R}^{n}$ by 
\begin{equation*}
    \left\langle a : b\right\rangle=\sum_{j=1}^{n} a_{j}b_{j},
\end{equation*}
 where $a=\left(a_{1},...,a_{n}\right)$ and $b=\left(b_{1},...,b_{n}\right).$
 
\end{notation}
\subsection{Organization of the manuscript}
\par First, from the global well-posedness of the partial differential equation \eqref{nlww}, we recall that if $\phi$ is a strong solution of \eqref{nlww} with finite energy satisfying $\lim_{x\to\pm\infty}\phi(t_{0},x)=\pm 1$ for some $t_{0}\in\mathbb{R},$ then the function $\phi$ satisfies
\begin{equation*}
   \norm{\phi(t,x)-H_{0,1}(x)-H_{{-}1,0}(x)}_{H^{1}_{x}(\mathbb{R})}<{+}\infty,
\end{equation*}
for all $t\in\mathbb{R}.$
\par In the Subsection \ref{appsec} of Section \ref{presection}, we are going to review our preview results from the paper \cite{second} about the existence of a sequence of approximate solutions $\left(\varphi_{k,v}\right)_{k\geq 2}$ of \eqref{nlww} for which there exists a set of real numbers $\left(y_{k}(v)\right)_{k\geq 2}$ satisfying 
\begin{equation*}
   \lim_{t\to{+}\infty}\norm{\overrightarrow{\varphi_{k}}(t,x)-\overrightarrow{H_{0,1,v,y_{k}}}(t,x)-\overrightarrow{H_{{-}1,0,v,y_{k}}}(t,x)}_{H^{1}_{x}\times L^{2}_{x}}=0,
\end{equation*}
and if $v\ll 1,$ then $\norm{\partial^{l}_{t}\Lambda\left(\varphi_{k,v}\right)(t,x)}_{H^{s}_{x}}\lesssim_{s,l} v^{2k+l-\frac{1}{2}}e^{{-}2\sqrt{2}v\vert t\vert}$ for all $t\in\mathbb{R},\,l\in\mathbb{N}\cup\{0\},$ and $s\geq 0.$ 
\par In the Subsection \ref{auxaux} of Section \ref{presection}, we are going to verify that any solution of \eqref{nlww} with finite energy close to a sum of two kinks can be written as 
\begin{multline}\label{p1}
    \phi(t,x)=\varphi_{k,v}(t,x)+\frac{y_{1}(t)}{\sqrt{1-\frac{\dot d(t)^{2}}{4}}} H^{'}_{0,1}\left(\frac{x-\frac{d(t)}{2}+c_{k}(t)}{\sqrt{1-\frac{\dot d(t)^{2}}{4}}}\right)+
    \frac{y_{2}(t)}{\sqrt{1-\frac{\dot d(t)^{2}}{4}}} H^{'}_{0,1}\left(\frac{-x-\frac{d(t)}{2}+c_{k}(t)}{\sqrt{1-\frac{\dot d(t)^{2}}{4}}}\right)\\
    +u(t,x),
\end{multline}
such that, for any $t\in\mathbb{R},\,u(t)\in H^{1}_{x}(\mathbb{R})$ satisfies the following orthogonality conditions
\begin{align*}
    \left\langle u(t,x),\, H^{'}_{0,1}\left(\frac{x-\frac{d(t)}{2}+c_{k}(t)}{\sqrt{1-\frac{\dot d(t)^{2}}{4}}}\right) \right\rangle=0,\\
    \left\langle u(t,x),\, H^{'}_{0,1}\left(\frac{-x-\frac{d(t)}{2}+c_{k}(t)}{\sqrt{1-\frac{\dot d(t)^{2}}{4}}}\right) \right\rangle=0.
\end{align*}
Moreover, using $\Lambda(\phi)\equiv 0,$ we can verify that $y_{1},\,y_{2}\in C^{2}(\mathbb{R}).$
Furthermore, using the formula \eqref{p1}, we will estimate $\Lambda(\phi)(t,x).$ More precisely, we estimate the expression
$
   \Lambda\left(\phi\right)(t,x) -\Lambda\left(\varphi_{k,v}\right)(t,x),
$
in function of $y_{1}(t),\,y_{2}(t),\,d(t),\,\,u(t,x)$ and the estimate of the term
$
\Lambda\left(\varphi_{k,v}\right)(t,x)    
$ will follow from the main results of Subsection \ref{appsec} about the decay of approximate solutions.
The  function $c_{k}(t)$ will not appear in the evaluation of $\Lambda(\phi)(t,x),$ since we are going to use only its decay.
\par Next, in Section  \ref{energysection}, we are going to construct a function $L(t)$ to estimate $\norm{\left(u(t),\partial_{t}u(t)\right)}_{H^{1}_{x}\times L^{2}_{x}}$ during a large time interval.  
The main argument in this section is analogous to the ideas of Section $4$ of \cite{first}. More precisely, for
\begin{equation*}
    w_{k,v}(t,x)=\frac{x-\frac{d(t)}{2}+c_{k}(t)}{\sqrt{1-\frac{\dot d(t)^{2}}{4}}},
\end{equation*}
we consider first
\begin{equation*}
    L_{1}(t)=\int_{\mathbb{R}}\partial_{t}u(t,x)^{2}+\partial_{x}u(t,x)^{2}+ U^{''}\left(H_{0,1}\left(w_{k,v}(t,x)\right)-H_{0,1}\left(w_{k,v}(t,{-}x)\right)\right)u(t,x)^{2}\,dx.
\end{equation*}
From the orthogonality conditions satisfied by $u(t,x),$ if $v\ll 1,$ we deduce the following coercivity inequality 
\begin{equation*}
    \norm{\left(u(t),\partial_{t}u(t)\right)}_{H^{1}_{x}\times L^{2}_{x}}^{2}\lesssim L_{1}(t).
\end{equation*}
The function $L(t)$ will be constructed after correction terms $L_{2}(t)$ and $L_{3}(t)$ are added to $L_{1}(t).$  The motivation for the usage of the correction term $L_{3}(t)$ is to reduce the growth of the modulus of the following expression
\begin{multline*}
    2\int_{\mathbb{R}}\left[\partial^{2}_{t}u(t,x)-\partial^{2}_{x}u(t,x)+ U^{''}\left(H_{0,1}\left(w_{k,v}(t,x)\right)-H_{0,1}\left(w_{k,v}(t,x)\right)\right)u(t,x)\right]\partial_{t}u(t,x)\,dx
\end{multline*}
 in $\dot L_{1}(t).$ The time derivative of $L_{2}(t)$ will cancel with the expression
 \begin{equation*}
     \int_{\mathbb{R}}\frac{\partial}{\partial t}\left[ U^{''}\left(H_{0,1}\left(w_{k,v}(t,x)\right)-H_{0,1}\left(w_{k,v}(t,x)\right)\right)\right]u(t,x)^{2}\,dx,
 \end{equation*}
 from $\dot L_{1}(t).$
 Finally, under additional conditions in the growth of the functions $y_{1}(t),\,y_{2}(t),$ if $0<v\ll 1,$ the function $L(t)=\sum_{j=1}^{3}L_{j}(t)$ will satisfy for a constant $C(k)$ depending only on $k$ the following estimates  
 \begin{align*}
     \left\vert \dot L(t)\right\vert\lesssim &\frac{v}{\ln{\left(\frac{1}{v}\right)}}\norm{(u(t),\partial_{t}u(t))}_{H^{1}_{x}\times L^{2}_{x}}^{2},\\
     \norm{(u(t),\partial_{t}u(t))}_{H^{1}_{x}\times L^{2}_{x}}^{2}\lesssim& L(t)+C(k) v^{4k}\ln{\left(\frac{1}{v}\right)}^{2n_{k}},
 \end{align*}
for all $t$ in a large time interval, $n_{k}$ is the number denoted in Theorem \ref{toobig}. Therefore, using Gronwall Lemma and the two estimates above, we are going to obtain an upper bound for $\norm{(u(t),\partial_{t}u(t))}_{H^{1}_{x}\times L^{2}_{x}}$ when $t$ belongs to a large time interval.
\par In Section \ref{p1main}, we are going to estimate $\norm{\phi(t)-\phi_{k,v}(t)}_{H^{1}_{x}\times L^{2}_{x}}$
during a large time interval. This estimate follows from the study of a linear ordinary differential system whose solutions $\hat{y}_{1},\,\hat{y}_{2}$ are close to $y_{1},\,y_{2}$ during a time interval of size much larger than $\frac{-\ln{(v)}}{v}$ and from the conclusions of the last section. Indeed, the closeness of the functions $y_{1},\,y_{2}$ with $\hat{y}_{1},\,\hat{y}_{2}$ during this large time interval is guaranteed because of the upper bound obtained for $\norm{(u(t),\partial_{t}u(t))}_{H^{1}_{x}\times L^{2}_{x}}$ from the control of $L(t),$ which implies that $y_{1},\,y_{2}$ will satisfy a ordinary differential system very close to the linear ordinary differential system satisfied by $\hat{y}_{1}$ and $\hat{y}_{2}.$ 
\par In Section \ref{kinkstravel}, we are going to prove Theorem \ref{orbittheo}, the proof of this result is inspired in the demonstration of Theorem $1$ of \cite{asympt} and Theorem $1$ of \cite{orbitalsch}. This result will imply in the next section the second item of Theorem \ref{maintheo}. In addition, the main techniques used in this section are modulation techniques based on section $2$ of \cite{asympt} and based on \cite{orbitalsch}, the use of conservation of energy of $\phi(t,x)$ and the monotonicity of the localized momentum given by
\begin{equation*}
    P_{+}\left(\phi(t),\partial_{t}\phi(t)\right)=-\int_{0}^{+\infty}\partial_{t}\phi(t,x)\partial_{x}\phi(t,x)\,dx.
\end{equation*}
\par Finally, in Section \ref{colsec}, we will show that the demonstration of Theorem \ref{maintheo} is a direct consequence of the main results of Sections \ref{p1main} and \ref{kinkstravel}. For complementary information, see the Appendices sections.
\section{Preliminaries}\label{presection}
\subsection{Approximate solutions}\label{appsec}
\begin{definition}
We define $\Lambda$ as the nonlinear operator with domain $C^{2}(\mathbb{R}^{2},\mathbb{R})$ that satisfies:
\begin{equation*}
    \Lambda (\phi_{1})(t,x)=\partial_{t}^{2}\phi_{1}(t,x)-\partial_{x}^{2}\phi_{1}(t,x)+\dot U(\phi_{1}(t,x)),
\end{equation*}
for any $\phi_{1}(t,x) \in C^{2}(\mathbb{R}^{2},\mathbb{R}).$
\end{definition}
\par In the paper \cite{second}, we constructed a sequence of approximate solutions $(\phi_{k}(v,t,x))_{k\in \mathbb{N}_{\geq 2}}$ of the partial differential equation \eqref{nlww} such that
\begin{align*}
    \lim_{t\to{+}\infty}\norm{\phi_{k}(v,t,x)-H_{0,1}\left(\frac{x-vt}{\sqrt{1-v^{2}}}\right)-H_{-1,0}\left(\frac{x+vt}{\sqrt{1-v^{2}}}\right)}_{H^{1}_{x}}=& 0,\\
    \lim_{t\to{+}\infty}\norm{\partial_{t}\phi_{k}(v,t,x)+\frac{v}{\sqrt{1-v^{2}}} H^{'}_{0,1}\left(\frac{x-vt}{\sqrt{1-v^{2}}}\right)-\frac{v}{\sqrt{1-v^{2}}} H^{'}_{-1,0}\left(\frac{x+vt}{\sqrt{1-v^{2}}}\right)}_{L^{2}_{x}}=&0
\end{align*}
More precisely, in \cite{second} we proved the following result
\begin{theorem}\label{approximated theorem}
There exist a sequence of functions $\left(\phi_{k}(v,t,x)\right)_{k\geq 2},$ a sequence of real values $\delta(k)>0$ and a sequence of numbers $n_{k}\in\mathbb{N}$ such that for any $0<v<\delta(k),\,\phi_{k}(v,t,x)$ satisfies 
\begin{align*}
    \lim_{t\to +\infty}\norm{\phi_{k}(v,t,x)-H_{0,1}\left(\frac{x-vt}{\sqrt{1-v^{2}}}\right)-H_{-1,0}\left(\frac{x+vt}{\sqrt{1-v^{2}}}\right)}_{H^{1}_{x}}=0,\\
    \lim_{t\to +\infty}\norm{\partial_{t}\phi_{k}(v,t,x)+\frac{v}{\sqrt{1-v^{2}}}H^{'}_{0,1}\left(\frac{x-vt}{\sqrt{1-v^{2}}}\right)-\frac{v}{\sqrt{1-v^{2}}}H^{'}_{-1,0}\left(\frac{x+vt}{\sqrt{1-v^{2}}}\right)}_{L^{2}_{x})}=0,\\
    \lim_{t\to -\infty}\norm{\phi_{k}(v,t,x)-H_{0,1}\left(\frac{x+vt-e_{v,k}}{\sqrt{1-v^{2}}}\right)-H_{-1,0}\left(\frac{x-vt+e_{v,k}}{\sqrt{1-v^{2}}}\right)}_{H^{1}_{x}}=0,\\
     \lim_{t\to -\infty}\norm{\partial_{t}\phi_{k}(v,t,x)-\frac{v}{\sqrt{1-v^{2}}}H^{'}_{0,1}\left(\frac{x+vt-e_{v,k}}{\sqrt{1-v^{2}}}\right)+\frac{v}{\sqrt{1-v^{2}}}H^{'}_{-1,0}\left(\frac{x-vt+e_{v,k}}{\sqrt{1-v^{2}}}\right)}_{ L^{2}_{x}}=0,
    \end{align*}
    with $e_{v,k}\in\mathbb{R}$ satisfying
\begin{equation*}
    \lim_{v\to 0}\frac{\left\vert e_{v,k} -\frac{\ln{\left(\frac{8}{v^{2}}\right)}}{\sqrt{2}}\right\vert}{v\vert\ln{(v)}\vert^{3}}=0.
\end{equation*} 
Moreover, if $0<v<\delta(k),$ then for any $s\geq 0$ and $l\in\mathbb{N}\cup\{0\},$ there is $C(k,s,l)>0$ such that
     \begin{equation*}
     \norm{\frac{\partial^{l}}{\partial t^{l}}\Lambda(\phi_{k}(v,t,x))}_{H^{s}_{x}(\mathbb{R})}\leq C(k,s,l)v^{2k+l}\left(\vert t\vert v+\ln{\left(\frac{1}{v^{2}}\right)}\right)^{n_{k}}e^{-2\sqrt{2}\vert t\vert v}.
    \end{equation*}
\end{theorem}
 We consider the Schwartz function $\mathcal{G}$ defined by
\begin{equation}\label{G(x)}
     \mathcal{G}(x)=e^{-\sqrt{2}x}-\frac{e^{-\sqrt{2}x}}{(1+e^{2\sqrt{2}x})^{\frac{3}{2}}}+2\sqrt{2}x\frac{e^{\sqrt{2}x}}{(1+e^{2\sqrt{2}x})^{\frac{3}{2}}}+k_{1}\frac{e^{\sqrt{2}x}}{(1+e^{2\sqrt{2}x})^{\frac{3}{2}}},
\end{equation}
for all $x\in\mathbb{R},$ where $k_{1}$ is the real number such that $\mathcal{G}$ satisfies $\left\langle \mathcal{G}(x),\, H^{'}_{0,1}(x) \right\rangle_{L^{2}_{x}(\mathbb{R})}=0.$ The function $\mathcal{G}$ satisfies the following identity
 \begin{equation}\label{Group}
 -\frac{d^{2}}{dx^{2}} \mathcal{G}(x)+U^{''}(H_{0,1}(x))\mathcal{G}(x)=\left[-24 H_{0,1}(x)^{2}+30 H_{0,1}(x)^{4}\right]e^{-\sqrt{2}x}+8\sqrt{2} H^{'}_{0,1}(x),
 \end{equation}
see Lemma $A.1$ and Remark $A.2$ in the Appendix of \cite{second} for the proof.
\par From now on, for any $v\in(0,1),$ we consider the function $d_{v}:\mathbb{R}\to\mathbb{R}$ defined by
\begin{equation*}
d_{v}(t)=\frac{1}{\sqrt{2}}\ln{\left(\frac{8}{v^{2}}\cosh{\left(\sqrt{2}v t\right)}^{2}\right)} \text{, for any $t\in\mathbb{R}.$}
\end{equation*} The function $d_{v}$ describes the movement between two kinks for the $\phi^{6}$ model during a large time interval when their total energy is small and their initial speeds are both zero. For more information, see Theorem $1.11$ from \cite{first}.
\par Moreover, from the proof of Theorem \ref{approximated theorem} in \cite{second}, we can construct inductively an explicit sequence of smooth functions $(\varphi_{k,v})_{k\in\mathbb{N}_{\geq 2}}$ and for each $k\in\mathbb{N}_{\geq 2}$ there exists  a real number $\tau_{k,v}$ satisfying $\vert \tau_{k,v}\vert <\frac{\sqrt{2}}{v}\ln{\left(\frac{8}{v^{2}}\right)}$ such that $\phi_{k}(v,t,x)\coloneqq\varphi_{k,v}(t+\tau_{k,v},x)$ satisfies Theorem \ref{approximated theorem} for all $k\in\mathbb{N}_{\geq 2}.$ More precisely, from the paper \cite{second}, we have the following theorem:
\begin{theorem}\label{toobig}
There exist a sequence of approximate solutions $\varphi_{k,v}(t,x),$ functions $r_{k}(v,t)$ that are smooth and even on $t,$ and numbers $n_{k}\in\mathbb{N}$ such that if $0<v\ll 1,$ then for any $m\in\mathbb{N}_{\geq 1}$
\begin{equation}\label{aproxr}
    \frac{\vert r_{k}(v,t)\vert}{\mathcal{C}(k,0,0,0)}\leq v^{2(k-1)} \ln{\left(\frac{1}{v}\right)}^{n_{k}},\,\frac{\left\vert\frac{\partial^{m}}{\partial t^{m}}r_{k}(v,t)\right\vert}{\mathcal{C}(k,0,l,0)}\leq v^{2(k-1)+m}\left[\ln{\left(\frac{1}{v}\right)}+\vert t\vert v\right]^{n_{k}}e^{-2\sqrt{2}\vert t\vert v},
\end{equation}
$\varphi_{k,v}(t,x)$ satisfies for $\rho_{k}(v,t)=-\frac{d_{v}(t)}{2}+\sum_{j=2}^{k}r_{j}(v,t)$ the identity
\begin{align} \nonumber
    \varphi_{k,v}(t,x)=&H_{0,1}\left(\frac{x+\rho_{k}(v,t)}{\sqrt{1-\frac{\dot d_{v}(t)^{2}}{4}}}\right)+H_{-1,0}\left(\frac{x-\rho_{k}(v,t)}{\sqrt{1-\frac{\dot d_{v}(t)^{2}}{4}}}\right)\\ \nonumber &{+}e^{-\sqrt{2}d_{v}(t)}\left[\mathcal{G}\left(\frac{x+\rho_{k}(v,t)}{\sqrt{1-\frac{\dot d_{v}(t)^{2}}{4}}}\right)-\mathcal{G}\left(\frac{-x+\rho_{k}(v,t)}{\sqrt{1-\frac{\dot d_{v}(t)^{2}}{4}}}\right)\right]
    \\ \label{aproxfor} &{+}\mathcal{R}_{k,v}\left(vt,\frac{x+\rho_{k}(v,t)}{\sqrt{1-\frac{\dot d_{v}(t)^{2}}{4}}}\right)-\mathcal{R}_{k,v}\left(vt,\frac{-x+\rho_{k}(v,t)}{\sqrt{1-\frac{\dot d_{v}(t)^{2}}{4}}}\right)
\end{align}
the following estimates for any $l\in\mathbb{N}\cup\{0\}$ and $s\geq 1$
\begin{equation}\label{geraldecay}
    \norm{\frac{\partial^{l}}{\partial t^{l}}\Lambda(\varphi_{k,v}(t,x)}_{H^{s}_{x}(\mathbb{R})}\lesssim_{k,s,l}v^{2k+l}\left[\ln{\left(\frac{1}{v^{2}}\right)}+\vert t\vert v\right]^{n_{k}}e^{-2\sqrt{2}\vert t\vert v},
\end{equation}
and 
\begin{equation}\label{orthodecay}
 \left\vert\frac{d^{l}}{dt^{l}}\left[\left\langle \Lambda(\varphi_{k,v})(t,x),\, H^{'}_{0,1}\left(\frac{x+\rho_{k}(v,t)}{(1-\frac{\dot d_{v}(t)^{2}}{4})^{\frac{1}{2}}}\right)\right\rangle\right]\right\vert\lesssim_{k,l} v^{2k+l+2}\left[\ln{\left(\frac{1}{v^{2}}\right)}+\vert t\vert v\right]^{n_{k}+1}e^{-2\sqrt{2}\vert t\vert v},    
\end{equation}
where $\mathcal{R}_{k}(t,x)$ is a finite sum of functions $p_{k,i,v}(t)h_{k,i}(x)$ with $h_{k,i}\in\mathscr{S}(\mathbb{R})$ and each $p_{k,i,v}(t)$ being an even function satisfying, for all $m\in \mathbb{N},$
\begin{equation*}
\left \vert\frac{d^{m}p_{k,i,v}(t)}{dt^{m}} \right\vert\lesssim_{k,m,3}v^{4}\left(\ln{\left(\frac{1}{v^{2}}\right)}+\vert t\vert \right)^{n_{k,i}}e^{-2\sqrt{2}\vert t\vert },
\end{equation*}
where $n_{k,i}\in\mathbb{N}$ depends only on $k$ and $i.$ 
\end{theorem}
\begin{remark}\label{n2}
 Furthermore, Remark $5.2$ of the paper \cite{second} implies that if $v>0$is small enough, then the function $r_{2}$ satisfies
 \begin{equation*}
    \norm{r_{2}(v,\cdot)}_{L^{\infty}(\mathbb{R})}\lesssim v^{2}\ln{\left(\frac{1}{v^{2}}\right)},\,\, \left\vert \frac{\partial^{l}}{\partial t^{l}}r_{2}(v,t)\ \right\vert \lesssim_{l} v^{2+l}\left[\ln{\left(\frac{1}{v^{2}}\right)}+\vert t \vert v\right]e^{{-}2\sqrt{2}\vert t\vert v},
 \end{equation*}
for all $l\in\mathbb{N}.$
\end{remark}
\begin{remark}
At first look, the statement of Theorem \ref{toobig} seems to contain excessive information about the approximate solutions $\phi_{k}(v,t,x)$ of \cite{second}. However, we are going to need every information of Theorem \ref{toobig} to study the elasticity and stability of the collision of two kinks with low speed $0<v<1.$
\end{remark}
\subsection{Auxiliary estimates}\label{auxaux}
First, we recall the Lemma $2.1$ of \cite{first}.
\begin{lemma}\label{interactt}
There exists a real function $C:\mathbb{R}^{3}\times\mathbb{N}\cup\{0\}\to\mathbb{R}_{>0}$ such that for any real numbers $x_{2},x_{1}$ satisfying $z=x_{2}-x_{1}>0$ and $\alpha,\,\beta,\,m>0$ with $\alpha\neq \beta$ the following bound holds:
\begin{equation*}
    \int_{\mathbb{R}}\vert x-x_{1}\vert ^{m} e^{-\alpha(x-x_{1})_{+}}e^{-\beta(x_{2}-x)_{+}}\leq C(\alpha,\beta,m) \max\left(\left(1+z^{m}\right)e^{-\alpha z},e^{-\beta z}\right),
\end{equation*}
Furthermore, for any $\alpha>0$, the following bound holds
\begin{equation*}
    \int_{\mathbb{R}}\vert x-x_{1}\vert^{m} e^{-\alpha(x-x_{1})_{+}}e^{-\alpha(x_{2}-x)_{+}}\leq C(\alpha,\alpha,m)\left[1+z^{m+1}\right] e^{-\alpha z}.
\end{equation*}
\end{lemma}
Actually, we will also need to use the following lemma, which we proved in \cite{second}.
\begin{lemma}\label{dots+}
In notation of Theorem \ref{toobig}, for $0<v\ll 1,$ let $w_{k,v}:\mathbb{R}^{2}\to\mathbb{R}$ be the following function
\begin{equation*}
    w_{k,v}(t,x) =\frac{x+\rho_{k}(v,t)}{\sqrt{1-\frac{\dot d_{v}(t)^{2}}{4}}},
\end{equation*}
and let $f\in L^{\infty}_{x}(\mathbb{R})$ be a function satisfying $\dot f\in \mathscr{S}(\mathbb{R}).$
Then, if 0<$v\ll 1,$ we have for any $l\in\mathbb{N}$ that
\begin{equation*}
    \frac{\partial^{l}}{\partial t^{l}}f\left(w_{k,v}(t,x)\right)
\end{equation*}
is a finite sum of functions $q_{k,l,i,v}(t)h_{i}\left(w_{k,v}(t,x)\right)$ with each $h_{i}\in \mathscr{S}(\mathbb{R})$ and any $q_{k,l,i,v}(t)$ is a smooth real function satisfying
\begin{equation*}
\norm{q_{k,l,i,v}}_{L^{\infty}(\mathbb{R})}\lesssim v^{l}.
\end{equation*}
Furthermore, if $0<v\ll 1,$ we have for all $l\in\mathbb{N}$ and any $s\geq 0$ that 
\begin{equation*}
    \norm{\frac{\partial^{l}}{\partial t^{l}}f\left(w_{k,v}(t,x)\right)}_{H^{s}_{x}(\mathbb{R})}\lesssim_{k,s,l} v^{l}.
\end{equation*}
\end{lemma}
Moreover, we are going to use the following result several times in the computation of the estimates of this manuscript.
\begin{lemma}\label{mulsoblemma}
For any $s\geq 1,$ we have for any functions $f,\, g\in\mathscr{S}(\mathbb{R})$ that
\begin{equation*}
   \norm{f g}_{H^{s}_{x}(\mathbb{R})}\lesssim_{s} \norm{f}_{H^{s}_{x}(\mathbb{R})}\norm{g}_{L^{\infty}_{x}(\mathbb{R})}+ \norm{ g}_{H^{s}_{x}(\mathbb{R})} \norm{f}_{L^{\infty}_{x}(\mathbb{R})}\norm{g}_{L^{\infty}_{x}(\mathbb{R})} \lesssim_{s} \norm{f}_{H^{s}_{x}(\mathbb{R})}\norm{g}_{H^{s}_{x}(\mathbb{R})}.
\end{equation*}
As a consequence,
\begin{equation*}
   \norm{f g}_{H^{s}_{x}(\mathbb{R})}\lesssim_{s_{0}} \norm{f}_{H^{s+1}_{x}(\mathbb{R})}\norm{g}_{H^{s+1}_{x}(\mathbb{R})}, 
\end{equation*}
for all $s\geq 0.$
\end{lemma}
\begin{proof}
    See the proof of Lemma $A.8$ in the book \cite{dsipersivebook}.
\end{proof}
\par Finally, we need also Lemma $2.5$ of \cite{first} which studies the coercive properties of the operator
\begin{equation*}
    -\partial^{2}_{x}+U^{''}\left(H^{z}_{0,1}(x)+H_{-1,0}(x)\right),
\end{equation*}
when $z\gg 1.$ More precisely:
\begin{lemma}\label{coerccc}
 There exist $c,\,\delta>0$ such that if $z\geq\frac{1}{\delta}$, then for any $g \in H^{1}(\mathbb{R})$ satisfying
 \begin{equation*}
    \left\langle g(x),\, H^{'}_{0,1}(x-z) \right\rangle=\left\langle g(x),\,  H^{'}_{-1,0}(x)\right\rangle=0,
\end{equation*}
we have that
\begin{equation*}
    \left\langle -\frac{d^{2}}{dx^{2}} g(x)+ U^{''}\left(H_{0,1}(x-z)+H_{-1,0}(x)\right)g(x),\,g(x)\right\rangle\geq c\norm{g}_{H^{1}_{x}(\mathbb{R})}^{2}.
\end{equation*}
\end{lemma}
\begin{proof}
See the proof of Lemma $9$ in \cite{first}.
\end{proof}
In this manuscript, to simplify our notation, we denote $d_{v}(t)$ by $d(t),$ which means that
\begin{equation}\label{d}
    d(t)=\frac{1}{\sqrt{2}}\ln{\left(\frac{8}{v^{2}}\cosh{(\sqrt{2}vt)}^{2}\right)}.
\end{equation}
In Lemma $3.1$ of \cite{second}, we have verified by induction the following estimates 
\begin{equation}\label{decayd}
   \vert \dot d(t)\vert\lesssim v \text{, and for any $l\in\mathbb{N}_{\geq 2}$} \left\vert d^{(l)}(t) \right\vert\lesssim_{l} v^{l}e^{{-}2\sqrt{2}\vert t\vert v}.
\end{equation}
From now on, we consider for each $k\in\mathbb{N}_{\geq 2}$ the function $\phi_{k,v}(t,x)$  satisfying Theorem \ref{toobig}. Next, for $T_{0,k}>0$ to be chosen later, we consider the following kind of Cauchy problem
\begin{equation}\label{IVP1}
    \begin{cases}
       \partial^{2}_{t}\phi(t,x)-\partial^{2}_{x}\phi(t,x)+ U^{'}\left(\phi(t,x)\right)=0,\\
       \norm{(\phi(T_{0,k},x),\partial_{t}\phi(T_{0,k},x))-(\phi_{k,v}(T_{0,k},x),\partial_{t}\phi_{k,v}(T_{0,k},x)))}_{H^{1}_{x}(\mathbb{R})\times L^{2}_{x}(\mathbb{R})}<v^{8k}.
    \end{cases}
\end{equation}
Our first objective is to prove the following theorem.
\begin{theorem}\label{energyE}
There is a constant $C>0$ and for any for any $0<\theta<\frac{1}{4},\,k\in\mathbb{N}_{\geq 3}$ there exist $C_{1}(k)>0,\,\delta_{k,\theta}>0$ and $\eta_{k}\in\mathbb{N}$ such that if $0<v<\delta_{k,\theta}$ and $T_{0,k}=\frac{32k}{2\sqrt{2}}\frac{\ln{\left(\frac{1}{v^{2}}\right)}}{v},$ then any solution $\phi(t,x)$ of \eqref{IVP1} satisfies:
\begin{equation}\label{globalstronger}
    \norm{(\phi(t,x),\partial_{t}\phi(t,x))-(\varphi_{k,v}(t,x),\partial_{t}\varphi_{k,v}(t,x))}_{H^{1}_{x}\times L^{2}_{x}}<C_{1}(k)v^{2k}\ln{\left(\frac{1}{v}\right)}^{\eta_{k}}\exp\left(C\frac{v\vert t-T_{0,k}\vert}{\ln{\left(v\right)}}\right),
\end{equation}
if 
\begin{equation*}
    \vert t-T_{0,k}\vert<\frac{\ln{\left(\frac{1}{v}\right)}^{2-\theta}}{v}.    
\end{equation*}
 \end{theorem}
Clearly, we can obtain from Theorem  \ref{energyE} and Theorem \ref{toobig} the following result:
\begin{corollary}\label{coenergyE}
There is a constant $C>0$ and for any $0<\theta<\frac{1}{4},\,k\in\mathbb{N}_{\geq 3}$ there exist $C_{1}(k)>0,\,\delta_{k,\theta}>0$ and $\eta_{k}\in\mathbb{N}$ such that if $0<v<\delta_{k,\theta}$ and $T_{0,k}=\frac{32k}{2\sqrt{2}}\frac{\ln{\left(\frac{1}{v^{2}}\right)}}{v},$ then any solution $\phi(t,x)$ 
of
\begin{equation*}\
    \begin{cases}
       \partial^{2}_{t}\phi(t,x)-\partial^{2}_{x}\phi(t,x)+ U^{'}\left(\phi(t,x)\right)=0,\\
       \norm{(\phi(T_{0,k},x),\partial_{t}\phi(T_{0,k},x))-(\phi_{k}(v,T_{0,k},x),\partial_{t}\phi_{k}(v,T_{0,k},x)))}_{H^{1}_{x}(\mathbb{R})\times L^{2}_{x}(\mathbb{R})}<v^{8k}
    \end{cases}
\end{equation*}
satisfies
\begin{equation*}
    \norm{(\phi(t,x),\partial_{t}\phi(t,x))-(\phi_{k}(v,t,x),\partial_{t}\phi_{k}(v,t,x))}_{H^{1}_{x}\times L^{2}_{x}}<C_{1}(k)v^{2k}\ln{\left(\frac{1}{v}\right)}^{\eta_{k}}\exp\left(C\frac{v\vert t-T_{0,k}\vert}{\ln{\left(v\right)}}\right),
\end{equation*}
if 
\begin{equation*}
    \vert t-T_{0,k}\vert<\frac{\ln{\left(\frac{1}{v}\right)}^{2-\theta}}{v}.
\end{equation*}
\end{corollary}
\begin{proof}[Proof of Corollary \ref{coenergyE}.]
It follows from Theorem \ref{energyE} and Theorems \ref{approximated theorem}, \ref{toobig}. 
\end{proof}
\par With the objective of simplifying the demonstration of Theorem \ref{energyE}, we are going to elaborate on necessary lemmas before the proof of Theorem \ref{energyE}. 
Similarly to the paper \cite{second}, using the notation of Theorem \ref{toobig}, we consider 

\begin{equation}\label{wk}
    w_{k,v}(t,x)=\frac{x-\frac{d(t)}{2}+c_{k}(v,t)}{\sqrt{1-\frac{\dot d(t)^{2}}{4}}}.
\end{equation}
From now on, we denote any solution $\phi(t,x)$ of the partial differential equation \eqref{IVP1}  as
\begin{equation}\label{notnot}
    \phi(t,x)=\varphi_{k,v}(t,x)+\frac{y_{1}(t)}{\sqrt{1-\frac{\dot d(t)^{2}}{4}}} H^{'}_{0,1}\left(w_{k,v}(t,x)\right)\do+\frac{y_{2}(t)}{\sqrt{1-\frac{\dot d(t)^{2}}{4}}} H^{'}_{0,1}\left(w_{k,v}(t,{-}x)\right)+u(t,x),
\end{equation}
such that
\begin{equation}\label{orthogonall}
    \left\langle u(t,x), H^{'}_{0,1}\left(w_{k,v}(t,x)\right)\right\rangle=\left\langle u(t,x), H^{'}_{0,1}\left(w_{k,v}(t,{-}x)\right)\right\rangle=0.
\end{equation}
\par Therefore, for $\zeta_{k}(t)=d(t)-2c_{k}(v,t)$ and from the orthogonal conditions \eqref{orthogonall} satisfied by $u(t,x),$ we deduce the following identity
\begin{equation}\label{liny1y2}
    \begin{bmatrix}
     y_{1}(t)\\
     y_{2}(t)
    \end{bmatrix}
    =
    M(t)^{{-1}}\begin{bmatrix}
     \left\langle \phi(t,x)-\varphi_{k,v}(t,x), H^{'}_{0,1}\left(w_{k,v}(t,x)\right)\right\rangle\\
     \left\langle \phi(t,x)-\varphi_{k,v}(t,x), H^{'}_{{-}1,0}\left(w_{k,v}(t,{-}x)\right)\right\rangle
    \end{bmatrix}.
\end{equation}
where, for any $t\in\mathbb{R},$ $M(t)$ is denoted by
\begin{equation*}
    M(t)=\begin{bmatrix}
      \norm{ H^{'}_{0,1}}_{L^{2}_{x}}^{2} & \left\langle  H^{'}_{0,1}(x-\zeta_{k}(t)), H^{'}_{{-}1,0}(x) \right\rangle\\
      \left\langle  H^{'}_{0,1}(x-\zeta_{k}(t)), H^{'}_{{-}1,0}(x) \right\rangle & \norm{ H^{'}_{0,1}}_{L^{2}_{x}}^{2}
    \end{bmatrix}.
\end{equation*}
Moreover, since  $\ln{\left(\frac{1}{v}\right)}\lesssim\zeta_{k},$ we obtain from Lemma \ref{interactt} that
$ \left\langle  H^{'}_{0,1}(x-\zeta_{k}(t)), H^{'}_{{-}1,0}(x) \right\rangle\ll 1.$ Therefore, since the matrix $M(t)$ is a smooth function with domain $\mathbb{R}$, then $M(t)^{{-}1}$ is also smooth on $\mathbb{R}.$ 
\par Next, for $\psi(t,x)=\phi(t,x)-\varphi_{k,v}(t,x),$ we obtain from the partial differential equation \eqref{IVP1} that $\psi(t,x)$ satisfies the following partial differential equation
\begin{align}\label{pp1}
  \frac{\partial^{2}}{\partial t^{2}}\psi(t,x)-\frac{\partial^{2}}{\partial x^{2}}\psi(t,x)+ \Lambda(\varphi_{k,v})(t,x)+\sum_{j=2}^{6}\frac{U^{(j)}\left(\varphi_{k,v}(t,x)\right)}{(j-1)!}\psi(t,x)^{j-1}=0.
\end{align}
Since $\varphi_{k,v}$ satisfies Theorem \ref{toobig} and the partial differential equation \eqref{nlww} is globally well-posed in the energy space, we can verify for any initial data $(\psi_{0}(x),\psi_{1}(x))\in H^{1}_{x}(\mathbb{R})\times L^{2}_{x}(\mathbb{R})$ that there exists a unique solution $\psi(t,x)$ of \eqref{pp1} satisfying $(\psi(0,x),\partial_{t}\psi(0,x))=(\psi_{0}(x),\psi_{1}(x))$ and 
\begin{equation}\label{gltt}
(\psi(t,x),\partial_{t}\psi(t,x))\in C\left(\mathbb{R};H^{1}_{x}(\mathbb{R})\times L^{2}_{x}(\mathbb{R})\right).
\end{equation}
\par Therefore, for any function $h\in\mathscr{S}(\mathbb{R}),$ we deduce from \eqref{pp1} that
\begin{align*}
    \frac{d}{dt}\left\langle \psi(t,x),h(x)\right\rangle=&\left\langle \partial_{t}\psi(t,x),h(x)\right\rangle,\\
    \frac{d^{2}}{dt^{2}}\left\langle \psi(t,x),h(x)\right\rangle=&\left\langle \frac{\partial^{2}}{\partial x^{2}}\psi(t,x)- U^{'}\left(\varphi_{k,v}(t,x)+\psi(t,x)\right)+ U^{'}\left(\varphi_{k,v}(t,x)\right),h(x)\right\rangle,
\end{align*}
which implies that the real function $\mathcal{P}_{1}(t)=\left\langle \psi(t,x), H^{'}_{0,1}\left(w_{k,v}(t,x)\right)\right\rangle$ and  the real function $\mathcal{P}_{2}(t)=\left\langle \psi(t,x), H^{'}_{{-}1,0}\left(w_{k,v}(t,{-}x)\right)\right\rangle$ are in $C^{2}(\mathbb{R}).$ In conclusion, using equation \eqref{liny1y2} and the product rule of derivative, we deduce that $y_{1},\, y_{2}\in C^{2}(\mathbb{R}).$
\par In conclusion, we obtain the following lemma:
\begin{lemma}\label{pdeu}
Assuming the same hypotheses of Theorem \ref{energyE}, there exist functions $y_{1},\,y_{2}:\mathbb{R}\to\mathbb{R}$ of class $C^{2}$ such that any solution $\phi(t,x)$ of \eqref{IVP1} satisfies for any $t\in\mathbb{R}$ the following identity
\begin{equation*}
 \phi(t,x)=\varphi_{k,v}(t,x)+\frac{y_{1}(t)}{\sqrt{1-\frac{\dot d(t)^{2}}{4}}} H^{'}_{0,1}\left(w_{k,v}(t,x)\right)+\frac{y_{2}(t)}{\sqrt{1-\frac{\dot d(t)^{2}}{4}}} H^{'}_{0,1}\left(w_{k,v}(t,{-}x)\right)+u(t,x), 
\end{equation*}
where $(u(t),\partial_{t}u(t))\in H^{1}_{x}(\mathbb{R})\times L^{2}_{x}(\mathbb{R})$ and the function $u$ satisfies the following orthogonality conditions:
\begin{align*}
    \left\langle u(t,x),H^{'}_{0,1}\left(w_{k,v}(t,x)\right)\right\rangle=& 0,\\
    \left\langle u(t,x),H^{'}_{0,1}\left(w_{k,v}(t,{-}x)\right)\right\rangle=& 0.
\end{align*}
\end{lemma}
\begin{remark}\label{easyremark}
Using Lemmas \ref{dots+}, \ref{mulsoblemma}, \ref{pdeu} $,\Lambda(\phi)=0,$ Theorem \ref{toobig}, Remark \ref{n2} and  identities $ \frac{d^{3}}{dx^{3}}H_{0,1}(x)=U^{''}\left(H_{0,1}(x)\right) H^{'}_{0,1}(x),\, \ddot d(t)=16\sqrt{2}e^{{-}\sqrt{2}d(t)},$ we can deduce that $u$ satisfies the following partial differential equation
\begin{multline}\label{uequation}
 \Lambda\left(\varphi_{k,v}\right)(t,x)+\partial^{2}_{t}u(t,x)-\partial^{2}_{x}u(t,x)+ U^{''}\left(\varphi_{k,v}(t,x)\right)\left(\phi(t,x)-\varphi_{k,v}(t,x)\right)\\{+}\frac{\ddot y_{1}(t)}{\sqrt{1-\frac{\dot d(t)^{2}}{4}}} H^{'}_{0,1}\left(w_{k,v}(t,x)\right)
  +\frac{\ddot y_{2}(t)}{\sqrt{1-\frac{\dot d(t)^{2}}{4}}} H^{'}_{0,1}\left(w_{k,v}(t,{-}x)\right)
  -\frac{y_{1}(t)8\sqrt{2}e^{-\sqrt{2}d(t)}}{1-\frac{\dot d(t)^{2}}{4}} H^{''}_{0,1}\left(w_{k,v}(t,x)\right)\\-\frac{y_{2}(t)8\sqrt{2}e^{-\sqrt{2}d(t)}}{1-\frac{\dot d(t)^{2}}{4}} H^{''}_{0,1}\left(w_{k,v}(t,{-}x)\right)
  -\frac{\dot y_{1}(t)\dot d(t)}{1-\frac{\dot d(t)^{2}}{4}} H^{''}_{0,1}\left(w_{k,v}(x,t)\right)-\frac{\dot y_{2}(t)\dot d(t)}{1-\frac{\dot d(t)^{2}}{4}} H^{''}_{0,1}\left(w_{k,v}(t,{-}x)\right)
  \\-y_{1}(t)\frac{ U^{''}\left(H_{0,1}\left(w_{k,v}(t,x)\right)\right)}{\sqrt{1-\frac{\dot d(t)^{2}}{4}}} H^{'}_{0,1}\left(w_{k,v}(t,x)\right)-y_{2}(t)\frac{ U^{''}\left(H_{0,1}\left(w_{k,v}(t,{-}x)\right)\right)}{\sqrt{1-\frac{\dot d(t)^{2}}{4}}} H^{'}_{0,1}\left(w_{k,v}(t,{-}x)\right)\\
 =\mathcal{Q}(t,x),
\end{multline}
where $\mathcal{Q}(t,\cdot )$ is a function in $H^{1}_{x}(\mathbb{R})$ satisfying for all $t\in\mathbb{R}$
\begin{multline*}
    \norm{\mathcal{Q}(t,x)}_{H^{1}_{x}(\mathbb{R})}\lesssim 
    \norm{u(t)}_{H^{1}_{x}}^{2}+\norm{u(t)}_{H^{1}_{x}}^{6}+\max_{j\in\{1,\,2\}}\vert y_{j}(t)\vert^{2}+\max_{j\in\{1,\,2\}}\vert y_{j}(t)\vert^{6}\\{+}\left[\max_{j\in\{1,2\}}\vert\dot y_{j}(t)\vert+v\max_{j\in\{1,2\}}\vert y_{j}(t)\vert\right] v^{3}\left(\ln{\left(\frac{1}{v^{2}}\right)}+\vert t\vert v\right)e^{-2\sqrt{2}\vert t\vert v},
\end{multline*}
if $v>0$ is small enough.
\end{remark}
\par Next, from equation \eqref{uequation} of Remark \ref{easyremark},  we consider the terms
\begin{align}\label{termy1}
     Y_{1}(t,x)=&\left[ U^{''}\left(\varphi_{k,v}(t,x)\right)- U^{''}\left(H_{0,1}\left(w_{k,v}(t,x)\right)\right)\right]\frac{y_{1}(t)}{\sqrt{1-\frac{\dot d(t)^{2}}{4}}} H^{'}_{0,1}\left(w_{k,v}(t,x)\right),\\ \label{Y2(t)}
   Y_{2}(t,x)=&\left[U^{''}\left(\varphi_{k,v}(t,x)\right)- U^{''}\left(H_{0,1}\left(w_{k,v}(t,{-}x)\right)\right)\right]\frac{y_{2}(t)}{\sqrt{1-\frac{\dot d(t)^{2}}{4}}} H^{'}_{0,1}\left(w_{k,v}(t,{-}x)\right).
\end{align}
Now, we will estimate the expressions
\begin{equation*}
    \left\langle Y_{1}(t), H^{'}_{0,1}\left(w_{k,v}(t,x)\right) \right\rangle,\,\left\langle Y_{2}(t), H^{'}_{0,1}\left(w_{k,v}(t,{-}x)\right) \right\rangle.
\end{equation*}
\begin{lemma}\label{y1y2}
In notation of Theorem \ref{toobig} and Lemma \ref{pdeu}, the functions $Y_{1}(t)$ and $Y_{2}(t)$ satisfy
\begin{align*}
    \left\langle Y_{1}(t),\, H^{'}_{0,1}\left(w_{k,v}(t,x)\right) \right\rangle=&4\sqrt{2}e^{-\sqrt{2}d(t)}y_{1}(t)+y_{1}(t)Res_{1}(v,t),\\
    \left\langle Y_{2}(t),\, H^{'}_{0,1}\left(w_{k,v}(t,x)\right) \right\rangle=&{-}4\sqrt{2}e^{-\sqrt{2}d(t)}y_{2}(t)+y_{2}(t)Res_{2}(v,t),
\end{align*}
where, for any $j\in\{1,2\}$ and all $v\in(0,1),$ the function $Res_{j}(v,t)$ is a Schwartz function on $t$ satisfying for any $l\in\mathbb{N}\cup\{0\},$ if $0<v\ll 1,$ the following estimate
\begin{equation}\label{esttt}
    \left\vert\frac{\partial^{l}}{\partial t^{l}}Res_{j}(v,t)\right\vert\lesssim_{l} v^{l+4}\left[\ln{\left(\frac{1}{v^{2}}\right)}+\vert t\vert v\right]^{\eta_{k}}e^{-2\sqrt{2}\vert t\vert v},
\end{equation}
for a number $\eta_{k}\geq 0$ depending only on $k\in\mathbb{N}_{\geq 2}.$
\end{lemma}
\begin{proof}[Proof of Lemma \ref{y1y2}.]
\par First, we observe that

\begin{equation*}
    \left \vert \frac{d^{l}}{d t^{l}}e^{{-}\sqrt{2}d(t)}\right\vert=\left\vert\frac{d^{l}}{d t^{l}}\frac{v^{2}}{8}\sech{\left(\sqrt{2}v t\right)}^{2}\right\vert\lesssim_{l} v^{2+l}e^{{-}2\sqrt{2}\vert t\vert v}.
\end{equation*}
\par Using Taylor's Expansion Theorem, Theorem \ref{toobig} and Lemma \ref{mulsoblemma}, we deduce that
\begin{multline*}
U^{''}\left(\varphi_{k,v}(t,x)\right)\\
    \begin{aligned}
    =& U^{''}\left(H_{0,1}(w_{k,v}(t,x))-H_{0,1}\left(w_{k,v}(t,{-}x)\right)\right)\\
&{+}e^{-\sqrt{2}d(t)}U^{(3)}\left(H_{0,1}(w_{k,v}(t,x))-H_{0,1}\left(w_{k,v}(t,{-}x)\right)\right)\left[\mathcal{G}(w_{k,v}(t,x))-\mathcal{G}(w_{k,v}(t,{-}x))\right]\\
 &{+}res_{1}(v,t,x),
    \end{aligned}
\end{multline*}
where, if $0<v\ll 1,$ $res_{1}(v,t,x)$ is a smooth function on the variables $(t,x)$  which satisfies for some $\eta_{k}\in\mathbb{N}$ and any $s\geq0,\,l\in\mathbb{N}\cup\{0\}$ the following inequality
\begin{equation}\label{dre1}
\norm{\frac{\partial^{l}}{\partial t^{l}}res_{1}(v,t,x)}_{H^{s}_{x}}\lesssim_{s,l}v^{4+l}\left[\ln{\left(\frac{1}{v^{2}}\right)}+\vert t\vert v\right]^{\eta_{k}}e^{{-}2\sqrt{2}\vert t\vert v}.
\end{equation}
Therefore, using identity
\begin{multline*}
U^{''}\left(\varphi_{k,v}(t,x)\right)- U^{''}\left(H_{0,1}(w_{k,v}(t,x)\right)\\= U^{''}\left(\varphi_{k,v}(t,x)\right)- U^{''}\left(H_{0,1}(w_{k,v}(t,x)-H_{0,1}(w_{k,v}(t,{-}x))\right)\\
     {+} U^{''}\left(H_{0,1}(w_{k,v}(t,x)-H_{0,1}(w_{k,v}(t,{-}x))\right)- U^{''}\left(H_{0,1}(w_{k,v}(t,x)\right),
\end{multline*}
 we obtain that
\begin{multline}\label{Y1}
Y_{1}(t,x)\sqrt{1-\frac{\dot d(t)^{2}}{4}}
 \\
 \begin{aligned}
=&\left[ U^{''}\left(H_{0,1}(w_{k,v}(t,x))-H_{0,1}\left(w_{k,v}(t,{-}x)\right)\right)-U^{''}\left(H_{0,1}(w_{k,v}(t,x))\right)\right]y_{1}(t) H^{'}_{0,1}\left(w_{k,v}(x,t)\right)\\
&{+}y_{1}(t)e^{-\sqrt{2}d(t)}U^{(3)}\left(H_{0,1}(w_{k,v}(t,x))-H_{0,1}\left(w_{k,v}(t,{-}x)\right)\right)\mathcal{G}(w_{k,v}(t,x)) H^{'}_{0,1}\left(w_{k,v}(t,x)\right)\\
    &{-}y_{1}(t)e^{-\sqrt{2}d(t)}U^{(3)}\left(H_{0,1}(w_{k,v}(t,x))-H_{0,1}\left(w_{k,v}(t,{-}x)\right)\right)\mathcal{G}(w_{k,v}(t,-x)) H^{'}_{0,1}\left(w_{k,v}(t,x)\right)\\
    &{+}y_{1}(t)res_{1}(v,t,x).
\end{aligned}
\end{multline}
By a similar reasoning, we obtain that
\begin{multline}\label{y2}
      Y_{2}(t,x)\sqrt{1-\frac{\dot d(t)^{2}}{4}}\\
      \begin{aligned}
      =&\left[U^{''}\left(H_{0,1}(w_{k,v}(t,x))-H_{0,1}\left(w_{k,v}(t,{-}x)\right)\right)-U^{''}\left(H_{0,1}(w_{k,v}(t,{-}x))\right)\right]y_{2}(t)H^{'}_{0,1}\left(w_{k,v}(t,{-}x)\right)\\
    &{+}y_{2}(t)e^{{-}\sqrt{2}d(t)}U^{(3)}\left(H_{0,1}(w_{k,v}(t,x))-H_{0,1}\left(w_{k,v}(t,{-}x)\right)\right)\mathcal{G}(w_{k,v}(t,x))H^{'}_{0,1}\left(w_{k,v}(t,{-}x)\right)\\
    &{-}y_{2}(t)e^{-\sqrt{2}d(t)}U^{(3)}\left(H_{0,1}(w_{k,v}(t,x))-H_{0,1}\left(w_{k,v}(t,{-}x)\right)\right)\mathcal{G}(w_{k,v}(t,-x))H^{'}_{0,1}\left(w_{k,v}(t,{-}x)\right)\\
    &{+}y_{2}(t)res_{2}(v,t,x), 
    \end{aligned}
\end{multline}
where if $0<v\ll 1,$ $res_{2}(v,t,x)$ is a smooth function on $t,\,x$ satisfying, for some constant $\eta_{k}\geq 0,$ any $l\in\mathbb{N}\cup\{0\}$ and $s\geq 0,$ the following estimate
\begin{equation}\label{res2}
\norm{\frac{\partial^{l}}{\partial t^{l}}res_{2}(v,t,x)}_{H^{s}_{x}}\lesssim_{s,l}v^{4+l}\left[\ln{\left(\frac{1}{v^{2}}\right)}+\vert t\vert v\right]^{\eta_{k}}e^{{-}2\sqrt{2}\vert t\vert v}.
\end{equation}
\par Next, from the Fundamental Theorem of Calculus, we have for any $\zeta>1$ that
\begin{multline*}
    \left[U^{''}\left(H^{\zeta}_{0,1}(x)+H_{-1,0}(x)\right)- U^{''}\left(H^{\zeta}_{0,1}(x)\right)\right]\partial_{x}H^{\zeta}_{0,1}(x)
    \\=U^{(3)}\left(H^{\zeta}_{0,1}(x)\right)H_{-1,0}(x)\partial_{x} H^{\zeta}_{0,1}(x)+\int_{0}^{1}U^{(4)}\left(H^{\zeta}_{0,1}+\theta H_{-1,0}\right)(1-\theta)H_{-1,0}(x)^{2}\partial_{x} H^{\zeta}_{0,1}(x)\,d\theta,
\end{multline*}
from which with Lemma \ref{interactt}, estimates \eqref{le2}, \eqref{le1} and $\left\vert \frac{d^{l}}{dx^{l}}\left[H_{{-}1,0}(x)+e^{{-}\sqrt{2}x}\right] \right\vert\lesssim_{l}\min\left(e^{{-}\sqrt{2}x},e^{{-}3\sqrt{2}x}\right),$ we obtain that
\begin{multline}\label{partp1}
   \left\langle \left[U^{''}\left(H^{\zeta}_{0,1}(x)+H_{-1,0}(x)\right)-U^{''}\left(H^{\zeta}_{0,1}(x)\right)\right]\partial_{x}H^{\zeta}_{0,1}(x),\,\partial_{x} H^{\zeta}_{0,1}(x)\right\rangle\\=
   {-}e^{-\sqrt{2}\zeta}\int_{\mathbb{R}} U^{(3)}\left(H_{0,1}(x)\right)H^{'}_{0,1}(x)^{2}e^{-\sqrt{2}x}\,dx+res_{3}(\zeta),
\end{multline}
with $res_{3}\in C^{\infty}(\mathbb{R}_{\geq 1})$ satisfying for all $l\in\mathbb{N}\cup\{0\}$ and $\zeta\geq 1$
\begin{equation*}
    \left\vert res^{(l)}_{3}(\zeta)\right\vert \lesssim_{l} \zeta e^{-2\sqrt{2}\zeta}.
\end{equation*}
\par Next, using $U\in C^{\infty}(\mathbb{R})$ and estimates \eqref{le2}, \eqref{le1}, we deduce for al $\zeta\geq 1$ and any $l\in\mathbb{N}\cup\{0\}$ that
\begin{equation*}
  \left\vert
  \frac{\partial^{l}}{\partial \zeta^{l}}\left[U^{(3)}\left(H^{\zeta}_{0,1}(x)+H_{{-}1,0}(x)\right)-U^{(3)}\left(H^{\zeta}_{0,1}(x)\right)
  \right]\right\vert\lesssim_{l} \left\vert   H_{{-}1,0}(x)\right\vert.   
\end{equation*}
Therefore, since $\mathcal{G}$ defined  in \eqref{G(x)} is a Schwartz function, Lemma \ref{interactt} implies that
\begin{equation*}
    int(\zeta)=\left\langle \left[U^{(3)}\left(H^{\zeta}_{0,1}(x)+H_{{-}1,0}(x)\right)-U^{(3)}\left(H^{\zeta}_{0,1}(x)\right)\right]\mathcal{G}(x-\zeta) \partial_{x} H^{\zeta}_{0,1}(x),\,\partial_{x}H^{\zeta}_{0,1}(x)\right\rangle
\end{equation*}
satisfies for all $\zeta\geq 1$ and any $l\in\mathbb{N}\cup\{0\}$ the following inequality
$
    \left\vert int^{(l)}(\zeta) \right\vert\lesssim_{l} e^{-\sqrt{2}\zeta}.
$ Moreover, using the following identity
\begin{equation}\label{u3ide}
    U^{(3)}(\phi)={-}48\phi+120\phi^{3},
\end{equation} we can deduce similarly that 
\begin{equation*}
    int_{2}(\zeta)= \left\langle U^{(3)}\left(H^{\zeta}_{0,1}(x)+H_{-1,0}(x)\right)\mathcal{G}({-}x)H^{'}_{{-}1,0}(x),\,\partial_{x} H^{\zeta}_{0,1}(x)\right\rangle
\end{equation*}
 satisfies
$
    \left\vert int^{(l)}_{2}(\zeta) \right\vert\lesssim_{l} e^{-\sqrt{2}\zeta}
$
for any $l\in\mathbb{N}\cup\{0\}$ and $\zeta\geq 1.$
As a consequence, we deduce that there exists a real function $int_{3}:\mathbb{R}_{\geq 1}\to \mathbb{R}$ satisfying for any $l\in\mathbb{N}\cup\{0\}$
\begin{equation*}
    \left\vert int^{(l)}_{3}(\zeta) \right\vert\lesssim_{l} e^{-\sqrt{2}\zeta},
\end{equation*}
where the function $int_{3}$ satisfies the following identity
\begin{multline}\label{partp2}
    \left\langle U^{(3)}\left(H^{\zeta}_{0,1}(x)+H_{-1,0}(x)\right)\mathcal{G}(x-\zeta) \partial_{x} H^{\zeta}_{0,1}(x),\partial_{x} H^{\zeta}_{0,1}(x)\right\rangle\\
    -\left\langle U^{(3)}\left(H^{\zeta}_{0,1}(x)+H_{{-}1,0}(x)\right)\mathcal{G}({-}x)  H^{'}_{0,1}({-}x),\partial_{x}H^{\zeta}_{0,1}(x)\right\rangle\\ =\int_{\mathbb{R}} U^{(3)}\left(H_{0,1}(x)\right) H^{'}_{0,1}(x)^{2}\mathcal{G}(x)\,dx+int_{3}(\zeta).
\end{multline}
From Theorem \ref{toobig}, estimates \eqref{decayd} and identity $e^{{-}\sqrt{2}d(t)}=\frac{v^{2}}{8}\sech{\left(\sqrt{2}\vert t\vert v\right)}^{2},$ it is not difficult to verify for any $l\in\mathbb{N}\cup\{0\}$ that if $0<v\ll 1,$ then
\begin{equation}\label{intdecayerror}
    \frac{d^{l}}{dt^{l}}\exp\left(\frac{2\rho_{k,v}(t)}{\sqrt{1-\frac{\dot d(t)^{2}}{4}}}\right)\lesssim_{l} v^{2+l}e^{{-}2\sqrt{2}\vert t\vert v}.
\end{equation}
In conclusion, from estimates
\eqref{Y1}, \eqref{partp1}, \eqref{partp2} and Lemma \ref{est2} of Appendix Section \ref{aux}, we obtain using identity
\begin{equation*}
    w_{k,v}(t,x)=\frac{x-\frac{d(t)}{2}+c_{k,v}}{\sqrt{1-\frac{\dot d(t)^{2}}{4}}},
\end{equation*}
and Theorem \ref{toobig} that $Y_{1}(t)$ satisfies Lemma \ref{y1y2}. 
\par The proof that $Y_{2}(t)$ satisfies Lemma \ref{y1y2} is similar. First, from the Fundamental Theorem of Calculus, we have for any real number $\zeta\geq 1$ the following identity
\begin{multline*}
    \left[ U^{''}\left(H^{\zeta}_{0,1}(x)+H_{{-}1,0}(x)\right)- U^{''}(H_{{-}1,0}(x))\right] H^{'}_{{-}1,0}(x)\\
    \begin{aligned}=&
    \left[U^{''}\left(H^{\zeta}_{0,1}(x)\right)-2\right] H^{'}_{{-}1,0}(x)+U^{(3)}\left(H^{\zeta}_{0,1}(x)\right)H_{{-}1,0}(x) H^{'}_{{-}1,0}(x)\\
    &{+}\int_{0}^{1}\left[U^{(4)}\left(H^{\zeta}_{0,1}(x)+\theta H_{{-}1,0}(x)\right)-U^{(4)}\left(\theta H_{{-}1,0}(x)\right)\right]H_{{-}1,0}(x)^{2}  H^{'}_{{-}1,0}(x)(1-\theta)\,d\theta.
\end{aligned}
\end{multline*}
Therefore, estimates \eqref{le2}, \eqref{le1}, identity \eqref{u3ide} and Lemma \ref{interactt} imply for any $\zeta\geq 1$ the following estimate
\begin{multline}\label{y2ppp1}
    \left\vert\frac{d^{l}}{d\zeta^{l}}\left\langle U^{''}\left(H^{\zeta}_{0,1}(x)+H_{{-}1,0}(x)\right)-U^{''}(H_{{-}1,0}(x))-U^{''}\left(H^{\zeta}_{0,1}(x)\right)+2 ,H^{'}_{{-}1,0}(x)\partial_{x}H^{\zeta}_{0,1}(x)\right\rangle\right \vert \\\lesssim_{l} \zeta e^{{-2}\sqrt{2}\zeta}.
\end{multline}
Similarly, Lemma \ref{interactt} and identity \eqref{u3ide} imply that the functions
\begin{align*}
    int_{4}(\zeta)=&\left\langle U^{(3)}\left(H^{\zeta}_{0,1}(x)+H_{{-}1,0}(x)\right)\mathcal{G}(x-\zeta) H^{'}_{{-}1,0}(x),\, \partial_{x} H^{\zeta}_{0,1}(x)\right\rangle,\\
    int_{5}(\zeta)=&\left\langle U^{(3)}\left(H^{\zeta}_{0,1}(x)+H_{{-}1,0}(x)\right)\mathcal{G}({-}x)H^{'}_{-1,0}(x),\, \partial_{x} H^{\zeta}_{0,1}(x)\right\rangle
\end{align*} 
satisfy the estimates
\begin{equation}\label{int45}
    \left\vert int^{(l)}_{4}(\zeta)\right\vert+ \left\vert int^{(l)}_{5}(\zeta)\right\vert\lesssim_{l} e^{-\sqrt{2}\zeta},
\end{equation}
for all $\zeta\geq 1$ and any $l\in\mathbb{N}\cup\{0\}.$
Therefore, from estimates \eqref{intdecayerror}, \eqref{y2}, \eqref{y2ppp1}, \eqref{int45}, Lemma \ref{interactt} and Theorem \ref{toobig} imply that
\begin{multline}
     \left\langle Y_{2}(t,x),\, H^{'}_{0,1}\left(w_{k,v}(t,x)\right) \right\rangle=
     y_{2}(t)\int_{\mathbb{R}}\left[ U^{''}(H_{0,1}(x))-2\right]H^{'}_{0,1}(x)H^{'}_{-1,0}\left(x+\frac{d(t)}{\sqrt{1-\frac{\dot d(t)^{2}}{4}}}\right)\,dx  \\+y_{2}(t)res_{6}(v,t),
\end{multline}
where $res_{6}(v,t)$ is a real function, which satisfies for some constant $\eta_{k}\geq 0,$ if $0<v\ll 1,$   
\begin{equation*}
\left\vert \frac{\partial^{l}}{\partial t^{l}}res_{6}(v,t)\right\vert\lesssim_{l}v^{4+l}\left[\ln{\left(\frac{1}{v^{2}}\right)}+\vert t\vert v\right]^{\eta_{k}}e^{{-}2\sqrt{2}\vert t\vert v},
\end{equation*}
for all $l\in\mathbb{N}\cup\{0\}.$
So, from identity \eqref{IDD2} of Appendix Section, estimates \eqref{decayd}, 
\begin{equation*}
\left\vert \frac{d^{l}}{dx^{l}}\left[H_{{-}1,0}(x)+e^{{-}\sqrt{2}x}\right] \right\vert\lesssim_{l}\min\left(e^{{-}\sqrt{2}x},e^{{-}3\sqrt{2}x}\right),\end{equation*}  and Lemma \ref{interactt}, we conclude the proof of Lemma \ref{y1y2} for $Y_{2}(t).$
\end{proof}
\begin{remark}\label{remarky1y2}
If $v\ll1,$ using the formula $U^{''}(\phi)=2-24\phi^{2}+30\phi^{4},$ Lemmas \ref{interactt}, \ref{dots+}, the estimates \eqref{dre1}, \eqref{Y1}, \eqref{y2} and \eqref{res2} of the proof of Lemma \ref{y1y2} imply for any $s\geq 0$ that
\begin{align*}
   \max_{j\in\{1,\,2\}}\norm{Y_{j}(t)}_{H^{s}_{x}}\lesssim_{s} & \max_{j\in\{1,2\}}\vert y_{j}(t)\vert v^{2}e^{-2\sqrt{2}\vert t\vert v},\\
   \max_{j\in\{1,\,2\}}\norm{\partial_{t}Y_{j}(t)}_{H^{s}_{x}}\lesssim_{s} & \max_{j\in\{1,2\}}\vert y_{j}(t)\vert v^{3}e^{-2\sqrt{2}\vert t\vert v}+\max_{j\in\{1,2\}}\vert \dot y_{j}(t)\vert v^{2}e^{-2\sqrt{2}\vert t\vert v},\\
   \max_{j\in\{1,\,2\}}\norm{\partial^{2}_{t}Y_{j}(t)}_{H^{s}_{x}}\lesssim_{s} & \max_{j\in\{1,2\}}\vert y_{j}(t)\vert v^{4}e^{-2\sqrt{2}\vert t\vert v}+
   \max_{j\in\{1,2\}}\vert \dot y_{j}(t)\vert v^{3}e^{-2\sqrt{2}\vert t\vert v}\\&{+}
   \max_{j\in\{1,2\}}\vert y^{(2)}_{j}(t)\vert v^{2}e^{-2\sqrt{2}\vert t\vert v}.
\end{align*}
These estimates above don't depend on $k,$ because from Theorem \ref{toobig} we can verify for any $l\in\mathbb{N}\cup\{0\}$ the existence of $0<\delta_{k,l}<1$ such that if $0<v<\delta_{k,l},$ then
\begin{equation*}
    \norm{\frac{\partial^{l}}{\partial t^{l}}c_{k}(v,t)}_{L^{\infty}_{t}(\mathbb{R})}\lesssim_{l} v^{2+l}\ln{\left(\frac{1}{v}\right)},
\end{equation*}
which implies for any $l\in\mathbb{N}$ and any $v\ll1$
\begin{equation*}
    \norm{\frac{\partial^{l}}{\partial t^{l}}\left[-\frac{d(t)}{2}+c_{k}(v,t)\right]}_{L^{\infty}_{t}(\mathbb{R})}\lesssim_{l} v^{l},\quad \frac{d(t)}{2}-v<\left\vert-\frac{d(t)}{2}+c_{k}(v,t)\right\vert.
\end{equation*}
\end{remark}

\section{Energy Estimate Method}\label{energysection}
    \par In this section, we are going to repeat the main argument of Section $4$ of \cite{first} to construct a function $L:\mathbb{R}\to\mathbb{R},$ which is going to be used to estimate the energy norm of $(u(t),\partial_{t}u(t))$ during a large time interval. 
    \par First, we consider a smooth cut-off function $\chi:\mathbb{R}\to\mathbb{R}$ satisfying $0\leq \chi\leq 1$ and
    \begin{equation}\label{chi}
   \chi(x)=
   \begin{cases}
        1, \text{if $x\leq \frac{49}{100},$}\\ 
        0, \text{if $x\geq \frac{1}{2}.$}
    \end{cases}
\end{equation}
    Next, using the notation of Theorem \ref{toobig}, we denote
    \begin{equation}\label{x1x2form}
        x_{1}(t)=-\frac{d(t)}{2}+\sum_{j=2}^{k}r_{j}(v,t),\,x_{2}(t)=\frac{d(t)}{2}-\sum_{j=2}^{k}r_{j}(v,t).
   \end{equation}
   Actually, Theorem \ref{toobig} and estimates \ref{decayd} imply that
   \begin{equation}\label{x12dt}
       \max_{j\in\{1,2\}}\vert\dot x_{j} (t)\vert\lesssim v,\, \ln{\left(\frac{1}{v}\right)}\lesssim x_{2}(t)-x_{1}(t),\,\max_{j\in\{1,2\}}\vert\ddot x_{j} (t)\vert\lesssim v^{2}e^{{-}2\sqrt{2}\vert t\vert v}.
   \end{equation}
   From now on, we define the function $\chi_{1}:\mathbb{R}^{2}\to \mathbb{R}$ by
   \begin{equation}
       \chi_{1}(t,x)=\chi\left(\frac{x-x_{1}(t)}{x_{2}(t)-x_{1}(t)}\right).
   \end{equation}
 Clearly, using the identities
  \begin{align*}
      \frac{\partial }{\partial t}\chi_{1}(t,x)= & \frac{{-}\dot x_{1}(t)}{x_{2}(t)-x_{1}(t)}\dot \chi\left(\frac{x-x_{1}(t)}{x_{2}(t)-x_{1}(t)}\right)-\frac{(\dot x_{2}(t)-\dot x_{1}(t))(x-x_{1}(t))}{(x_{2}(t)-x_{1}(t))^{2}}\dot \chi\left(\frac{x-x_{1}(t)}{x_{2}(t)-x_{1}(t)}\right),\\ \frac{\partial}{\partial x}\chi_{1}(t,x)=&\frac{1}{x_{2}(t)-x_{1}(t)}\dot \chi\left(\frac{x-x_{1}(t)}{x_{2}(t)-x_{1}(t)}\right),
  \end{align*}
 we obtain the following estimates
  \begin{equation}\label{dotchi}
      \norm{\frac{\partial }{\partial t}\chi_{1}(t,x)}_{L^{\infty}_{x}(\mathbb{R})}\lesssim \frac{v}{\ln{\left(\frac{1}{v}\right)}},\, \norm{\frac{\partial }{\partial x}\chi_{1}(t,x)}_{L^{\infty}_{x}(\mathbb{R})}\lesssim\frac{1}{\ln{\left(\frac{1}{v}\right)}}.
  \end{equation}
  \par Finally, using the notation \eqref{notnot} and the functions $Y_{1}(t),\,Y_{2}(t)$ denoted respectively by \eqref{termy1} and \eqref{Y2(t)}, we define the function $A:\mathbb{R}^{2}\to \mathbb{R}$ by
   \begin{multline}\label{A(t)}
       A(t,x)=-\Lambda(\varphi_{k,v})(t,x)-\frac{8\sqrt{2}e^{-\sqrt{2}d(t)}}{1-\frac{\dot d(t)^{2}}{4}}\left[ y_{1}(t)H^{''}_{0,1}\left(w_{k,v}(t,x)\right)+y_{2}(t)H^{''}_{0,1}\left(w_{k,v}(t,{-}x)\right)\right]\\
       {-}Y_{1}(t,x)-Y_{2}(t,x)+\frac{\dot y_{1}(t)\dot d(t)}{1-\frac{\dot d(t)^{2}}{4}} H^{''}_{0,1}\left(w_{k,v}(t,x)\right)+\frac{\dot y_{2}(t)\dot d(t)}{1-\frac{\dot d(t)^{2}}{4}} H^{''}_{0,1}\left(w_{k,v}(t,{-}x)\right),
   \end{multline}
 for any $(t,x)\in\mathbb{R}^{2}.$
 Clearly, in the notation of Remark \ref{easyremark}, we have the following identity
 \begin{multline}\label{almostpdeuu}
    \partial^{2}_{t}u(t,x)-\partial^{2}_{x}u(t,x)+ U^{''}\left(H_{0,1}\left(w_{k,v}(t,x)\right)-H_{0,1}\left(w_{k,v}(t,{-}x)\right)\right)u(t,x)
    \\={-}\frac{\ddot y_{1}(t)}{\sqrt{1-\frac{\dot d(t)^{2}}{4}}} H^{'}_{0,1}\left(w_{k,v}(t,x)\right)-\frac{\ddot y_{2}(t)}{\sqrt{1-\frac{\dot d(t)^{2}}{4}}} H^{'}_{0,1}\left(w_{k,v}(t,{-}x)\right)+A(t,x)+\mathcal{Q}(t,x)
    \\{+}\left[ U^{''}\left(H_{0,1}\left(w_{k,v}(t,x)\right)-H_{0,1}\left(w_{k,v}(t,{-}x)\right)\right)- U^{''}(\varphi_{k,v}(t,x))\right]u(t,x).
 \end{multline}
   \par Next, we consider
    \begin{multline}\label{L(t)}
        L(t)=\int_{\mathbb{R}} \partial_{t}u(t,x)^{2}+\partial_{x}u(t,x)^{2}+U^{''}\left(H_{0,1}\left(w_{k,v}(t,x)-H_{0,1}\left(w_{k,v}(t,{-}x)\right)\right)\right)u(t,x)^{2}\,dx\\
    {+}2\int_{\mathbb{R}}\partial_{t}u(t,x)\partial_{x}u(t,x)\left[\dot x_{1}(t)\chi_{1}(t,x)+\dot x_{2}(t)\left(1-\chi_{1}(t,x)\right)\right]\,dx
        \\{-}2\int_{\mathbb{R}} u(t,x)A(t,x)\,dx.
    \end{multline}
    From now on, we use the notation $\overrightarrow{u}(t)=(u(t),\partial_{t}u(t))\in H^{1}_{x}(\mathbb{R})\times L^{2}_{x}(\mathbb{R}).$
    The main objective of the Section $3$ is to demonstrate the following theorem. 
\begin{theorem}\label{energyL}
    There exist constants $K,\,c>0$ and, for any $k\in\mathbb{N}_{\geq 3},$ there exists $0<\delta(k)<1$  such that if $0<v\leq  \delta(k),$ then the function $L(t)$ denoted in \eqref{L(t)} satisfies, while the following condition
    \begin{equation}\label{conditiony1y2}
        \max_{j\in\{1,\,2\}} v^{2}\vert y_{j}(t)\vert+v\vert\dot y_{j}(t)\vert<v^{2k}\ln{\left(\frac{1}{v}\right)}^{n_{k}}
    \end{equation} is true, the estimates
    \begin{equation*}
        c\norm{\overrightarrow{u}(t)}_{H^{1}_{x}\times L^{2}_{x}}^{2}\leq L(t)+C(k) v^{4k}\ln{\left(\frac{1}{v}\right)}^{2n_{k}},
   \end{equation*}
   and
    \begin{align*}
        \left\vert \dot L(t) \right\vert\leq K\left[\frac{v}{\ln{\left(\frac{1}{v}\right)}}\norm{\overrightarrow{u}(t)}_{H^{1}_{x}\times L^{2}_{x}}^{2}+C(k)\norm{\overrightarrow{u}(t)}_{H^{1}_{x}\times L^{2}_{x}}v^{2k+1}\ln{\left(\frac{1}{v}\right)}^{n_{k}}\right]\\
        +v \max_{j\in\{1,2\}}\vert\ddot y_{j}(t)\vert\norm{\overrightarrow{u}(t)}_{H^{1}_{x}\times L^{2}_{x}}+K\max_{j\in\{3.7\}}\norm{\overrightarrow{u}(t)}_{H^{1}_{x}\times L^{2}_{x}}^{j},   \end{align*}
where $C(k)>0$ is a constant depending only on $k$ and $n_{k}$ is the number defined in the statement of Theorem \ref{toobig}.
\end{theorem}
\begin{proof}[Proof of Theorem \ref{energyL}.]
To simplify the proof of this theorem, we describe briefly the organization of our arguments. First, we denote $L(t)$ as
\begin{equation*}
    L(t)=L_{1}(t)+L_{2}(t)+L_{3}(t),
\end{equation*}
such that
\begin{align}\tag{L1}\label{l1}
    L_{1}(t)=&\int_{\mathbb{R}} \partial_{t}u(t,x)^{2}+\partial_{x}u(t,x)^{2}+\ddot U\left(H_{0,1}\left(w_{k,v}(t,x)-H_{0,1}\left(w_{k,v}(t,{-}x)\right)\right)\right)u(t,x)^{2}\,dx,\\ \tag{L2}\label{l2}
    L_{2}(t)=&2\int_{\mathbb{R}}\partial_{t}u(t,x)\partial_{x}u(t,x)\left[\dot x_{1}(t)\chi_{1}(t,x)+\dot x_{2}(t)\left(1-\chi_{1}(t,x)\right)\right]\,dx,\\ \tag{L3}\label{l3}
    L_{3}(t)=&{-}2\int_{\mathbb{R}} u(t,x)A(t,x)\,dx.
\end{align}
Next, instead of estimating the size of $\left\vert \dot L(t)\right\vert,$ we are going to estimate $ \dot L_{j}(t)$ for each $j\in\{1,\,2,\,3\}.$ Then, using these estimates, we can evaluate with high precision 
\begin{equation*}
    \left\vert  \dot L_{1}(t)+ \dot L_{2}(t) + \dot L_{3}(t) \right\vert,
\end{equation*}
and obtain the second inequality of Theorem \ref{energyL}. The proof of the first inequality of Theorem \ref{energyL} is short and it will be done later.
\par From identity \eqref{d}, Remark \ref{remarky1y2} and equation \eqref{A(t)} satisfied by $A(t,x),$ we deduce from the triangle inequality that
\begin{equation*}
    \norm{A(t,x)}_{H^{1}_{x}(\mathbb{R})}\lesssim \norm{\Lambda(\varphi_{k,v})(t,x)}_{H^{1}_{x}(\mathbb{R})}+v^{2}e^{-2\sqrt{2}v\vert t\vert}\max_{j\in\{1,2\}}\vert y_{j}(t)\vert +v\max_{j\in\{1,2\}}\vert\dot y_{j}(t)\vert.
\end{equation*}
Therefore, from Theorem \ref{approximated theorem} and Theorem \ref{toobig}, we obtain the existence of a value $C(k)>0$ depending only on $k$ such that if $v\ll 1,$ then
\begin{multline}\label{Aestimat}
    \norm{A(t,x)}_{H^{1}(\mathbb{R})}\lesssim C(k)v^{2k}\left(\ln{\left(\frac{1}{v}\right)}+\vert t\vert v\right)^{n_{k}}e^{-2\sqrt{2}\vert t\vert v}+v^{2}e^{-2\sqrt{2}\vert t\vert v}\max_{j\in\{1,2\}}\vert y_{j}(t)\vert\\{+}v\max_{j\in\{1,2\}}\vert\dot y_{j}(t)\vert.
\end{multline}
In conclusion, we obtain from \eqref{l3} and Cauchy-Schwartz inequality the existence of a value $C(k)>0$ depending only on $k$ satisfying
\begin{multline}\label{L2est}
    \vert L_{3}(t)\vert\lesssim \norm{u(t)}_{L^{2}_{x}}\Bigg[C(k)v^{2k}\left(\ln{\left(\frac{1}{v}\right)}+\vert t\vert v\right)^{n_{k}}e^{-2\sqrt{2}\vert t\vert v}+v^{2}e^{-2\sqrt{2}\vert t\vert v}\max_{j\in\{1,2\}}\vert y_{j}(t)\vert\\+\max_{j\in\{1,2\}}\vert\dot y_{j}(t)\vert v\Bigg].
\end{multline}
\par Next, Lemmas \ref{dots+}, \ref{mulsoblemma}, Remark \ref{remarky1y2} and identity \eqref{A(t)} satisfied by $A(t,x)$ imply the following inequality
\begin{align*}
    \norm{\partial_{t}A(t,x)}_{H^{1}_{x}(\mathbb{R})}\lesssim \norm{\frac{\partial}{\partial t}\left[\Lambda(\phi_{k})(v,t,x)\right]}_{H^{1}_{x}(\mathbb{R})}+\max_{j\in\{1,2\}}\vert y_{j}(t)\vert v^{3} e^{-2\sqrt{2}\vert t\vert v}+\max_{j\in\{1,2\}}\vert\dot y_{j}(t)\vert v^{2}\\+\max_{j\in\{1,2\}}\vert\ddot y_{j}(t) \vert v, 
\end{align*}
from which with Theorem \ref{toobig} we conclude the existence of a new value $C(k)$ depending only on $k$ satisfying
\begin{multline}\label{dA2estimat}
     \norm{\partial_{t}A(t,x)}_{H^{1}_{x}}\lesssim C(k)v^{2k+1}\left(\ln{\left(\frac{1}{v}\right)}+\vert t\vert v\right)^{n_{k}}e^{-2\sqrt{2}\vert t\vert v}+\max_{j\in\{1,2\}}\vert y_{j}(t)\vert v^{3}e^{-2\sqrt{2}\vert t\vert v}\\{+}\max_{j\in\{1,2\}}\vert\dot y_{j}(t)\vert v^{2}+\max_{j\in\{1,2\}}\vert\ddot y_{j}(t) \vert v.
\end{multline}
In conclusion,
the identity \eqref{l3}, estimate \eqref{dA2estimat} and Cauchy-Schwartz inequality imply the existence of a new value $C(k)>0$ depending only on $k,$ which satisfies 
\begin{multline}\label{dl2estimat}
   \left\vert \dot L_{3}(t)+2\int_{\mathbb{R}}\partial_{t}u(t,x)A(t,x)\,dx\right\vert \\ \lesssim \norm{u(t,x)}_{L^{2}_{x}}\left[C(k)
   v^{2k+1}\left(\ln{\left(\frac{1}{v}\right)}+\vert t\vert v\right)^{n_{k}}e^{-2\sqrt{2}\vert t\vert v}+\max_{j\in\{1,2\}}\vert y_{j}(t)\vert v^{3}e^{-2\sqrt{2}\vert t\vert v}\right]\\{+}\norm{u(t,x)}_{L^{2}_{x}}\left[\max_{j\in\{1,2\}}\vert\dot y_{j}(t)\vert v^{2}+\max_{j\in\{1,2\}}\vert\ddot y_{j}(t) \vert v
   \right].
\end{multline}
\par Next, Theorem \ref{toobig} imply that if $v\ll 1,$ then
\begin{multline}\label{dl1tt}
    \dot L_{1}(t)\\
    \begin{aligned}
    =&2\int_{\mathbb{R}}\partial_{t}u(t,x)\left[\partial^{2}_{t}u(t,x)-\partial^{2}_{x}u(t,x)+ U^{''}\left(H_{0,1}\left(w_{k,v}(t,x)\right)-H_{0,1}\left(w_{k,v}(t,{-}x)\right)\right)u(t,x)\right]\,dx\\
    &{-}\frac{\dot d(t)}{2\left(1-\frac{\dot d(t)^{2}}{4}\right)^{\frac{1}{2}}}\int_{\mathbb{R}}U^{(3)}\left(H_{0,1}\left(w_{k,v}(t,x)\right)-H_{0,1}\left(w_{k,v}(t,{-}x)\right)\right) H^{'}_{0,1}\left(w_{k,v}(t,x)\right)u(t,x)^{2}\,dx\\
    &{+}\frac{\dot d(t)}{2\left(1-\frac{\dot d(t)^{2}}{4}\right)^{\frac{1}{2}}}\int_{\mathbb{R}}U^{(3)}\left(H_{0,1}\left(w_{k,v}(t,x)\right)-H_{0,1}\left(w_{k,v}(t,{-}x)\right)\right) H^{'}_{0,1}\left(w_{k,v}(t,{-}x)\right)u(t,x)^{2}\,dx\\
    &{+}O\left(\frac{v}{\ln{\left(\frac{1}{v}\right)}}\norm{(u(t),\partial_{t}u(t))}_{H^{1}_{x},L^{2}_{x})}^{2}\right)
\end{aligned}
\end{multline}
Therefore, from Lemma \ref{pdeu}, identity \eqref{A(t)}, Remark \ref{easyremark}, hypothesis \eqref{conditiony1y2}, estimates \eqref{dl2estimat}, \eqref{dl1tt} and orthogonality conditions \eqref{orthogonall}, we obtain the existence of a value $C(k)>0$ depending only on $k$ such that if $v\ll 1,$ then 
\begin{multline}\label{suml1l3}
    \dot L_{1}(t)+\dot L_{3}(t)\\
    \begin{aligned}
    = &2\int_{\mathbb{R}}\partial_{t}u(t,x)\left[ U^{''}\left(H_{0,1}\left(w_{k,v}(t,x)\right)-H_{0,1}\left(w_{k,v}(t,{-}x)\right)\right)- U^{''}\left(\varphi_{k,v}(t,x)\right)\right]u(t,x)\,dx
    \\
&{+}\frac{\dot d(t)}{2\sqrt{1-\frac{\dot d(t)^{2}}{4}}}\int_{\mathbb{R}}U^{(3)}\left(H_{0,1}\left(w_{k,v}(t,x)\right)-H_{0,1}\left(w_{k,v}(t,{-}x)\right)\right) H^{'}_{0,1}\left(w_{k,v}(t,{-}x)\right)u(t,x)^{2}\,dx\\&{-}\frac{\dot d(t)}{2\sqrt{1-\frac{\dot d(t)^{2}}{4}}}\int_{\mathbb{R}}U^{(3)}\left(H_{0,1}\left(w_{k,v}(t,x)\right)-H_{0,1}(w_{k,v}(t,{-}x))\right) H^{'}_{0,1}\left(w_{k,v}(t,x)\right)u(t,x)^{2}\,dx\\
    &{+}O\left(v\max_{j\in\{1,2\}}\vert \ddot y_{j}(t)\vert\norm{u(t)}_{H^{1}_{x}(\mathbb{R})}+\max_{j\in\{3,7\}}\norm{\overrightarrow{u}(t)}_{H^{1}_{x}\times L^{2}_{x}}^{j}+\norm{\overrightarrow{u}(t)}_{H^{1}_{x}\times L^{2}_{x}}\max_{j\in\{1,2\}}\vert y_{j}(t)\vert^{2}\right)\\
&{+}O\left(\norm{\overrightarrow{u}(t)}_{H^{1}_{x}\times L^{2}_{x}}\left[\max_{j\in\{1,2\}}\vert\dot y_{j}(t)\vert v^{2}+\vert y_{j}(t)\vert v^{3}e^{-2\sqrt{2}\vert t\vert v}\right]+\norm{\overrightarrow{u}(t)}_{H^{1}_{x}\times L^{2}_{x}}^{2}\frac{v}{\ln{\left(\frac{1}{v}\right)}}\right)\\&{+}O\left(C(k)\norm{\overrightarrow{u}(t)}_{H^{1}_{x}\times L^{2}_{x}}v^{2k+1}\ln{\left(\frac{1}{v}\right)}^{n_{k}}\right).
\end{aligned}
\end{multline}
Moreover, using estimates \eqref{decayd}, Lemma \ref{mulsoblemma} and identity $U(\phi)=\phi^{2}(1-\phi^{2})^{2},$ we obtain from Theorem \ref{toobig} that if $0<v\ll 1$ and $s\geq 0,$ then
\begin{equation*}
    \norm{\left[ U^{''}\left(H_{0,1}\left(w_{k,v}(t,x)\right)-H_{0,1}\left(w_{k,v}(t,{-}x)\right)\right)- U^{''}\left(\phi_{k,v}(t,x)\right)\right]}_{H^{s
}_{x}}\lesssim_{s,k}  v^{2}e^{{-}2\sqrt{2}\vert t\vert v}.
\end{equation*}
Therefore, we deduce using Cauchy-Schwarz inequality that
\begin{multline*}
    \left\vert 2\int_{\mathbb{R}}\partial_{t}u(t,x)\left[ U^{''}\left(H_{0,1}\left(w_{k,v}(t,x)\right)-H_{0,1}\left(w_{k,v}(t,{-}x)\right)\right)- U^{''}\left(\phi_{k,v}(t,x)\right)\right]u(t,x)\,dx\right\vert\\ \lesssim
    \norm{\left[ U^{''}\left(H_{0,1}\left(w_{k,v}(t,x)\right)-H_{0,1}\left(w_{k,v}(t,x)\right)\right)- U^{''}\left(\phi_{k,v}(t,x)\right)\right]u(t,x)}_{L^{2}_{x}}
\norm{\partial_{t}u(t,x)}_{L^{2}_{x}}\\ \lesssim
\norm{\left[ U^{''}\left(H_{0,1}\left(w_{k,v}(t,x)\right)-H_{0,1}\left(w_{k,v}(t,{-}x)\right)\right)- U^{''}\left(\phi_{k,v}(t,x)\right)\right]}_{H^{1}_{x}(\mathbb{R})}\norm{\overrightarrow{u}(t)}_{H^{1}_{x}\times L^{2}_{x}}^{2}\\ \lesssim
v^{2}\norm{\overrightarrow{u}(t)}_{H^{1}_{x}\times L^{2}_{x}}^{2}.
\end{multline*}
In conclusion,
\begin{multline}\label{suml1l32}
    \dot L_{1}(t)+\dot L_{3}(t)\\= 
    \frac{\dot d(t)}{2\left(1-\frac{\dot d(t)^{2}}{4}\right)^{\frac{1}{2}}}\int_{\mathbb{R}}U^{(3)}\left(H_{0,1}\left(w_{k,v}(t,x)\right)-H_{0,1}\left(w_{k,v}(t,{-}x)\right)\right) H^{'}_{0,1}\left(w_{k,v}(t,{-}x)\right)u(t,x)^{2}\,dx\\-\frac{\dot d(t)}{2\sqrt{1-\frac{\dot d(t)^{2}}{4}}}\int_{\mathbb{R}}U^{(3)}\left(H_{0,1}\left(w_{k,v}(t,x)\right)-H_{0,1}(w_{k,v}(t,-x))\right) H^{'}_{0,1}\left(w_{k,v}(t,x)\right)u(t,x)^{2}\,dx\\
    {+}O\left(v\max_{j\in\{1,2\}}\vert \ddot y_{j}(t)\vert\norm{u(t)}_{H^{1}_{x}(\mathbb{R})}{+}\max_{j\in\{3,7\}}\norm{\overrightarrow{u}(t)}_{H^{1}_{x}\times L^{2}_{x}}^{j}+\norm{\overrightarrow{u}(t)}_{H^{1}_{x}\times L^{2}_{x}}\max_{j\in\{1,2\}}\vert y_{j}(t)\vert^{2}\right)\\
{+}O\left(\norm{\overrightarrow{u}(t)}_{H^{1}_{x}\times L^{2}_{x}}\left[\max_{j\in\{1,2\}}\vert\dot y_{j}(t)\vert v^{2}+\vert y_{j}(t)\vert v^{3}e^{-2\sqrt{2}\vert t\vert v}\right]+\norm{\overrightarrow{u}(t)}_{H^{1}_{x}\times L^{2}_{x}}^{2}\frac{v}{\ln{\left(\frac{1}{v}\right)}}\right)\\{+}O\left(C(k)\norm{\overrightarrow{u}(t)}_{H^{1}_{x}\times L^{2}_{x}}v^{2k+1}\ln{\left(\frac{1}{v}\right)}^{n_{k}}\right).
\end{multline}
\par Based on the arguments of \cite{jkl} and \cite{first}, we are going to estimate the derivative of $L_{2}(t),$ for more accurate information see the third step of Lemma $4.2$ in \cite{jkl} or Theorem $4.1$ of \cite{first}. Because of an argument of analogy, we only need to estimate the time derivative of
\begin{equation*}
    L_{2,1}(t)=2\dot x_{1}(t)\int_{\mathbb{R}}\chi_{1}(t,x)\partial_{t}u(t,x)\partial_{x}u(t,x)\,dx
\end{equation*} 
to evaluate with high precision the derivative of $L_{2}(t).$
From the estimates \eqref{dotchi}, we can verify first that if $v\ll 1,$ then
\begin{align*}
    \dot L_{2,1}(t)=2\dot x_{1}(t)\int_{\mathbb{R}}\chi_{1}(t,x)\partial^{2}_{t}u(t,x)\partial_{x}u(t,x)\,dx
    +2\dot x_{1}(t)\int_{\mathbb{R}}\chi_{1}(t,x)\partial_{t}u(t,x)\partial^{2}_{x,t}u(t,x)\,dx\\+O\left(\frac{v}{\ln{\left(\frac{1}{v}\right)}}\norm{\overrightarrow{u}(t)}_{H^{1}_{x}\times L^{2}_{x}}^{2}\right),
\end{align*}
from which we deduce, using integration by parts and estimates \eqref{x12dt}, \eqref{dotchi}, that
\begin{align*}
    \dot L_{2,1}(t)=&2\dot x_{1}(t)\int_{\mathbb{R}}\chi_{1}(t,x)\partial^{2}_{t}u(t,x)\partial_{x}u(t,x)\,dx+O\left(\frac{v}{\ln{\left(\frac{1}{v}\right)}}\norm{\overrightarrow{u}(t)}_{H^{1}_{x}\times L^{2}_{x}}^{2}\right)\\=&
    2\dot x_{1}(t)\int_{\mathbb{R}}\chi_{1}(t,x)\left[\partial^{2}_{t}u(t,x)-\partial^{2}_{x}u(t,x)\right]\partial_{x}u(t,x)\,dx\\&{+}2\dot x_{1}(t)\int_{\mathbb{R}}\chi_{1}(t,x) U^{''}\left(H_{0,1}\left(w_{k,v}(t,x)\right)-H_{0,1}\left(w_{k,v}(t,{-}x)\right)\right)u(t,x)\partial_{x}u(t,x)\,dx\\
    &{+}2\dot x_{1}(t)\int_{\mathbb{R}}\chi_{1}(t,x)\partial^{2}_{x}u(t,x)\partial_{x}u(t,x)\,dx
    \\&{-}2\dot x_{1}(t)\int_{\mathbb{R}}\chi_{1}(t,x)U^{''}\left(H_{0,1}\left(w_{k,v}(t,x)\right)-H_{0,1}\left(w_{k,v}(t,{-}x)\right)\right)u(t,x)\partial_{x}u(t,x)\,dx\\&{+}O\left(\frac{v}{\ln{\left(\frac{1}{v}\right)}}\norm{\overrightarrow{u}(t)}_{H^{1}_{x}\times L^{2}_{x}}^{2}\right),
\end{align*}
and, after using integration by parts again, we deduce from \eqref{dotchi} that 
\begin{align*}
    \dot L_{2,1}(t)=&
    2\dot x_{1}(t)\int_{\mathbb{R}}\chi_{1}(t,x)\left[\partial^{2}_{t}u(t)-\partial^{2}_{x}u(t)\right]\partial_{x}u(t)\,dx\\&{+}2\dot x_{1}(t)\int_{\mathbb{R}}\chi_{1}(t) U^{''}\left(H_{0,1}\left(w_{k,v}(t,x)\right)-H_{0,1}\left(w_{k,v}(t,{-}x)\right)\right)u(t)\partial_{x}u(t)\,dx\\
    &{+}\frac{\dot x_{1}(t)}{\sqrt{1-\frac{\dot d(t)^{2}}{4}}}\int_{\mathbb{R}}\chi_{1}(t) U^{(3)}\left(H_{0,1}\left(w_{k,v}(t,x)\right)-H_{0,1}\left(w_{k,v}(t,{-}x)\right)\right) H^{'}_{0,1}\left(w_{k,v}(t,x)\right)u(t)^{2}\,dx\\&{+}\frac{\dot x_{1}(t)}{\sqrt{1-\frac{\dot d(t)^{2}}{4}}}\int_{\mathbb{R}}\chi_{1}(t)U^{(3)}\left(H_{0,1}\left(w_{k,v}(t,x)\right)-H_{0,1}\left(w_{k,v}(t,{-}x)\right)\right) H^{'}_{0,1}\left(w_{k,v}(t,{-}x)\right) u(t)^{2}\,dx\\&{+}O\left(\frac{v}{\ln{\left(\frac{1}{v}\right)}}\norm{\overrightarrow{u}(t)}_{H^{1}_{x}(\mathbb{R})\times L^{2}_{x}(\mathbb{R})}^{2}\right).    
\end{align*}
Next, using estimates \eqref{le2} satisfied by $H_{0,1},$ definition of $\chi_{1}(t,x),$ Theorem \ref{toobig} and identity \eqref{wk}, we deduce, for $v\ll 1,$ the following inequality
\begin{equation*}
    \left\vert\chi_{1}(t,x) H^{'}_{0,1}\left(w_{k,v}(t,x)\right)\right\vert+\left\vert(1-\chi_{1}(t,x)) H^{'}_{0,1}\left(w_{k,v}(t,{-}x)\right)\right\vert\lesssim e^{-\sqrt{2}\frac{49d(t)}{100}}\lesssim v^{\frac{98}{100}}\ll \frac{1}{\ln{\left(\frac{1}{v}\right)}},
\end{equation*}
from which we conclude that
\begin{align*}
    \dot L_{2,1}(t)=&2\dot x_{1}(t)\int_{\mathbb{R}}\chi_{1}(t)\left[\partial^{2}_{t}u(t,x)-\partial^{2}_{x}u(t,x)\right]\partial_{x}u(t,x)\,dx\\&{+}2\dot x_{1}(t)\int_{\mathbb{R}} \chi_{1}(t)U^{''}\left(H_{0,1}\left(w_{k,v}(t,x)\right)-H_{0,1}\left(w_{k,v}(t,{-}x)\right)\right)u(t,x)\partial_{x}u(t,x)\,dx\\
    &{+}\frac{\dot x_{1}(t)}{\sqrt{1-\frac{\dot d(t)^{2}}{4}}}\int_{\mathbb{R}}U^{(3)}\left(H_{0,1}\left(w_{k,v}(t,x)\right)-H_{0,1}\left(w_{k,v}(t,{-}x)\right)\right) H^{'}_{0,1}\left(w_{k,v}(t,{-}x)\right)u(t,x)^{2}\,dx
    \\&{+}O\left(\frac{v}{\ln{\left(\frac{1}{v}\right)}}\norm{\overrightarrow{u}(t)}_{H^{1}_{x}\times L^{2}_{x}}^{2}\right).
\end{align*}
Furthermore, from Remark \ref{easyremark}, estimate \eqref{Aestimat} of $A(t,x)$ and identity \eqref{almostpdeuu}
satisfied by $u(t,x),$ we conclude the existence of a value $C(k)>0$ depending only on $k$ and satisfying, for any positive number $v\ll 1,$
\begin{align*}
    \dot L_{2,1}(t)=&
    \frac{\dot x_{1}(t)}{\sqrt{1-\frac{\dot d(t)^{2}}{4}}}\int_{\mathbb{R}}U^{(3)}\left(H_{0,1}\left(w_{k,v}(t,x)\right)-H_{0,1}\left(w_{k,v}(t,{-}x)\right)\right)H^{'}_{0,1}\left(w_{k,v}(t,{-}x)\right)u(t,x)^{2}\,dx\\&{+}O\left(\norm{\overrightarrow{u}(t)}_{H^{1}_{x}\times L^{2}_{x}}\left[v\max_{j\in\{1,2\}}\vert\ddot y_{j}(t)\vert+C(k)v^{2k+1}\ln{\left(\frac{1}{v}\right)}^{n_{k}}+v\max_{j\in\{2,6\}}\norm{\overrightarrow{u}(t)}_{H^{1}_{x}\times L^{2}_{x}}^{j}\right]\right)\\
&{+}O\left(\norm{\overrightarrow{u}(t)}_{H^{1}_{x}\times L^{2}_{x}}\left[v^{3}e^{-2\sqrt{2}v\vert t\vert}\max_{j\in\{1,2\}}\vert y_{j}(t) \vert+v^{2}\vert \dot y_{j}(t)\vert\right]+\frac{v}{\ln{\left(\frac{1}{v}\right)}}\norm{\overrightarrow{u}(t)}_{H^{1}_{x}\times L^{2}_{x}}^{2}\right).
\end{align*}
Therefore, using an argument of analogy, we obtain, for any positive number $v\ll1,$ that
\begin{align}\nonumber
    \dot L_{2}(t)=&
    \frac{\dot x_{2}(t)}{\sqrt{1-\frac{\dot d(t)^{2}}{4}}}\int_{\mathbb{R}}U^{(3)}\left(H_{0,1}\left(w_{k,v}(t,x)\right)-H_{0,1}\left(w_{k,v}(t,{-}x)\right)\right) H^{'}_{0,1}\left(w_{k,v}(t,x)\right)u(t,x)^{2}\,dx\\ \nonumber
    &{+}\frac{\dot x_{1}(t)}{\sqrt{1-\frac{\dot d(t)^{2}}{4}}}\int_{\mathbb{R}}U^{(3)}\left(H_{0,1}\left(w_{k,v}(t,x)\right)-H_{0,1}\left(w_{k,v}(t,{-}x)\right)\right) H^{'}_{0,1}\left(w_{k,v}(t,{-}x)\right)u(t,x)^{2}\,dx\\ \nonumber
    &{+}O\left(\norm{\overrightarrow{u}(t)}_{H^{1}_{x}\times L^{2}_{x}}\left[v\max_{j\in\{1,2\}}\vert\ddot y_{j}(t)\vert+C(k)v^{2k+1}\ln{\left(\frac{1}{v}\right)}^{n_{k}}\right]+v\max_{j\in\{3,7\}}\norm{\overrightarrow{u}(t)}_{H^{1}_{x}\times L^{2}_{x}}^{j}\right)\\ \label{dL11} &{+}O\left(\norm{\overrightarrow{u}(t)}_{H^{1}_{x}\times L^{2}_{x}}\left[v^{3}e^{-2\sqrt{2}v\vert t\vert}\max_{j\in\{1,2\}}\vert y_{j}(t) \vert+v^{2}\vert \dot y_{j}(t)\vert\right]+\frac{v}{\ln{\left(\frac{1}{v}\right)}}\norm{\overrightarrow{u}(t)}_{H^{1}_{x}\times L^{2}_{x}}^{2}\right)
,\end{align}
where $C(k)>0$ is a parameter depending only on $k.$ Moreover, using \eqref{x1x2form} and Theorem \ref{toobig}, we deduce from estimate \eqref{dL11} that
\begin{align} \nonumber
    \dot L_{2}(t)=
    &\frac{\dot d(t)}{\sqrt{4-\dot d(t)^{2}}}\int_{\mathbb{R}}U^{(3)}\left(H_{0,1}\left(w_{k,v}(t,x)\right)-H_{0,1}\left(w_{k,v}(t,{-}x)\right)\right) H^{'}_{0,1}\left(w_{k,v}(t,x)\right)u(t,x)^{2}\,dx\\
    \nonumber
    &{-}\frac{\dot d(t)}{\sqrt{4-\dot d(t)^{2}}}\int_{\mathbb{R}}U^{(3)}\left(H_{0,1}\left(w_{k,v}(t,x)\right)-H_{0,1}\left(w_{k,v}(t,{-}x)\right)\right) H^{'}_{0,1}\left(w_{k,v}(t,{-}x)\right)u(t,x)^{2}\,dx\\
    \nonumber &{+}O\left(\norm{\overrightarrow{u}(t)}_{H^{1}_{x}\times L^{2}_{x}}\left[v\max_{j\in\{1,2\}}\vert\ddot y_{j}(t)\vert+C(k)v^{2k+1}\ln{\left(\frac{1}{v}\right)}^{n_{k}}\right]+v\max_{j\in\{3,7\}}\norm{\overrightarrow{u}(t)}_{H^{1}_{x}\times L^{2}_{x}}^{j}\right)\\ \label{dL112}&{+}O\left(\norm{\overrightarrow{u}(t)}_{H^{1}_{x}\times L^{2}_{x}}\left[v^{3}e^{-2\sqrt{2}v\vert t\vert}\max_{j\in\{1,2\}}\vert y_{j}(t) \vert+v^{2}\vert \dot y_{j}(t)\vert\right]+\frac{v}{\ln{\left(\frac{1}{v}\right)}}\norm{\overrightarrow{u}(t)}_{H^{1}_{x}\times L^{2}_{x}}^{2}\right).
\end{align}
\par Finally, the estimate \eqref{dL112} and \eqref{suml1l3} imply, for any $k\in\mathbb{N}_{\geq 3},$ the existence of a parameter $C(k)>0,$ depending  only on $k,$ which satisfies for any positive number $v\ll 1$ the estimate
\begin{align}\nonumber
    \vert \dot L(t)\vert = &O\left(v\max_{j\in\{1,2\}}\vert \ddot y_{j}(t)\vert\norm{\overrightarrow{u}(t)}_{H^{1}_{x}\times L^{2}_{x}}+\max_{j\in\{3,7\}}\norm{\overrightarrow{u}(t)}_{H^{1}_{x}\times L^{2}_{x}}^{j}\right)\\ \nonumber &{+}O\left(\norm{\overrightarrow{u}(t)}_{H^{1}_{x}\times L^{2}_{x}}\max_{j\in\{1,2\}}\vert y_{j}(t)\vert^{2}\right)\\ \nonumber
&{+}O\left(\norm{\overrightarrow{u}(t)}_{H^{1}_{x}\times L^{2}_{x}}\left[\max_{j\in\{1,2\}}\vert\dot y_{j}(t)\vert v^{2}+\vert y_{j}(t)\vert v^{3}e^{-2\sqrt{2}\vert t\vert v}\right]\right)\\ \label{DL(t)}
&{+}O\left(\norm{\overrightarrow{u}(t)}_{H^{1}_{x}\times L^{2}_{x}}^{2}\frac{v}{\ln{\left(\frac{1}{v^{2}}\right)}}+C(k)\norm{\overrightarrow{u}(t)}_{H^{1}_{x}\times L^{2}_{x}}v^{2k+1}\ln{\left(\frac{1}{v}\right)}^{n_{k}}\right),
\end{align}
from which we obtain the existence of a new constant $C(k)>0$ satisfying the second inequality of Theorem \ref{energyL} if the condition \eqref{conditiony1y2} is true and $v\ll 1.$
\par Now, it remains to prove the first inequality of Theorem \ref{energyL}. Using change of variables and Lemma \ref{coerccc}, it is not difficult to verify that there exists $K>0$ such that if $v\ll 1,$ then
\begin{equation*}
    L_{1}(t)\geq K\norm{(u(t),\partial_{t}u(t))}_{H^{1}_{x}\times L^{2}_{x}}^{2}.
\end{equation*}
Next, from the definition of $L_{2}(t)$ and estimates \eqref{x12dt}, we obtain that if $v\ll 1,$ then
\begin{equation*}
    \left\vert L_{2}(t)\right\vert\ll v^{\frac{3}{4}}\norm{(u(t),\partial_{t}u(t))}_{H^{1}_{x}\times L^{2}_{x}}^{2},
\end{equation*}
and while condition \eqref{conditiony1y2} is true, we deduce from Theorem \ref{toobig} and estimate \eqref{Aestimat} the following inequality
\begin{equation*}
     \left\vert L_{3}(t)\right\vert\lesssim_{k} \norm{(u(t),\partial_{t}u(t))}_{H^{1}_{x}\times L^{2}_{x}} v^{2k}\ln{\left(\frac{1}{v}\right)}^{n_{k}}.
\end{equation*}
So, using Young inequality, we can find a parameter $C_{1}(k)>0$ large enough depending only on $k$ such that
\begin{equation*}
    \left\vert L_{3}(t)\right\vert\leq \frac{K}{2}\norm{(u(t),\partial_{t}u(t))}_{H^{1}_{x}\times L^{2}_{x}}^{2}+C_{1}(k)v^{4k}\ln{\left(\frac{1}{v}\right)}^{2n_{k}}.
\end{equation*}
In conclusion, all the estimates above imply the first inequality of Theorem \ref{energyL} if $0<v\ll 1$ and condition \eqref{conditiony1y2} is true.
\end{proof}

\section{Proof of Theorem \ref{energyE}}\label{p1main}
\par From the information of Theorem \ref{energyL} in the last section, we are ready to start the demonstration of Theorem \ref{energyE}.
\begin{proof}[Proof of Theorem \ref{energyE}]
First, for any $(t,x)\in\mathbb{R}^{2},$ Lemma \ref{pdeu} implies   that $\phi(t,x)$ has the following representation
\begin{equation*}
    \phi(t,x)=\varphi_{k,v}(t,x)+\frac{y_{1}(t)}{\sqrt{1-\frac{\dot d(t)^{2}}{4}}} H^{'}_{0,1}\left(w_{k,v}(t,x)\right)+\frac{y_{2}(t)}{\sqrt{1-\frac{\dot d(t)^{2}}{4}}} H^{'}_{0,1}\left(w_{k,v}(t,{-}x)\right)+u(t,x),
\end{equation*}
such that the function $u(t,x)$ satisfies the orthogonality conditions \eqref{orthogonall} and $y_{1},\,y_{2}$ are functions in $C^{2}(\mathbb{R}).$ 
\\
\textbf{Step $1.$}(Ordinary differential system of $y_{1}(t),\,y_{2}(t).$)
From Remarks \ref{n2}, \ref{easyremark} and the definition of $A(t,x)$ in \eqref{A(t)}, we have that $u(t,x)$ is a solution of a partial differential equation of the form
\begin{multline}\label{pdeu2}
\partial^{2}_{t}u(t,x)-\partial^{2}_{x}u(t,x)+ U^{''}\left(H_{0,1}\left(w_{k,v}(t,x)\right)-H_{0,1}\left(w_{k,v}(t,{-}x)\right)\right)u(t,x)\\=
    {-}\frac{\ddot y_{1}(t)}{\sqrt{1-\frac{\dot d(t)^{2}}{4}}} H^{'}_{0,1}\left(w_{k,v}(t,x)\right)
  {-}\frac{\ddot y_{2}(t)}{\sqrt{1-\frac{\dot d(t)^{2}}{4}}} H^{'}_{0,1}\left(w_{k,v}(t,{-}x)\right)
    \\{+}A(t,x)+\mathcal{P}_{1}(v,t,x),
\end{multline}
where $\mathcal{P}_{1}(v,t,x)$ satisfies for any  $0<v\ll 1$ and any $t\in\mathbb{R}$ the inequality
\begin{align*}
    \norm{\mathcal{P}_{1}(v,t,x)}_{H^{1}_{x}}\lesssim & \norm{u(t)}_{H^{1}_{x}}^{2}+\max_{j\in\{1,\,2\}}\vert y_{j}(t)\vert^{2}+
    \max_{j\in\{1,2\}}\vert\dot y_{j}(t)\vert v^{3}\left(\ln{\left(\frac{1}{v^{2}}\right)}+\vert t\vert v\right)e^{-2\sqrt{2}\vert t\vert v}\\& {+} \norm{u(t)}_{H^{1}_{x}}^{6}+\max_{j\in\{1,2\}} \vert y_{j}(t)\vert^{6}+\max_{j\in\{1,2\}}\vert y_{j}(t)\vert v^{4}\left(\ln{\left(\frac{1}{v^{2}}\right)}+\vert t\vert v\right)e^{-2\sqrt{2}\vert t\vert v}.
\end{align*}
With the objective of simplifying our computations, we denote
\begin{align}\label{NOLL}
NOL(t)=&\norm{u(t)}_{H^{1}}^{2}+\max_{j\in\{1,\,2\}}\vert y_{j}(t)\vert^{2}+v^{2(k+1)}\left(\vert t\vert v+\ln{\left(\frac{1}{v^{2}}\right)}\right)^{n_{k}+1}e^{-2\sqrt{2}\vert t\vert v}\\ \nonumber &{+}\norm{u(t)}_{H^{1}_{x}}^{6}+\max_{j\in\{1,2\}}\vert y_{j}(t) \vert^{6}+\max_{j\in\{1,2\}}\vert\dot y_{j}(t)\vert v^{3}\left(\ln{\left(\frac{1}{v^{2}}\right)}+\vert t\vert v\right)e^{-2\sqrt{2}\vert t\vert v}\\ \nonumber &{+}\max_{j\in\{1,2\}}\vert y_{j}(t)\vert v^{4}\left(\ln{\left(\frac{1}{v^{2}}\right)}+\vert t\vert v\right)^{\max\{1,\eta_{k}\}}e^{-2\sqrt{2}\vert t\vert v},
\end{align}
where $\eta_{k}$ is the number denoted in Lemma \ref{y1y2}.
Also, from Theorem \ref{toobig}, Lemma \ref{y1y2} and identity \eqref{A(t)}, we deduce that
\begin{equation}
    \begin{bmatrix}
        \left\langle A(t,x),H^{'}_{0,1}\left(w_{k,v}(t,x)\right)\right\rangle\\
        \left\langle A(t,x),H^{'}_{0,1}\left(w_{k,v}(t,{-}x)\right)\right\rangle
    \end{bmatrix}=
    e^{-\sqrt{2}d(t)}
    \begin{bmatrix}
        -4\sqrt{2} & 4\sqrt{2}\\
        4\sqrt{2} & -4\sqrt{2}
    \end{bmatrix}
    \begin{bmatrix}
        y_{1}(t)\\
        y_{2}(t)
    \end{bmatrix}
    +Rest(t),
\end{equation}
where, if $v\ll 1,$ the real function $Rest(t)$ satisfies
for any $t\in\mathbb{R}$
\begin{multline}\label{Rest}
    e^{2\sqrt{2}\vert t\vert v}\left\vert Rest(t)\right\vert\lesssim_{k} v^{2(k+1)}\left(\vert t\vert v+\ln{\left(\frac{1}{v^{2}}\right)}\right)^{n_{k}+1}+\max_{j\in\{1,2\}} \vert y_{j}(t)\vert v^{4}\left(\vert t\vert v+\ln{\left(\frac{1}{v^{2}}\right)}\right)^{\max\{1,\eta_{k}\}}\\+\max_{j\in\{1,2\}} \vert\dot y_{j}(t)\vert v^{3}\left(\vert t\vert v+\ln{\left(\frac{1}{v^{2}}\right)}\right).
\end{multline}
From the orthogonality conditions \eqref{orthogonall}, Theorem \ref{toobig} and Lemma \ref{dots+}, we obtain the following estimate
\begin{align}\nonumber
    \left\langle \partial^{2}_{t}u(t,x),\, H^{'}_{0,1}\left(w_{k,v}(t,x)\right)\right\rangle=&\frac{\dot d(t)}{\sqrt{1-\frac{\dot d(t)^{2}}{4}}}\left\langle\partial_{t}u(t,x),\, H^{''}_{0,1}\left(w_{k,v}(t,x)\right)\right\rangle_{L^{2}_{x}}\\ \label{dt2uprod} &{+}O\left(\norm{\overrightarrow{u}(t)}_{H^{1}_{x}\times L^{2}_{x}}v^{2}\right).
\end{align}
Also, using integration by parts, identity ${-}\frac{d^{3}}{dx^{3}}H_{0,1}(x)+ U^{''}(H_{0,1}(x)) H^{'}_{0,1}(x)=0,$ Lemma \ref{interactt} and Cauchy-Schwarz inequality, we deduce that if $0<v\ll 1,$ then
\begin{multline}\label{lineardifu}
 \left\langle -\partial^{2}_{x}u(t)+ U^{''}\left(H_{0,1}\left(w_{k,v}(t,x)\right)-H_{0,1}\left(w_{k,v}(t,{-}x)\right)\right)u(t),\, H^{'}_{0,1}\left(w_{k,v}(t,x)\right)\right\rangle\\
 \begin{aligned}
 =&\left\langle u(t),\,\left[ U^{''}\left(H_{0,1}\left(w_{k,v}(t,x)\right)-H_{0,1}\left(w_{k,v}(t,{-}x)\right)\right)- U^{''}\left(H_{0,1}\left(w_{k,v}(t,x)\right)\right)\right] H^{'}_{0,1}\left(w_{k,v}(t,x)\right)\right\rangle\\&{+}O\left(v^{2}\norm{\overrightarrow{u}(t)}_{H^{1}_{x}\times L^{2}_{x}}\right)\\=&
 O\left(v^{2}\norm{\overrightarrow{u}(t)}_{H^{1}_{x}\times L^{2}_{x}}\right).
 \end{aligned}
\end{multline}
\par From now on, we denote any continuous function $f(t)$ as $O_{k}\left(NOL(t)\right),$ if and only if $f$ satisfies the following estimate
\begin{equation*}
\vert f(t) \vert\lesssim_{k} NOL(t). 
\end{equation*}
In conclusion, applying the scalar product of the equation \eqref{pdeu2} with $H^{'}_{0,1}\left(w_{k,v}(t,x)\right)$ and $H^{'}_{0,1}\left(w_{k,v}(t,{-}x)\right),$ we obtain using Lemma \ref{interactt} and estimates \eqref{dt2uprod}, \eqref{lineardifu} that
\begin{align}\nonumber
    \begin{bmatrix} 
    \norm{H^{'}_{0,1}}_{L^{2}_{x}}^{2} & O\left(d(t)e^{-\sqrt{2}d(t)}\right)\\
    O\left(d(t)e^{-\sqrt{2}d(t)}\right) & \norm{H^{'}_{0,1}}_{L^{2}_{x}}^{2}
    \end{bmatrix} 
    \begin{bmatrix}
    \ddot y_{1}(t)\\
    \ddot y_{2}(t)
    \end{bmatrix}
    =&e^{-\sqrt{2}d(t)}
    \begin{bmatrix}
        -4\sqrt{2} & 4\sqrt{2}\\
        4\sqrt{2} & -4\sqrt{2}
    \end{bmatrix}
    \begin{bmatrix}
        y_{1}(t)\\
        y_{2}(t)
    \end{bmatrix}
    \\ \nonumber &{+}\begin{bmatrix}
    O\left(v^{2}\norm{\overrightarrow{u}(t)}_{H^{1}_{x}\times L^{2}_{x}}\right)\\
     O\left(v^{2}\norm{\overrightarrow{u}(t)}_{H^{1}_{x}\times L^{2}_{x}}\right)
    \end{bmatrix} 
    \\ \nonumber &{-}\begin{bmatrix}
    \frac{\dot d(t)}{\left(1-\frac{\dot d(t)^{2}}{4}\right)^{\frac{1}{2}}}\left\langle\partial_{t}u(t,x),\, H^{''}_{0,1}\left(w_{k,v}(t,x)\right)\right\rangle\\
    \frac{\dot d(t)}{\left(1-\frac{\dot d(t)^{2}}{4}\right)^{\frac{1}{2}}}\left\langle\partial_{t}u(t,x),\, H^{''}_{0,1}\left(w_{k,v}(t,{-}x)\right)\right\rangle \end{bmatrix}
    \\ \label{odeyj}&{+}
    \begin{bmatrix}
    O_{k}\left(NOL(t)\right)\\
    O_{k}\left(NOL(t)\right)
    \end{bmatrix}.
\end{align}
\textbf{Step 2.}(Refined ordinary differential system.)
Motivated by equation \eqref{odeyj}, for $j\in\{1,\,2\}$ we define the functions
\begin{equation*}
    c_{j}(t)=y_{j}(t)-y_{j}(T_{0,k})+2\sqrt{2}\int_{T_{0,k}}^{t}\frac{\dot d(s)}{\left(1-\frac{\dot d(s)^{2}}{4}\right)^{\frac{1}{2}}}\left\langle u(s),\, H^{''}_{0,1}\left(w_{k,v}(s,(-1)^{j+1}x)\right)\right\rangle\,ds.
\end{equation*}
Clearly, we can verify using \eqref{decayd}, Lemma \ref{dots+} and Cauchy-Schwarz inequality that
\begin{align*}
    \dot c_{j}(t)=&\dot y_{j}(t)+\frac{2\sqrt{2}\dot d(t)}{\left(1-\frac{\dot d(t)^{2}}{4}\right)^{\frac{1}{2}}}\left\langle u(t,x),\,H^{''}_{0,1}\left(w_{k,v}(t,(-1)^{j+1}x)\right)\right\rangle,\\
    \ddot c_{j}(t)=&\ddot y_{j}(t)+\frac{2\sqrt{2}\dot d(t)}{\left(1-\frac{\dot d(t)^{2}}{4}\right)^{\frac{1}{2}}}\left\langle \partial_{t} u(t,x),\,\ddot H_{0,1}\left(w_{k,v}(t,(-1)^{j+1}x)\right)\right\rangle+O\left(v^{2}\norm{u(t)}_{H^{1}_{x}}\right).
\end{align*}
In conclusion, from the ordinary differential system of equations \eqref{odeyj} we deduce that
\begin{align*}
    \frac{d}{dt}\begin{bmatrix}
    y_{1}(t)\\
    y_{2}(t)\\
    \dot c_{1}(t)\\
    \dot c_{2}(t)
    \end{bmatrix}=&
    \begin{bmatrix}
    0 & 0 & 1 & 0\\
    0 & 0 & 0 & 1\\
    -16e^{-\sqrt{2}d(t)} & 16e^{-\sqrt{2}d(t)} & 0 & 0\\
    16e^{-\sqrt{2}d(t)} & -16e^{-\sqrt{2}d(t)} & 0 & 0
    \end{bmatrix}
    \begin{bmatrix}
    y_{1}(t)\\
    y_{2}(t)\\
    \dot c_{1}(t)\\
    \dot c_{2}(t)
    \end{bmatrix}
    \\&{+}\begin{bmatrix}
    O(v\norm{u(t)}_{H^{1}_{x}})\\
    O(v\norm{u(t)}_{H^{1}_{x}})\\
    O_{k}(NOL(t))+O\left(v^{2}\norm{\overrightarrow{u}(t)}_{H^{1}_{x}\times L^{2}_{x}}\right)\\
    O_{k}(NOL(t))+O\left(v^{2}\norm{\overrightarrow{u}(t)}_{H^{1}_{x}\times L^{2}_{x}}\right)
    \end{bmatrix}.
\end{align*}
Actually, using the following change of variables $e_{1}(t)=y_{1}(t)-y_{2}(t),\,e_{2}(t)=y_{1}(t)+y_{2}(t),\,\xi_{1}(t)=c_{1}(t)-c_{2}(t)$ and $\xi_{2}(t)=c_{1}(t)+c_{2}(t),$ we obtain from the ordinary differential system of equations above that
\begin{equation}\label{Trueode}
    \frac{d}{dt}\begin{bmatrix}
    e_{1}(t)\\
    e_{2}(t)\\
    \dot \xi_{1}(t)\\
    \dot \xi_{2}(t)
    \end{bmatrix}=
    \begin{bmatrix}
    0 & 0 & 1 & 0\\
    0 & 0 & 0 & 1\\
    -32e^{-\sqrt{2}d(t)} & 0 & 0 & 0\\
   0 & 0 & 0 & 0
    \end{bmatrix}
    \begin{bmatrix}
    e_{1}(t)\\
    e_{2}(t)\\
    \dot \xi_{1}(t)\\
    \dot \xi_{2}(t)
    \end{bmatrix}
    +\begin{bmatrix}
    O(v\norm{u(t)}_{H^{1}_{x}})\\
    O(v\norm{u(t)}_{H^{1}_{x}})\\
     O_{k}(NOL(t))+O\left(v^{2}\norm{\overrightarrow{u}(t)}_{H^{1}_{x}\times L^{2}_{x}}\right)\\
     O_{k}(NOL(t))+O\left(v^{2}\norm{\overrightarrow{u}(t)}_{H^{1}_{x}\times L^{2}_{x}}\right)
    \end{bmatrix}.
\end{equation}
To simplify our notation, we denote
\begin{equation}
    M(t)=
    \begin{bmatrix}
    0 & 0 & 1 & 0\\
    0 & 0 & 0 & 1\\
    -32e^{-\sqrt{2}d(t)} & 0 & 0 & 0\\
   0 & 0 & 0 & 0
    \end{bmatrix}.
\end{equation}
It is not difficult to verify that all the solutions of linear ordinary differential equation
\begin{equation*}
    \dot L(t)=M(t)L(t) \text{ for $L(t)\in \mathbb{R}^{4},$}
\end{equation*}
are the linear space generated by the following functions
\begin{gather*}
    L_{1}(t)=
    \begin{bmatrix}
    \tanh{(\sqrt{2}vt)}\\
    0\\
    \sqrt{2}v\sech{(\sqrt{2}vt)}^{2}\\
    0
    \end{bmatrix},\,
    L_{2}(t)=
    \begin{bmatrix}
    \sqrt{2}vt\tanh{(\sqrt{2}vt)}-1\\
    0\\
    2v^{2}t\sech{(\sqrt{2}vt)}^{2}+\sqrt{2}v\tanh{(\sqrt{2}vt)}\\
    0
    \end{bmatrix},\\
    L_{3}(t)=\begin{bmatrix}
    0\\
    1\\
    0\\
    0
    \end{bmatrix},\,
    L_{4}(t)=\begin{bmatrix}
    0\\
    t\\
    0\\
    1
    \end{bmatrix}.
\end{gather*}
Also, by elementary computation, we can verify for any $t\in\mathbb{R}$ that
\begin{equation}\label{det}
    \det{\left[L_{1}(t),L_{2}(t),L_{3}(t),L_{4}(t)\right]}=-\sqrt{2}v.
\end{equation}
In conclusion, using the variation of parameters technique, we can write any $C^{1}$ solution of \eqref{Trueode} as
$L(t)=\sum_{i=1}^{4}a_{i}(t)L_{i}(t),$ such that $a_{i}(t)\in C^{1}(\mathbb{R})$ for all $1\leq i\leq 4$ and
\begin{multline}\label{linsys}
    \begin{bmatrix}
    \tanh{(\sqrt{2}vt)} &  \sqrt{2}vt\tanh{(\sqrt{2}vt)}-1 & 0 & 0\\
    0 & 0 & 1 & t\\
    \sqrt{2}v\sech{(\sqrt{2}vt)}^{2} & 2v^{2}t\sech{(\sqrt{2}vt)}^{2}+\sqrt{2}v\tanh{(\sqrt{2}vt)} & 0 & 0\\
    0 & 0 & 0 & 1
    \end{bmatrix}
    \begin{bmatrix}
    \dot a_{1}(t)\\
    \dot a_{2}(t)\\
    \dot a_{3}(t)\\
    \dot a_{4}(t)
    \end{bmatrix}\\
    =\begin{bmatrix}
    O(v\norm{u(t)}_{H^{1}_{x}})\\
    O(v\norm{u(t)}_{H^{1}_{x}})\\
     O_{k}(NOL(t))+O\left(v^{2}\norm{\overrightarrow{u}(t)}_{H^{1}_{x}\times L^{2}_{x}}\right)\\
     O_{k}(NOL(t))+O\left(v^{2}\norm{\overrightarrow{u}(t)}_{H^{1}_{x}\times L^{2}_{x}}\right)
    \end{bmatrix},
\end{multline}
with
\begin{multline}\label{initialcond}
    \begin{bmatrix}
    \tanh{(\sqrt{2}vT_{0,k})} &  \sqrt{2}vT_{0,k}\tanh{(\sqrt{2}vT_{0,k})}-1 & 0 & 0\\
    0 & 0 & 1 & T_{0,k}\\
    \sqrt{2}v\sech{(\sqrt{2}vT_{0,k})}^{2} & 2v^{2}t\sech{(\sqrt{2}vT_{0,k})}^{2}+\sqrt{2}v\tanh{(\sqrt{2}vT_{0,k})} & 0 & 0\\
    0 & 0 & 0 & 1
    \end{bmatrix}
    \begin{bmatrix}
     a_{1}(T_{0,k})\\
     a_{2}(T_{0,k})\\
     a_{3}(T_{0,k})\\
     a_{4}(T_{0,k})
    \end{bmatrix}\\
    =\begin{bmatrix}
     y_{1}(T_{0,k})-y_{2}(T_{0,k})\\
     y_{1}(T_{0,k})+y_{1}(T_{0,k})\\
     \dot c_{1}(T_{0,k})\\
     \dot c_{2}(T_{0,k})
    \end{bmatrix}.
\end{multline}
\textbf{Step $3.$}(Estimate of $\norm{\overrightarrow{u}(t)}_{H^{1}_{x}\times L^{2}_{x}}.$) 
From now on, for $C_{1}>1,\,C_{2}>0$ being fixed numbers to be chosen later, we consider the following set
\begin{equation*}
    B_{C_{1},C_{2}}=\left\{t \in \mathbb{R}\big\vert \, \,\max_{j\in\{1,\,2\}}\vert y_{j}(t)\vert v^{2}+\vert \dot y_{j}(t) \vert v\leq C_{1} v^{2(k+1)}\ln{\left(\frac{1}{v}\right)}^{n_{k}+3}\exp\left(\frac{C_{2} v \vert t-T_{0,k}\vert}{\ln{\left(\frac{1}{v}\right)}}\right)\right\}.
\end{equation*}
We also consider the following set
\begin{equation*}
    D_{u,v}=\left\{t\in\mathbb{R}\big\vert\, \norm{\overrightarrow{u}(t)}_{H^{1}_{x}\times L^{2}_{x}}<v^{2}\right\}.
\end{equation*}
\par First, if $v^{2}\vert y(T_{0,k})\vert+v\vert\dot y(T_{0,k})\vert< v^{3k}$ and $v\ll 1,$ then $T_{0,k}\in B_{C_{1},C_{2}}\cap D_{u,v}.$ Indeed, this happens when
\begin{equation*}
    \norm{(\varphi_{k,v}(T_{0,k}),\partial_{t}\varphi_{k,v}(T_{0,k}))-(\phi(T_{0,k}),\partial_{t}\phi(T_{0,k})))}_{H^{1}_{x}\times L^{2}_{x}}<v^{4k}, 
\end{equation*}
because since $u(t,x)$ satisfies the orthogonality conditions \eqref{orthogonall}, we can verify using Lemma \ref{interactt} that
\begin{equation}\label{ygro}
    \norm{\varphi_{k,v}(T_{0,k})-\phi(T_{0,k})}_{H^{1}_{x}}^{2}\cong \max_{j\in\{1,2\}} y_{j}(T_{0,k})^{2}+\norm{u\left(T_{0,k}\right)}_{H^{1}_{x}}^{2}.
\end{equation}
By a similar reasoning but using now Lemma \ref{dots+} and estimate \eqref{ygro}, we can verify that if $0<v\ll 1,$ then
\begin{equation}\label{ygr02}
     \max_{j\in\{1,2\}} \dot y_{j}(T_{0,k})^{2}+\norm{\partial_{t}u\left(T_{0,k}\right)}_{L^{2}_{x}}^{2}\lesssim\norm{\left(\varphi_{k,v}(T_{0,k}),\partial_{t}\varphi_{k,}(T_{0,k})\right)-\left(\phi(T_{0,k}),\partial_{t}\phi(T_{0,k})\right)}_{H^{1}_{x}\times L^{2}_{x}}^{2},
\end{equation}
where $T_{0,k}$ satisfies the hypothesis of Theorem \ref{energyE}, for more details see Appendix $B$ in \cite{first}.  Also, for any $\theta\in(0,1),$ if $v\ll 1,$ then while
\begin{equation*}
\vert t-T_{0,k}\vert<\frac{\ln{\left(\frac{1}{v}\right)}^{2-\theta}}{v},
\end{equation*}
and $t \in B_{C_{1},C_{2}}\cap D_{u,v},$ we can verify the following estimate
\begin{equation*}
    \max_{j\in\{1,2\}} v^{2}\vert y_{j}(t)\vert+v\vert\dot y_{j}(t)\vert< v^{2k+1}\ln{\left(\frac{1}{v}\right)}^{n_{k}},
\end{equation*}
from which with estimate \eqref{odeyj}, the definition of $NOL(t)$ at \eqref{NOLL}, the definition of $D_{u,v}$ and the assumption of $k\geq 2,$ we obtain that
 \begin{equation*}
  \max_{j\in\{1,2\}} \vert \ddot y_{j}(t) \vert\lesssim_{k} v^{2k}\ln{\left(\frac{1}{v}\right)}^{n_{k}}+v\norm{\overrightarrow{u}(t)}_{H^{1}_{x}\times L^{2}_{x}}+ \norm{\overrightarrow{u}(t)}_{H^{1}_{x}\times L^{2}_{x}}^{2}.
\end{equation*}
 In conclusion, if $v\ll 1,$ from Theorem \ref{energyL}, we deduce that the functional $L(t)$ defined in last section satisfies, for a constant $ C_{0}$ and a parameter $C(k)$ depending only on $k,$ the  estimates
\begin{align*}
    \vert \dot L(t)\vert\lesssim & v\max_{j\in\{1,2\}}\vert \ddot y_{j}(t)\vert\norm{\overrightarrow{u}(t)}_{H^{1}_{x}\times L^{2}_{x}}+\norm{\overrightarrow{u}(t)}_{H^{1}_{x}\times L^{2}_{x}}^{3}
   \\&{+}C(k)\norm{\overrightarrow{u}(t)}_{H^{1}_{x}\times L^{2}_{x}}v^{2k+1}\ln{\left(\frac{1}{v}\right)}^{n_{k}}+\norm{\overrightarrow{u}(t)}_{H^{1}_{x}\times L^{2}_{x}}^{2}\frac{v}{\ln{\left(\frac{1}{v^{2}}\right)}},\\
    C_{0}\norm{\overrightarrow{u}(t)}_{H^{1}_{x}(\mathbb{R})\times L^{2}_{x}(\mathbb{R})}^{2}\leq & L(t)+C(k)v^{4k}\ln{\left(\frac{1}{v}\right)}^{2n_{k}}.
\end{align*}
Therefore, from the ordinary differential system of equations defined in \eqref{odeyj}, we conclude for $v\ll 1$ that if $t \in B_{C_{1},C_{2}}\cap D_{u,v}$ and
\begin{equation}\label{t0}
\vert t-T_{0,k}\vert<\frac{\ln{\left(\frac{1}{v}\right)}^{2-\theta}}{v},
\end{equation}
then there exists a constant $C(k)>0$ depending only on $k$ satisfying
\begin{equation*}
    \vert \dot L(t)\vert\lesssim 
C(k)\norm{\overrightarrow{u}(t)}_{H^{1}_{x}\times L^{2}_{x}}v^{2k+1}\ln{\left(\frac{1}{v}\right)}^{n_{k}}+\norm{\overrightarrow{u}(t)}_{H^{1}_{x}\times L^{2}_{x}}^{2}\frac{v}{\ln{\left(\frac{1}{v^{2}}\right)}}.
\end{equation*}
Therefore, by a similar argument to the proof of Theorem $4.5$ in \cite{first}, we can verify from Theorem \ref{energyL} and the Gronwall Lemma applied on $L(t)$ that there exists a constant $K>1$ non depending on $k$ and $v$ such that if $t$ satisfies condition \eqref{t0} and $t\in B_{C_{1},C_{2}}\cap D_{u,v}$, then we have the following estimate 
\begin{equation}\label{unorm1}
    \norm{(u(t),\partial_{t}u(t))}_{H^{1}_{x}\times L^{2}_{x}}\lesssim_{k}\max\left(\norm{\overrightarrow{u}\left(T_{0,k}\right)}_{H^{1}_{x}\times L^{2}_{x}},v^{2k}\ln{\left(\frac{1}{v}\right)}^{n_{k}+1}\right)\exp\left(\frac{K \vert t-T_{0,k}\vert v}{\ln{\left(\frac{1}{v}\right)}}\right).
\end{equation}
In conclusion, if $v\ll 1,$ $t\in B_{C_{1},C_{2}}$ and $t$ satisfies \eqref{t0}, then $t\in D_{u,v}$ and
\eqref{unorm1} is true.\\
\textbf{Step 4.}(Estimate of $y_{1}(t),\,y_{2}(t).$)
 Next, we are going to use the estimate \eqref{unorm1} in the ordinary differential system of equations \eqref{Trueode} to estimate the evolution of $y_{1}(t)$ and $y_{2}(t)$ while $t \in B_{C_{1},C_{2}}$ and $t$ satisfies condition \eqref{t0}.  From \eqref{NOLL}, we have that if $t \in B_{C_{1},C_{2}},$ $t$ satisfies condition \eqref{t0} and $0<v\ll 1,$ then
\begin{equation}\label{nloestt}
    NOL(t)\ll v^{2}\max\left(\norm{\overrightarrow{u}\left(T_{0,k}\right)}_{H^{1}_{x}\times L^{2}_{x}},v^{2k}\ln{\left(\frac{1}{v}\right)}^{n_{k}+1}\right)\exp\left(\frac{K \vert t-T_{0,k}\vert v}{\ln{\left(\frac{1}{v}\right)}}\right).
\end{equation}
In conclusion, from the Cauchy problem \eqref{IVP1} satisfied by $\phi,$ identity \eqref{det} and estimates \eqref{ygro}, \eqref{ygr02}, and \eqref{nloestt}, we deduce from the linear system \eqref{linsys} the following estimates
\begin{align*}
    \vert \dot a_{1}(t)\vert\lesssim_{k} &  v^{2k+1}\left[\vert t\vert v+1\right]\ln{\left(\frac{1}{v}\right)}^{n_{k}+1}\exp\left(K\frac{v\vert t-T_{0,k}\vert}{\ln{\left(\frac{1}{v}\right)}}\right),\\
    \vert \dot a_{2}(t)\vert\lesssim_{k} & v^{2k+1}\ln{\left(\frac{1}{v}\right)}^{n_{k}+1}\exp\left(K\frac{v\vert t-T_{0,k}\vert}{\ln{\left(\frac{1}{v}\right)}}\right),\\
    \vert \dot a_{3}(t)\vert\lesssim_{k} & v^{2k+1}\left[\vert t\vert v+1\right]\ln{\left(\frac{1}{v}\right)}^{n_{k}+1}\exp\left(K\frac{v\vert t-T_{0,k}\vert}{\ln{\left(\frac{1}{v}\right)}}\right),\\
    \vert \dot a_{4}(t)\vert\lesssim_{k} & v^{2k+2}\ln{\left(\frac{1}{v}\right)}^{n_{k}+1}\exp\left(K\frac{v\vert t-T_{0,k}\vert}{\ln{\left(\frac{1}{v}\right)}}\right).
\end{align*}
In conclusion, using the initial condition \eqref{initialcond}, we deduce from the fact that $T_{0,k}$ is in $B_{C_{1},C_{2}},$the Fundamental Theorem of Calculus and the elementary estimate
\begin{equation*}
    \vert t\vert v<\ln{\left(\frac{1}{v}\right)}\exp\left(\frac{v\vert t\vert }{\ln{\left(\frac{1}{v}\right)}}\right),
\end{equation*}
that if  $\{\theta t+(1-\theta)T_{0,k}\vert\, 0<\theta<1\} \subset B_{C_{1},c_{2}}$ and $t$ satisfies \eqref{t0}, then
\begin{align*}
   \vert a_{1}(t)\vert+\vert a_{3}(t)\vert\lesssim_{k} & v^{2k}\ln{\left(\frac{1}{v}\right)}^{n_{k}+3}\exp\left(\frac{(K+1) \vert t-T_{0,k}\vert v}{\ln{\left(\frac{1}{v}\right)}}\right),\\
  v \vert a_{2}(t)\vert +\vert a_{4}(t)\vert \lesssim_{k} & v^{2k+1}\ln{\left(\frac{1}{v}\right)}^{n_{k}+2}\exp\left(\frac{K \vert t-T_{0,k}\vert v}{\ln{\left(\frac{1}{v}\right)}}\right).
\end{align*}
In conclusion from the ordinary differential system of equations \eqref{Trueode} satisfied by $e_{j}(t)$ for $j\in\{1,2,3,4\},$ the fact that $e_{1}(t)=y_{1}(t)-y_{2}(t),\,e_{2}(t)=y_{1}(t)+y_{2}(t)$ and $\xi_{1}(t)=c_{1}(t)-c_{2}(t),\,\xi_{2}(t)=c_{1}(t)+c_{2}(t),$ we can verify by triangle inequality and the identity 
\begin{equation*}
    \begin{bmatrix}
     e_{1}(t)\\
     e_{2}(t)\\
     e_{3}(t)\\
     e_{4}(t)
    \end{bmatrix}
    =\sum_{j=1}^{4}a_{j}L_{j}(t)
\end{equation*}
the existence of $C_{1}(k)>0$ depending on $k$ such that for $C_{2}=K+2$ and $v\ll 1$ we have that if
\begin{equation*}
    \vert t-T_{0,k}\vert<\frac{\ln{\left(\frac{1}{v}\right)}^{2-\theta}}{v},
\end{equation*}
then $t\in B_{C_{1}(k),C_{2}}.$  
\end{proof}
\begin{remark}\label{Important}
For any constants $\theta,\gamma \in (0,1),$ obviously
\begin{equation*}
    \lim_{v\to^{+}0}v^{\gamma}\exp\left(\ln{\left(\frac{1}{v}\right)}^{\theta}\right)=0.
\end{equation*}
In conclusion, for fixed $k \in \mathbb{N}$ large and $0<\theta<\frac{1}{4}$, we can deduce from Theorem \ref{energyE} that there is a $\Delta_{k,\theta}>0$ such that if $0<v<\Delta_{k,\theta},$ then
\begin{equation*}
     \norm{(\phi(t,x),\partial_{t}\phi(t,x))-(\phi_{k}(v,t,x),\partial_{t}\phi_{k}(v,t,x))}_{H^{1}_{x}\times L^{2}_{x}}<v^{2k-\frac{1}{2}},
\end{equation*}
for all $t$ satisfying
\begin{equation*}
    \vert t-T_{0,k}\vert<\frac{\ln{\left(\frac{1}{v}\right)}^{2-\theta}}{v}. 
\end{equation*}
\end{remark}

\section{Proof of Theorem \ref{orbittheo}}\label{kinkstravel}
The main objective of this section is to prove Theorem \ref{orbittheo}.
\begin{remark}\label{importance}
The importance of this theorem is to describe the dynamics of the two solitons before the collision instant, for all $t<0$ and $\vert t\vert \gg 1.$ More precisely, if two moving kinks are coming from an infinite distance with a sufficiently low speed $v$ satisfying $v\leq \delta(2k),$ then the inelasticity of the collision is going to be of order at most $O(v^{k})$ and the kinks will move away each one with the speed of size in modulus $v+O(v^{k})$ when $t$ goes to $-\infty.$  
\end{remark}
The proof of Theorem \ref{orbittheo} uses energy estimate techniques from the article \cite{stabtheory}. Furthermore, the demonstration of Theorem \ref{orbittheo}  is quite similar to the proof of Theorem $1$ of the article \cite{asympt} and also uses modulation techniques inspired by \cite{blowup} and \cite{asympt}. \par From now on, we consider
\begin{equation}\label{P+}
P_{+}\left(\phi(t),\partial_{t}\phi(t)\right)={-}\int_{0}^{+\infty}\partial_{t}\phi(t,x)\partial_{x}\phi(t,x)\,dx,
\end{equation}
and since the solution $\phi(t,x)$ is an odd function in the variable $x$ for all $t\in\mathbb{R},$ we have that
\begin{equation*}
    E(\phi)=2\left[\int_{0}^{+\infty}\frac{\partial_{x}\phi(t,x)^{2}+\partial_{t}\phi(t,x)^{2}}{2}+U(\phi(t,x))\,dx\right]=2E_{+}\left(\phi(t),\partial_{t}\phi(t)\right),
\end{equation*}
where
\begin{equation}\label{E+}
E_{+}\left(\phi(t),\partial_{t}\phi(t)\right)=\int_{0}^{+\infty}\frac{\partial_{x}\phi(t,x)^{2}+\partial_{t}\phi(t,x)^{2}}{2}+U(\phi(t,x))\,dx
\end{equation}
is a conserved quantity.
\subsection{Modulation techinques}\par First, similarly to \cite{asympt}, we consider for any $0<v<1,\,y\in\mathbb{R}$ the following function on $x\in\mathbb{R}$
\begin{equation}
    \overrightarrow{H_{0,1}}((v,y),x)=
    \begin{bmatrix}
     H_{0,1}\left(\frac{x-y}{\sqrt{1-v^{2}}}\right)\\
     \frac{{-}v}{\sqrt{1-v^{2}}} H^{'}_{0,1}\left(\frac{x-y}{\sqrt{1-v^{2}}}\right)
    \end{bmatrix},
\end{equation}
and $\overrightarrow{H_{-1,0}}((v,y),x)=-\overrightarrow{H_{0,1}}((v,y),-x),$ for all $x\in\mathbb{R}.$
\par Next, we consider the anti-symmetric map 
\begin{equation}\label{Jmap}
    J=\begin{bmatrix}
     0 & 1\\
     -1& 0
    \end{bmatrix},
\end{equation}
and based on \cite{asympt}, we consider for any $0<v<1$ and any $y\in\mathbb{R}$ the following functions, which were defined in subsection $2.3$ of \cite{asympt},
\begin{align}\label{Cc4}
C_{v,y}(x)=&\begin{bmatrix}
    \frac{1}{\sqrt{1-v^{2}}} H^{'}_{0,1}\left(\frac{x-y}{\sqrt{1-v^{2}}}\right)\\
    \frac{{-}v}{1-v^{2}} H^{''}_{0,1}\left(\frac{x-y}{\sqrt{1-v^{2}}}\right)
\end{bmatrix},\\ \label{D}
D_{v,y}(x)=&
    \begin{bmatrix}
        \frac{v}{1-v^{2}}\frac{x-y}{\sqrt{1-v^{2}}} H^{'}_{0,1}\left(\frac{x-y}{\sqrt{1-v^{2}}}\right)\\
        \frac{{-}1}{(1-v^{2})^{\frac{3}{2}}} H^{'}_{0,1}\left(\frac{x-y}{\sqrt{1-v^{2}}}\right)-\frac{v^{2}}{(1-v^{2})^{\frac{3}{2}}}\frac{x-y}{\sqrt{1-v^{2}}} H^{''}_{0,1}\left(\frac{x-y}{\sqrt{1-v^{2}}}\right)
    \end{bmatrix},
\end{align}
see also the article \cite{dichotomy}.
\par The following identity is going to be useful for our next results. 
\begin{lemma}\label{usefulidentity}
 For any $v\in (0,1),$ it holds 
  \begin{equation*}
      \left\langle \partial_{x}\overrightarrow{H_{0,1}}\left((v,0),x\right),J D_{0,v} \right\rangle={-}\left(1-v^{2}\right)^{{-}\frac{3}{2}}\norm{ H^{'}_{0,1}}_{L^{2}_{x}}^{2}. 
  \end{equation*} 
\end{lemma}
\begin{proof}
    See the proof of Lemma $2.4$ from the article \cite{asympt}. 
\end{proof}
\par Next, for any value $y_{0}\gg1,$ we are going to modulate any odd function $(\phi_{0},\phi_{1})$ close to $\overrightarrow{H_{-1,0}}((v,y_{0}),x)+\overrightarrow{H_{0,1}}((v,y_{0}),x)$ in the energy norm in terms of an orthogonal condition.

\begin{lemma}\label{newmodulation}
There exist $K>0$ and $\delta_{0},\delta_{1}\in(0,1)$ such that if $0<v<\delta_{1},\,y_{0}>\frac{1}{\delta_{1}},\, 0\leq\delta\leq \delta_{0}$ and $(\phi_{1}-H_{0,1}-H_{{-}1,0},\phi_{2})\in H^{1}_{x}(\mathbb{R})\times L^{2}_{x}(\mathbb{R})$ is an odd function satisfying 
\begin{equation}
\norm{(\phi_{1}(x),\phi_{2}(x))-\overrightarrow{H_{-1,0}}((v,y_{0}),x)-\overrightarrow{H_{0,1}}((v,y_{0}),x)}_{H^{1}_{x}\times L^{2}_{x}}\leq \delta v,   
\end{equation}
then there exists a unique $\hat{y}>1$ such that $\vert \hat{y}-y_{0} \vert\leq K\delta v$ 
and the function
\begin{equation*}
\overrightarrow{\kappa}(x)=(\phi_{1}(x),\phi_{2}(x))-\overrightarrow{H_{-1,0}}((v,\hat{y}),x)-\overrightarrow{H_{0,1}}((v,\hat{y}),x)
\end{equation*}
satisfies 
 \begin{equation}\label{smallness}
\norm{\overrightarrow{\kappa}}_{H^{1}_{x}\times L^{2}_{x}}\leq K \delta v,
    \end{equation}
    and  $
        \left \langle \overrightarrow{\kappa}(x), J\circ D_{v,\hat{y}}(x)\right\rangle=0.
    $
\end{lemma}
\begin{proof}[Proof of Lemma \ref{newmodulation}]
The proof is completely analogous to the proof of Lemma $2.1$ of the article \cite{asympt}.  
\end{proof}

\begin{corollary}\label{cmodulation}
  In the notation of Lemma \ref{newmodulation}, there exists a constant $C>1$ such that if $v\in(0,1)$ is small enough, then there exists at most one number $y\geq 2\ln{\frac{1}{v}}$ satisfying with the function
\begin{equation*}
\overrightarrow{\kappa_{0}}(x)=(\phi_{1}(x),\phi_{2}(x))-\overrightarrow{H_{-1,0}}((v,y),x)-\overrightarrow{H_{0,1}}((v,y),x)
\end{equation*}
the estimate $\norm{\overrightarrow{\kappa_{0}}}_{H^{1}_{x}\times L^{2}_{x}}\leq \min\{\delta_{0}v,\frac{K}{3C}\delta_{0} v\}$
and $
        \left \langle \overrightarrow{\kappa_{0}}(x), J\circ D_{v,y}(x)\right\rangle=0.
    $
\end{corollary}
\begin{proof}[Proof of Corollary \ref{cmodulation}.]
Let $y_{1},\,y_{2}$ two real numbers satisfying the results of Corollary \ref{cmodulation}. We consider the following functions
\begin{align*}
\overrightarrow{\kappa_{1}}(x)=&(\kappa_{1,0}(x),\kappa_{1,1}(x))=(\phi_{1}(x),\phi_{2}(x))-\overrightarrow{H_{-1,0}}((v,y_{1}),x)-\overrightarrow{H_{0,1}}((v,y_{1}),x),\\
\overrightarrow{\kappa_{2}}(x)=&(\kappa_{2,0}(x),\kappa_{2,1}(x))=(\phi_{0}(x),\phi_{1}(x))-\overrightarrow{H_{-1,0}}((v,y_{2}),x)-\overrightarrow{H_{0,1}}((v,y_{2}),x).
\end{align*}
Choosing $x=y_{1},$ we obtain the following identity
\begin{equation}\label{eqba}
    H_{0,1}(0)-H_{0,1}\left(\frac{y_{1}-y_{2}}{\sqrt{1-v^{2}}}\right)={-}H_{0,1}\left(\frac{{-}2y_{1}}{\sqrt{1-v^{2}}}\right)+H_{0,1}\left(\frac{{-}y_{1}-y_{2}}{\sqrt{1-v^{2}}}\right)+\kappa_{2,0}(y_{1})-\kappa_{1,0}(y_{1}).
\end{equation}
Since there exists a constant $c>0$ satisfying for any $f\in H^{1}_{x}(\mathbb{R})$ the inequality
\begin{equation*}
    \norm{f}_{L^{\infty}_{x}(\mathbb{R})}\leq c\norm{f}_{H^{1}_{x}},
\end{equation*}
we deduce from equation \eqref{eqba} and the hypotheses of Corollary \ref{cmodulation} that
\begin{equation*}
    \left\vert H_{0,1}(0)-H_{0,1}\left(\frac{y_{1}-y_{2}}{\sqrt{1-v^{2}}}\right)\right\vert\leq \frac{2cK}{3C}\delta_{0} v +\left\vert H_{0,1}\left(\frac{{-}2y_{1}}{\sqrt{1-v^{2}}}\right)\right\vert +\left\vert H_{0,1}\left(\frac{{-}y_{1}-y_{2}}{\sqrt{1-v^{2}}}\right)\right\vert, 
\end{equation*}
from which we deduce the following estimate
\begin{equation*}
     \left\vert H_{0,1}(0)-H_{0,1}\left(\frac{y_{1}-y_{2}}{\sqrt{1-v^{2}}}\right)\right\vert\leq \frac{2cK}{3C}\delta_{0} v+2v^{4}.
\end{equation*}
Consequently, since $H_{0,1}$ is an increasing function and $H^{'}_{0,1}(0)=\frac{1}{2},$ we obtain that if $\delta_{1}\ll 1$ and $0<v<\delta_{1},$ then
\begin{equation*}
    \vert y_{1}-y_{2}\vert\leq \frac{5Kc}{3C}\delta_{0} v.
\end{equation*}
Therefore, choosing $C=2c+1,$ from Lemma \ref{newmodulation}, we shall have $y_{1}=y_{2}$ if $v>0$ is small enough.  
\end{proof}
\par Finally, using Lemma \ref{newmodulation} and repeating the argument of the demonstration of Lemma $2.11$ in \cite{asympt}, we can verify the following result.
\begin{lemma}\label{dependentmod}
There exist $K>1,\,\delta_{0}>0$ and $\delta_{1}\in(0,1)$ such that if $0<\delta_{2}<\delta_{0},\,0<v<\delta_{1},\,y_{0}>\frac{7}{2}\ln{\frac{1}{v}}$ and the solution $(\phi(t,x),\partial_{t}\phi(t,x))$ of \eqref{nlww} satisfies for a $T>0$
\begin{equation}\label{hypopo}
    \sup_{t\in\left[0,T\right]}\inf_{y\in\mathbb{R}_{\geq y_{0}}}\norm{(\phi(t,x),\partial_{t}\phi(t,x))-\overrightarrow{H_{-1,0}}((v,y),x)-\overrightarrow{H_{0,1}}((v,y),x)}_{H^{1}_{x}\times L^{2}_{x}}\leq \delta_{2}v,
\end{equation} then there exist a real function $y_{1}:\left[0,T\right]\to\mathbb{R}_{\geq\frac{y_{0}}{2}}$ such that the solution
$(\phi(t),\partial_{t}\phi(t))$ satisfies for any $0\leq t\leq T:$
\begin{align}\label{normalest0}
   (\phi(t),\partial_{t}\phi(t))=\overrightarrow{H_{-1,0}}((v,y_{1}(t)),x)+\overrightarrow{H_{0,1}}((v,y_{1}(t)),x)+(\psi_{1}(t),\psi_{2}(t)),\\
\label{normalest}
\norm{(\psi_{1}(t),\psi_{2}(t))}_{H^{1}_{x}\times L^{2}_{x}}\leq K \delta_{2}v,
\end{align} 
    where $(\psi_{1}(t),\psi_{2}(t))\in H^{1}_{x}(\mathbb{R})\times L^{2}_{x}(\mathbb{R})$ and $y_{1}(t)$ satisfy the orthogonality condition of Lemma \ref{newmodulation}, and
     $y_{1}(t)$
is a functions of class $C^{1}$ satisfying the following inequality:    \begin{equation}\label{derivest}
        \left\vert\dot y_{1}(t)-v\right\vert\leq K\left[ \norm{(\psi_{1}(t),\psi_{2}(t))}_{H^{1}_{x}\times L^{2}_{x}}+e^{-2\sqrt{2}y_{1}(t)}\right].
    \end{equation}
\end{lemma}
\begin{proof}
First, from Lemma \ref{newmodulation} and the fact that  $\overrightarrow{\phi}\in C\left(\mathbb{R};H^{1}_{x}(\mathbb{R})\times L^{2}_{x}(\mathbb{R})\right)$, if $\delta_{1}$ is small enough, we can find a constant $K>0$ and a function $\hat{y}:\left[0,T\right]\to\left(3\ln{\frac{1}{v}},{+}\infty\right)$ such that for
\begin{equation}\label{m1c4}
    \overrightarrow{\kappa}(t,x)= (\phi(t,x),\partial_{t}\phi(t,x))-\overrightarrow{H_{-1,0}}((v,\hat{y}(t)),x)-\overrightarrow{H_{0,1}}((v,\hat{y}(t)),x),
\end{equation}
we have $\overrightarrow{\kappa}(t),\,\hat{y}(t)$ satisfying the orthogonality condition of Lemma \ref{newmodulation} and 
\begin{equation}\label{m0}
\norm{\overrightarrow{\kappa}(t)}_{H^{1}_{x}\times L^{2}_{x}}\leq K\delta_{2}v,\end{equation}
for all $0\leq t\leq T.$  
\par Next, we are going to construct a linear ordinary differential system of equations with solution $y_{1}(t)$ and we are going to verify that if $y_{1}(0)=\hat{y}(0),$ then $y_{1}(t)=\hat{y}(t),$ for all $t\in\left[0, T\right].$ \\
\textbf{Step 1.}(Construction of the ordinary differential equation satisfied by $y_{1}.$)
 \par The argument of the demonstration of the remaining part of Lemma \ref{dependentmod} is completely analogous to the proof of Lemma $2.11$ of \cite{asympt}. More precisely, similarly to Lemma $2.11$ of \cite{asympt}, we will construct an ordinary differential equation with solution $y_{1}(t),$ which, during their time of existence, preserves the following orthogonality conditions
\begin{equation}\label{ode}
       \left \langle (\psi_{1}(t,x),\psi_{2}(t,x)), JD_{v,y_{1}(t)}(x)\right\rangle=0,
\end{equation}
where $J$ is defined in \eqref{Jmap}, and we are going to verify that if $y_{1}(0)=\hat{y}(0),$ then $y_{1}(t)=\hat{y}(t)$ for all $0\leq t\leq T.$
From the global well-posedness of the partial differential \eqref{nlww} in the energy space, we have for any $T_{0}>0$ that $\phi(t,x)-H_{0,1}(x)-H_{{-}1,0}(x)\in C\left(\left[{-}T_{0},T_{0}\right],H^{1}_{x}(\mathbb{R})\right)$ and $\partial_{t}\phi(t,x)\in C\left(\left[{-}T_{0},T_{0}\right],L^{2}_{x}(\mathbb{R})\right).$ Therefore, if there exists a interval $\left[0,T_{1}\right]\subset \left[0,T \right]$ such that $ y_{1} \in C^{1}(\left[0,T_{1}\right])$ when restricted to this interval and
\begin{equation}\label{idpsi}
\left(\phi(t),\partial_{t}\phi(t)\right)=\overrightarrow{H_{{-}1,0}}\left(\left(v,y_{1}(t)\right),x\right)+\overrightarrow{H_{0,1}}\left(\left(v,y_{1}(t)\right),x\right)+(\psi_{1}(t),\psi_{2}(t)) \text{, for any $t\in\left[0,T_{1}\right],$}
\end{equation}
then $(\psi_{1}(t),\psi_{2}(t))=(\psi_{1}(t,x),\psi_{2}(t,x))$ satisfies, for any functions $h_{1},\,h_{2}\in\mathscr{S}(\mathbb{R}),$ the following identity
\begin{equation*}
   \frac{d}{dt}\left\langle (\psi_{1}(t,x),\psi_{2}(t,x)), (h_{1}(x),h_{2}(x))\right\rangle= \left\langle\partial_{t}(\psi_{1}(t,x),\psi_{2}(t,x)),(h_{1}(x),h_{2}(x))\right\rangle,
\end{equation*}
if $t\in\left[0,T_{1}\right].$ 
\par Consequently, if we derive the equation \eqref{ode} in time, we obtain the following linear ordinary differential equation satisfied by $y_{1}(t)$
\begin{equation}
    \label{odes2}
    \dot y_{1}(t)\left\langle (\psi_{1}(t,x),\psi_{2}(t,x)), J \partial_{y_{1}}D_{v,y_{1}(t)}(x) \right\rangle +\left\langle \partial_{t}(\psi_{1}(t,x),\psi_{2}(t,x)),J D_{v,y_{1}(t)}(x)\right\rangle=0.
\end{equation}
Clearly, since $x^{m} H^{'}_{0,1}(x)\in \mathscr{S}(\mathbb{R})$ for all $m\in\mathbb{N}\cup\{0\},$ we have that the functions $\omega_{1},\,\omega_{2}:\left[0,T\right]\times \left(1,{+}\infty\right)\to\mathbb{R}$ defined by
\begin{equation*}
\omega_{1}(t,y)=\left\langle (\psi_{1}(t,x),\psi_{2}(t,x)), J \partial_{y}D_{v,y}(x) \right\rangle,\,
\omega_{2}(t,y)=\left\langle \partial_{t}(\psi_{1}(t,x),\psi_{2}(t,x)),J D_{v,y}(x)\right\rangle
\end{equation*}
are continuous and, for any $t\in\left[0,T\right],$ $\omega_{1}(t,\cdot),\,\omega_{2}(t,\cdot):\left(1,{+}\infty\right)\to\mathbb{R}$ are smooth. 
\\
\textbf{Step 2.}(Partial differential equation satisfied by $\overrightarrow{\psi}.$)
First, we consider the following self-adjoint operator $\hess (y_{1}(t),x):H^{2}_{x}(\mathbb{R})\subset L^{2}_{x}(\mathbb{R})\to\mathbb{R},$ which satisfies, for all $t\in\left[0,T\right],$ 
\begin{equation}\label{hess}
    \hess(y_{1}(t),x)=\begin{bmatrix}
     {-}\partial^{2}_{x}+ U^{''}\left(H_{0,1}\left(\frac{x-y_{1}(t)}{\sqrt{1-v^{2}}}\right)-H_{0,1}\left(\frac{{-}x-y_{1}(t)}{\sqrt{1-v^{2}}}\right)\right) & 0\\
     0 & 1
    \end{bmatrix},
\end{equation}
and the self-adjoint operator $\hess_{1}(y_{1}(t),x):H^{2}_{x}(\mathbb{R})\subset L^{2}_{x}(\mathbb{R})\to\mathbb{R}$ denoted by
\begin{equation}\label{hess1}
    \hess_{1}(y_{1}(t),x)=\begin{bmatrix}
     {-}\partial^{2}_{x}+ U^{''}\left(H_{0,1}\left(\frac{x-y_{1}(t)}{\sqrt{1-v^{2}}}\right)\right) & 0\\
     0 & 1
    \end{bmatrix}.
\end{equation}
Next, we consider the following maps $Int: \mathbb{R}^{2}\to \mathbb{R}^{2}$ and $\mathcal{T}:\mathbb{R}^{2}\times H^{1}_{x}(\mathbb{R})\to \mathbb{R}^{2},$ which we denote by
\begin{align}\label{intf}
    \begin{split}
       Int(y,x)=\begin{bmatrix}
         0\\
        U^{'}\left({-}H_{0,1}\left(\frac{{-}x-y_{1}}{\sqrt{1-v^{2}}}\right)\right)+ U^{'}\left(H_{0,1}\left(\frac{x-y}{\sqrt{1-v^{2}}}\right)\right)
       \end{bmatrix}\\
       {-}\begin{bmatrix}
       0\\
        U^{'}\left(H_{0,1}\left(\frac{x-y}{\sqrt{1-v^{2}}}\right)-H_{0,1}\left(\frac{{-}x-y}{\sqrt{1-v^{2}}}\right)\right)
        \end{bmatrix}
    \end{split},\\
   \label{tff}
    \begin{split}
        \mathcal{T}(y,x,\psi)=
        \begin{bmatrix}
         0\\
         {-}\sum_{j=3}^{6}
        U^{(j)}\left(H_{0,1}\left(\frac{x-y}{\sqrt{1-v^{2}}}\right)-H_{0,1}\left(\frac{{-}x-y}{\sqrt{1-v^{2}}}\right)\right)\frac{\psi(x)^{j-1}}{(j-1)!}
        \end{bmatrix},
    \end{split}
\end{align}
for any $(y,x)\in\mathbb{R}^{2}$ and $\psi\in H^{1}_{x}(\mathbb{R}).$
Therefore, if $\left[0,T_{1}\right]\subset \left[0,T\right],\,y_{1}\in C^{1}\left(\left[0,T_{1}\right]\right)$ and $y_{1}\geq 1,\,0<v_{1}<1$ then, from the partial differential equation \eqref{nlww} and identity \eqref{idpsi}, we deduce that $(\psi_{1}(t,x),\psi_{2}(t,x))$ is a solution in the space $C\left(\left[0,T_{1}\right],H^{1}_{x}(\mathbb{R})\times L^{2}_{x}(\mathbb{R})\right)$ of the following partial differential equation
\begin{multline}\label{pdepsi}
    \partial_{t}(\psi_{1}(t,x),\psi_{2}(t,x))=(\dot y_{1}(t)-v)\left[C_{v,y_{1}(t)}(x)-C_{v,y_{1}(t)}(-x)\right]\\{+}J\hess(y_{1}(t),x)(\psi_{1}(t,x),\psi_{2}(t,x))+Int(y_{1}(t),x)+\mathcal{T}(y_{1}(t),x,\psi_{1}(t)),
\end{multline}
where $J$ is the antissymetric operator defined in \eqref{Jmap}. 
\par In the next step, we are going to assume the existence of $0\leq T_{1}\leq T$ such that $y_{1}$ is of class $C^{1}$ in the interval $\left[0,T_{1}\right],$ and $y_{1}\geq 1$ for any $t\in \left[0,T_{1}\right].$ Moreover, we will prove that when this condition is true, then $\vert \dot y_{1}(t)-v\vert $ is sufficiently small for all $t\in\left[0,T_{1}\right].$\\
\textbf{Step 3.}(Estimate of $\left\vert\dot y_{1}(t)-v\right\vert.$) Uniquely in this step, for any continuous non-negative function $f:[0,T_{1}]\times (0,1)\times (1,{+}\infty)\to \mathbb{R},$ we say that a function $g:[0,T_{1}]\times (0,1)\times (1,{+}\infty)\to \mathbb{R}$ is $O(f),$ if and only if, $g$ is a continuous function satisfying the following properties:
\begin{itemize}
\item there exists a constant $c>0$ such that $\vert g(t,v,y)\vert<c f(t,v,y)$ for all $(t,v,y)$ in $[0,T_{1}]\times (0,1)\times (1,{+}\infty),$
\item $g(t,\cdot):(0,1)\times (1,{+}\infty)\to \mathbb{R}$ is smooth for all $t\in\left[0,T_{1}\right].$
\end{itemize}
 \par We recall that $J,\,C_{v,y_{1}(t)}$ and $D_{v,y_{1}(t)}$ are defined, respectively, in \eqref{Jmap}, \eqref{Cc4} and \eqref{D}.
Using Lemma \ref{interactt}, we obtain that if $y_{1}(t)\geq 1$ and $v\in (0,1)$ is small enough, then
\begin{multline}\label{intermodul}
  \begin{aligned}
  \left\vert \left\langle C_{v,y_{1}(t)}(x),J\circ D_{v,y_{1}(t)}(-x) \right\rangle\right\vert
   +\left\vert \left\langle C_{v,y_{1}(t)}(x),J C_{v,y_{1}(t)}(-x) \right\rangle\right\vert &\\{+} \left\vert \left\langle D_{v,y_{1}(t)}(x),J D_{v,y_{1}(t)}(-x) \right\rangle\right\vert &\lesssim y_{1}(t)^{4}e^{-2\sqrt{2}y_{1}(t)}.
\end{aligned}
\end{multline}
\par Furthermore, using the partial differential equation \eqref{pdepsi} satisfied by $(\psi_{1}(t,x),\psi_{2}(t,x))$, we deduce for any $t\in\left[0,T_{1}\right]\subset \left[0,T\right]$ the following identity
\begin{multline}\label{utc}
\begin{aligned}
\left\langle \partial_{t}(\psi_{1}(t,x),\psi_{2}(t,x)),J D_{v,y_{1}(t)}(x)\right\rangle = &(\dot y_{1}(t)-v)\left\langle C_{v,y_{1}(t)}(x),J D_{v,y_{1}(t)}(x)\right\rangle
 \\&{-} (\dot y_{1}(t)-v)\left\langle C_{v,y_{1}(t)}({-}x),J D_{v,y_{1}(t)}(x)\right\rangle  
    \\&{+}\left\langle J\hess(y_{1}(t),x)(\psi_{1}(t,x),\psi_{2}(t,x)),J D_{v,y_{1}(t)}(x)
    \right\rangle\\&{+}\left\langle\mathcal{T}(y_{1}(t),x,\psi_{1}(t))+Int(y_{1}(t),x),JD_{v,y_{1}(t)}(x)\right\rangle.
\end{aligned}
\end{multline}
Moreover, from Lemma \ref{usefulidentity} and identity $J^{*}={-}J,$ we have \begin{equation}\label{biggest}
\left\langle J D_{v,y_{1}(t)}(x), C_{v,y_{1}(t)}(x)\right\rangle={-}\left\langle  D_{v,y_{1}(t)}(x), J C_{v,y_{1}(t)}(x)\right\rangle=\left(1-v^{2}\right)^{-\frac{3}{2}}\norm{ H^{'}_{0,1}}_{L^{2}_{x}}^{2}.
\end{equation}
Therefore, using equation \eqref{utc}, estimates \eqref{intermodul} and Lemma \ref{interactt}, we deduce the following estimate
\begin{align*}
    \left\langle\partial_{t}(\psi_{1}(t,x),\psi_{2}(t,x)),J D_{v,y_{1}(t)}(x)\right\rangle=  &\left(\dot y_{1}(t)-v\right)\left[\left(1-v^{2}\right)^{{-}\frac{3}{2}}\norm{H^{'}_{0,1}}_{L^{2}_{x}}^{2} +O\left(y_{1}(t)^{4}e^{{-}2\sqrt{2}y_{1}(t)}\right)\right]\\&{+}\left\langle J\hess(y_{1}(t),x)(\psi_{1}(t,x),\psi_{2}(t,x)),JD_{v,y_{1}(t)}\right\rangle\\
    &{+} \left\langle\mathcal{T}(y_{1}(t),x,\psi_{1}(t)),JD_{v,y_{1}(t)}(x)\right\rangle\\
    &{+}\left\langle Int(y_{1}(t),x),JD_{v,y_{1}(t)}(x)
    \right\rangle.
\end{align*}
\par Furthermore, since, for any $\zeta\in\mathbb{R},$ we have the following identity
\begin{multline*}
     U^{'}\left(H^{\zeta}_{0,1}(x)+H_{{-}1,0}(x)\right)- U^{'}\left(H^{\zeta}_{0,1}(x)\right)- U^{'}\left(H_{{-}1,0}(x)\right)\\={-}24 H_{{-}1,0}(x)H^{\zeta}_{0,1}(x)\left(H_{{-}1,0}(x)+H^{\zeta}_{0,1}(x)\right)+\sum_{j=1}^{4}
    \begin{pmatrix}
        5\\
        j
    \end{pmatrix}
        H_{{-}1,0}(x)^{j}H^{\zeta}_{0,1}(x)^{5-j},
\end{multline*}
we deduce from Lemma \ref{interactt} and the definition of function $Int$ that
$
    \norm{Int(y_{1}(t),x,\psi(t))}_{L^{2}_{x}}\lesssim e^{-2\sqrt{2}y_{1}(t)}.
$
Next, since $\norm{U^{(l)}}_{L^{\infty}\left[{-}1,1\right]}<{+}\infty$ for any $l\in\mathbb{N}\cup\{0\},$ we deduce using Lemma \ref{mulsoblemma} and the definition of function $\mathcal{T}$ that
\begin{equation*}
    \norm{\mathcal{T}(y_{1}(t),x,\psi_{1}(t))}_{L^{2}_{x}}\leq \norm{\mathcal{T}(y_{1}(t),x,\psi_{1}(t))}_{H^{1}_{x}}\lesssim \norm{\psi_{1}(t,x)}_{H^{1}_{x}}^{2}.
\end{equation*}
As a consequence,
\begin{align}\nonumber
    \left\langle\partial_{t}(\psi_{1}(t,x),\psi_{2}(t,x)),J D_{v,y_{1}(t)}(x)\right\rangle=\nonumber &(\dot y_{1}(t)-v)\left[\left(1-v^{2}\right)^{{-}\frac{3}{2}}\norm{ H^{'}_{0,1}}_{L^{2}_{x}}^{2}+O\left(y_{1}(t)^{4}e^{-2\sqrt{2}y_{1}(t)}\right)\right]\\&{+}
 \left\langle
    J\hess(y_{1}(t),x)(\psi_{1}(t,x),\psi_{2}(t,x)),J D_{v_{1}(t),y_{1}(t)}(x)
    \right\rangle\\ \label{almostlin1} &{+}O\left(e^{-2\sqrt{2}y_{1}(t)}+\norm{\overrightarrow{\psi}(t)}_{H^{1}_{x}\times L^{2}_{x}}^{2}\right),
\end{align}
for any $t\in\left[0,T_{1}\right].$
\par Furthermore, using identities \eqref{hess}, \eqref{hess1}, the formula of $D_{v,y}$ in \eqref{D} and Lemma \ref{interactt}, we can deduce the following estimate
\begin{equation*}
\norm{\left[\hess(y_{1}(t),x)-\hess_{1}(y_{1}(t),x)\right]D_{v,y_{1}(t)}(x)}_{L^{2}_{x}(\mathbb{R};\mathbb{R}^{2})}\lesssim e^{{-}2\sqrt{2}y_{1}(t)}, 
\end{equation*}
for all $t\in\left[0,T_{1}\right].$
Thus, after using integration by parts and Cauchy-Schwarz inequality, we deduce for all $t\in \left[0,T_{1}\right]$ that
\begin{equation*}
\left\vert\left\langle J\left[\hess(y_{1}(t),x)-\hess_{1}(y_{1}(t),x)\right]\overrightarrow{\psi}(t),J D_{v_{1}(t),y_{1}(t)}(x) \right\rangle\right\vert\lesssim \norm{\overrightarrow{\psi}(t)}_{H^{1}_{x}\times L^{2}_{x}}e^{-2\sqrt{2}y_{1}(t)}.
\end{equation*}
Consequently, since $\left\langle j(a):a \right\rangle=0$ for all $a\in\mathbb{R}^{2},$ we obtain that if $y_{1}$ is a function of class $C^{1}$ in the interval $\left[0,T_{1}\right]$ and $v\in(0,1)$ is small enough, then
\begin{multline}\label{almostlin20}
 \begin{aligned}   
    \left\langle\partial_{t}(\psi_{1}(t,x),\psi_{2}(t,x)),J D_{v,y_{1}(t)}(x)\right\rangle=
&\left(\dot y_{1}(t)- v\right)\left[{-}\frac{\norm{ H^{'}_{0,1}}_{L^{2}_{x}}^{2}}{\left(1-v^{2}\right)^{\frac{3}{2}}}+O\left(y_{1}(t)^{4}e^{-2\sqrt{2}y_{1}(t)}\right)\right]\\&{+}
 \left\langle
    J\hess_{1}(y_{1}(t),x)(\psi_{1}(t,x),\psi_{2}(t,x)),J D_{v,y_{1}(t)}(x)
    \right\rangle\\&{+}O\left(e^{-2\sqrt{2}y_{1}(t)}+\norm{\overrightarrow{\psi}(t)}_{H^{1}_{x}\times L^{2}_{x}}^{2}\right), 
\end{aligned}
\end{multline}
for any $t\in\left[0,T_{1}\right].$
\par Next, using \eqref{hess1}, it is not difficult to verify the following identity
\begin{equation*}
\hess_{1}(y_{1}(t),x)D_{v,y_{1}(t)}(x)-vJ\left[\partial_{x}D_{v,y_{1}(t)}(x)\right] =J C_{v,y_{1}(t)}(x),
\end{equation*}
see Lemma $2.4$ of \cite{asympt} for the proof. Consequently, we have for any $t\in\left[0,T_{1}\right]$ that
\begin{align*}
    \left\langle
    J\hess_{1}(y_{1}(t),x)(\psi_{1}(t,x),\psi_{2}(t,x)),J D_{v,y_{1}(t)}(x)
    \right\rangle=&{-}v\left\langle (\psi_{1}(t,x),\psi_{2}(t,x)),J\partial_{y_{1}}D_{v,y_{1}(t)}(x) \right\rangle\\&{+}\left\langle(\psi_{1}(t,x),\psi_{2}(t,x)),J C_{v,y_{1}(t)}(x)\right\rangle.
\end{align*}
In conclusion, estimate \eqref{almostlin20} and identity \eqref{odes2} imply that
\begin{multline}\label{appode}
    \left(  \dot y_{1}(t)-v\right)\left[\frac{{-}\norm{ H^{'}_{0,1}}_{L^{2}_{x}}^{2}}{(1-v^{2})^{\frac{3}{2}}}+O\left(\norm{(\psi_{1}(t),\psi_{2}(t))}_{H^{1}_{x}\times L^{2}_{x}}+y_{1}(t)^{4}e^{{-}2\sqrt{2}y_{1}(t)}\right)\right]\\=O\left(e^{{-}2\sqrt{2}y_{1}(t)}+\norm{(\psi_{1}(t),\psi_{2}(t))}_{H^{1}_{x}\times L^{2}_{x}}\right),
\end{multline}
 for all $t\in\left[0,T_{1}\right].$
 \\
\textbf{Step 4.}(Proof that $y_{1}\in C^{1}.$)
 Furthermore, the equations \eqref{odes2} and \eqref{utc} imply that $y_{1}$ shall satisfy the following ordinary differential equation
 \begin{multline}\label{odedey1}
     \left(\dot y_{1}(t)-v\right)\Bigg[\left\langle C_{v,y_{1}(t)}(x),J D_{v,y_{1}(t)}(x)\right\rangle-\left\langle C_{v,y_{1}(t)}({-}x),J D_{v,y_{1}(t)}(x)\right\rangle\\{+}\left\langle(\psi_{1}(t),\psi_{2}(t)),J\partial_{y_{1}}D_{v,y_{1}(t)}(x)\right\rangle\Bigg]
     \\
     \begin{aligned}
     =&{-}v\left\langle (\psi_{1}(t,x),\psi_{2}(t,x)),J\partial_{y_{1}}D_{v,y_{1}(t)}(x)\right\rangle
     \\ &{-}\left\langle J\hess(y_{1}(t),x)(\psi_{1}(t,x),\psi_{2}(t,x))+\mathcal{T}(y_{1}(t),x,\psi_{1}(t))+Int(y_{1}(t),x),JD_{v,y_{1}(t)}(x)\right\rangle,
     \end{aligned}
 \end{multline}
which is a first-order non-autonomous differential system of the form
\begin{equation*}
    \left(\dot y_{1}(t)-v\right)\alpha_{v}\left(t,y_{1}(t)\right)=\beta_{v}\left(t,y_{1}(t)\right),
\end{equation*}
where the functions the functions $\alpha_{v},\,\beta_{v}:\left[0,T\right]\times \mathbb{R}\to\mathbb{R}$ are continuous when $v\in (0,1).$ 
\par Moreover, from the hypotheses of Lemma \ref{dependentmod}, Lemma \ref{interactt} and identities \eqref{hess}, \eqref{intf}, \eqref{tff}, we can deduce for any $t\in \left[0,T\right]$ that the restrictions of $\alpha_{v}(t,\cdot)$ and $\beta_{v}(t,\cdot)$ in the set $\left(3\ln{\frac{1}{v}},{+}\infty\right)$ are locally Lipschitz when $v$ is small enough. 
\par Furthermore, from the first step, we have $y_{1}(0)=\hat{y}(0)>3\ln{\frac{1}{v}}$ which implies 
$y_{1}(0)^{4}e^{{-}2\sqrt{2}y_{1}(0)}<v^{3},$  if $v$ is small enough. Moreover,
we deduce from \eqref{m1c4} and \eqref{m0} that $\norm{(\psi_{1}(0),\psi_{2}(0))}_{H^{1}_{x}\times L^{2}_{x}}\leq K\delta_{2}v$ and we also have 
\begin{equation*}
    \alpha_{v}(0,y_{1}(0))=\frac{{-}\norm{ H^{'}_{0,1}}_{L^{2}_{x}}^{2}}{\left(1-v^{2}\right)^{\frac{3}{2}}}+O(v)>0,
\end{equation*}
because of the estimate \eqref{appode} when $v$ is small enough. 
\par Consequently, Picard-Lindelöf Theorem implies the existence of an interval $\left[0,T_{1}\right]\subset \left[0,T\right]$ such that $y_{1}:\left[0,T_{1}\right]\to\mathbb{R}_{>2\ln{\frac{1}{v}}}$ is a $C^{1}$ function and 
since $y_{1}$ satisfies \eqref{odes2}, we have for any $t\in\left[0,T_{1}\right]$ that
\begin{align}\label{mody1y1}
 \left\langle (\psi_{1}(t,x),\psi_{2}(t,x)),J D_{v,y_{1}(t)}(x) \right\rangle=&  \left\langle \overrightarrow{\psi}(0,x),J D_{v,y_{1}(0)}(x) \right\rangle 
= 0.
\end{align}
\par Furthermore, since $\hat{y}(t)\geq 3\ln{\frac{1}{v}},$ we can deduce from the continuity of function $y_{1},$ Lemma \ref{newmodulation} and Corollary \ref{cmodulation} the identity $y_{1}(t)=\hat{y}(t)$ for all $t\in \left[0,T_{1}\right].$ As a consequence, $y_{1}(t)\geq 3\ln{\frac{1}{v}}$ for all $t\in \left[0,T_{1}\right]$ and 
\begin{equation}\label{y1estt}
\norm{(\psi_{1}(t),\psi_{2}(t))}_{H^{1}_{x}\times L^{2}_{x}}=\norm{\overrightarrow{\phi}(t,x)-\overrightarrow{H_{-1,0}}((v,y_{1}(t)),x)-\overrightarrow{H_{0,1}}((v,y_{1}(t)),x)}_{H^{1}_{x}\times L^{2}_{x}}\leq K\delta_{2}v
\end{equation}
for all $t\in \left[0,T_{1}\right],$ because of estimate \eqref{m1c4} and identity \eqref{m0}. 
\par Therefore, using a bootstrap argument and estimate \eqref{appode}, we can conclude that the function $y_{1}$ is in $ C^{1}\left[0,T\right]$ and satisfies \eqref{mody1y1} for all $t\in \left[0,T\right].$ 
 Finally, estimate \eqref{derivest} is a direct consequence of \eqref{appode}, \eqref{y1estt} and the fact that $y_{1}\geq 3\ln{\frac{1}{v}}.$ 
\end{proof}
\subsection{Orbital stability of the parameter $y$}
In this subsection, we consider $\phi(t,x)$ as a solution of \eqref{nlww} having finite energy and with an initial data $(u_{1}(x),u_{2}(x))$ satisfying the hypotheses of Theorem \ref{orbittheo}. Moreover, if $v$ is small enough, from the local well-posedness of the partial differential equation \eqref{nlww} in the space of solutions with finite energy, we can deduce from Lemma \ref{newmodulation} the existence of a constant $C>0$ and a positive number $\epsilon$ such that
for all $t\in[0,\epsilon]$
\begin{equation*}
    (\phi(t,x),\partial_{t}\phi(t,x))=\overrightarrow{H_{-1,0}}((v,y(t)),x)+\overrightarrow{H_{0,1}}((v,y(t)),x)+(\psi_{1}(t,x),\psi_{2}(t,x)),
\end{equation*}
where $(\psi_{1}(t,x),\psi_{2}(t,x))$ is an odd function in $x,$ and $y(t),\, (\psi_{1}(t,x),\psi_{2}(t,x))$ satisfy the orthogonality conditions in Lemma \ref{newmodulation} and the following inequality
\begin{equation}\label{continous1}
     \vert y(t)-y_{0} \vert+\norm{(\psi_{1}(t,x),\psi_{2}(t,x))}_{H^{1}_{x}\times L^{2}_{x}}\leq 2C \norm{(u_{1},u_{2})}_{H^{1}_{x}\times L^{2}_{x}}.
\end{equation}
Finally, we are ready to start the proof of Theorem \ref{orbittheo}
\begin{remark}[Main argument]\label{restriction}
The main techniques of the demonstration of Theorem \ref{orbittheo}
are inspired by the proof of Theorem $1$ of \cite{asympt}. 
\par More precisely, recalling the functions $E_{+}$ and $P_{+}$ from \eqref{E+} and \eqref{P+}, we will analyze the function
\begin{equation}\label{M(t)}
M(\phi(t))=E_{+}(\phi(t))-vP_{+}(\phi(t)).
\end{equation}
\par First, from the local well-posedness of the partial differential equation \eqref{nlww} in the energy space, it is enough to verify Theorem \ref{orbittheo} to the case where $(u_{1}(x),u_{2}(x))$ is a smooth odd function because the estimate \eqref{globalorbit} and the density of smooth functions in Sobolev spaces would imply that \eqref{globalorbit} would be true for any $(u_{1}(x),u_{2}(x))\in H^{1}_{x}\times L^{2}_{x}$ satisfying the hypothesis of Theorem \ref{orbittheo}.
 
\par Since $P_{+}(t)$ is not necessarily a conserved quantity, $M(t)$ is not necessarily a constant function given any smooth initial initial data of $(\phi(0,x),\partial_{t}\phi(0,x))$ satisfying the hypotheses of Theorem \ref{globalorbit}.
\par However, $P_{+}(t)$ is a non-increasing function in time, more precisely, for smooth solutions $\phi(t,x)$ of \eqref{nlwwsmooth}, we can verify using integration by parts, from the fact that $\phi(t,x)$ is an odd function in $x$ for any $t\in \mathbb{R},$ the estimate
\begin{equation}
    \frac{d}{dt}\left[{-}\int_{0}^{+\infty}\partial_{t}\phi(t,x)\partial_{x}\phi(t,x)\,dx\right]=\frac{1}{2}\phi(t,0)^{2}\geq 0.
\end{equation}
\par In conclusion, since it was verified before that $E_{+}(t)$ is a conserved quantity, we have that
\begin{equation*}
    M(\phi(t))\leq M(\phi(0)) \text{ for any $t\geq 0,$}
\end{equation*}
and using Lemma \ref{newmodulation}, we are going to verify that $M(0)-M(t)$ satisfies a coercive inequality, from which we will deduce \eqref{globalorbit}.  
\end{remark}
\begin{proof}[Proof of Theorem \ref{orbittheo}]
From the observations in Remark \ref{restriction}, it is enough to prove Theorem \ref{orbittheo} for the case where $\overrightarrow{\psi}_{0}(x)$ is a smooth odd function. To simplify our proof, we separate the argument into different steps.\\
\textbf{Step 1.}(Local description of solution $\phi(t,x).$)
From the observation of inequality \eqref{continous1} and from the Lemma \ref{newmodulation}, we can verify the existence of an interval $[0,\epsilon]$ such that if $t\in [0,\epsilon],$ then
\begin{equation}\label{formm}
 (\phi(t,x),\partial_{t}\phi(t,x))=\overrightarrow{H_{-1,0}}((v,y(t)),x)+\overrightarrow{H_{0,1}}((v,y(t)),x)+(\psi_{1}(t,x),\psi_{2}(t,x)),
\end{equation}
with $v(t),\,y(t),\,(\psi_{1}(t,x),\psi_{2}(t,x))$ satisfying all the conditions of Lemma \ref{newmodulation}.\\
\textbf{Step 2.}(Estimate of $E_{+}\left(\phi(t),\partial_{t}\phi(t)\right)$ around the kinks.)
We recall the definition of $E_{+}\left(\phi(t),\partial_{t}\phi(t)\right)$ in \eqref{E+} given by
\begin{equation*}
E_{+}\left(\phi(t),\partial_{t}\phi(t)\right)=\int_{0}^{+\infty}\frac{\partial_{x}\phi(t,x)^{2}+\partial_{t}\phi(t,x)^{2}}{2}+U(\phi(t,x))\,dx.
\end{equation*}
Next, we substitute $\phi(t,x)$ and $\partial_{t}\phi(t,x)$ in the equation above by the formula of $(\phi(t,x),\partial_{t}\phi(t,x))$ in Step $1.$ 
 Using \eqref{le1}, \eqref{le2} and since $y(t)>1$ for $0\leq t\leq \epsilon,$ we obtain for all $x\geq 0$ that
\begin{equation}\label{lo1}
   \left\vert \frac{\partial^{l}}{\partial x^{l}}H_{-1,0}\left(\frac{x+y(t)}{\sqrt{1-v^{2}}}\right)\right\vert\lesssim_{l} (1-v^{2})^{{-}\frac{l}{2}}e^{-\sqrt{2}(y(t)+x)} \text{ for any $l\in\mathbb{N}\cup\{0\},$}
\end{equation}
from which we also deduce, using Lemma \ref{interactt}, the following estimate
\begin{equation}\label{lo2}
    \int_{\mathbb{R}} H^{'}_{0,1}\left(\frac{x-y(t)}{\sqrt{1-v^{2}}}\right) H^{'}_{-1,0}\left(\frac{x+y(t)}{\sqrt{1-v^{2}}}\right)\lesssim \left(1-v^{2}\right)^{\frac{1}{2}}y(t)e^{-2\sqrt{2}y(t)}.
\end{equation}
In addition, since  $\norm{U^{(l)}}_{L^{\infty}\left[{-}1,1\right]}<{+}\infty$ for any $l\in\mathbb{N},$ we can deduce using Lemma \ref{mulsoblemma} the following inequality
\begin{align*}
    \norm{U^{(l)}\left(H_{0,1}\left(\frac{x-y(t)}{\sqrt{1-v^{2}}}\right)+H_{-1,0}\left(\frac{x+y(t)}{\sqrt{1-v^{2}}}\right)\right)\psi_{1}(t,x)^{l}}_{H^{1}_{x}}\lesssim_{l}\norm{\psi_{1}(t,x)}_{H^{1}_{x}}^{l}.
\end{align*}
In conclusion, since \begin{align}\label{idd1}
    \phi(t,x)=& H_{0,1}\left(\frac{x-y(t)}{\sqrt{1-v^{2}}}\right)+H_{-1,0}\left(\frac{x+y(t)}{\sqrt{1-v^{2}}}\right)+\psi_{1}(t,x),\\ \label{idd2}
    \partial_{t}\phi(t,x)= &{-}\frac{v}{\sqrt{1-v^{2}}} H^{'}_{0,1}\left(\frac{x-y(t)}{\sqrt{1-v^{2}}}\right)+\frac{v}{\sqrt{1-v^{2}}} H^{'}_{-1,0}\left(\frac{x+y(t)}{\sqrt{1-v^{2}}}\right)+\psi_{2}(t,x),
\end{align}
we deduce from the formula \eqref{E+}, estimates \eqref{lo1}, \eqref{lo2} and Taylor's Expansion Theorem that
\begin{multline}\label{taylorE}
\begin{aligned}
E_{+}\left(\phi(t),\partial_{t}\phi(t)\right)=& \int_{0}^{+\infty} \frac{1+v^{2}}{2(1-v^{2})}  H^{'}_{0,1}\left(\frac{x-y(t)}{\sqrt{1-v^{2}}}\right)^{2}+U\left(H_{0,1}\left(\frac{x-y(t)}{\sqrt{1-v^{2}}}\right)\right)\,dx\\
    &{-}\frac{1}{\sqrt{1-v^{2}}}\int_{0}^{+\infty} v H^{'}_{0,1}\left(\frac{x-y(t)}{\sqrt{1-v^{2}}}\right)\psi_{2}(t,x)\,dx- H^{'}_{0,1}\left(\frac{x-y(t)}{\sqrt{1-v^{2}}}\right)\partial_{x}\psi_{1}(t,x)\\&{+}\int_{0}^{{+}\infty} U^{'}\left(H_{0,1}\left(\frac{x-y(t)}{\sqrt{1-v^{2}}}\right)\right)\psi_{1}(t,x)\,dx \\&{+}\frac{1}{2}\left[\int_{0}^{+\infty}\partial_{x}\psi_{1}(t,x)^{2}+ U^{''}\left(H_{0,1}\left(\frac{x-y(t)}{\sqrt{1-v^{2}}}\right)\right)\psi_{1}(t,x)^{2}+\psi_{2}(t,x)^{2}\right]\,dx\\&{+}O\left(\left(1-v^{2}\right)^{-\frac{1}{2}}y(t)e^{-2\sqrt{2}y(t)}\right)\\&{+}O\left(\norm{\overrightarrow{\psi}(t)}_{H^{1}_{x}\times L^{2}_{x}}e^{-\sqrt{2}y(t)}+\norm{\psi_{1}(t,x)}_{H^{1}_{x}(\mathbb{R})}^{3}\right),
\end{aligned}
\end{multline}
while $(\phi(t,x),\partial_{t}\phi(t,x))$ satisfies identities \eqref{idd1} and \eqref{idd2}. Moreover, from \eqref{idd1}, we can obtain from \eqref{taylorE}, while $(\phi(t),\partial_{t}\phi(t))$ satisfies \eqref{idd1} and \eqref{idd2}, that
\begin{multline}\label{E+2}
    \begin{aligned}
E_{+}\left(\phi(t),\partial_{t}\phi(t)\right)=& \int_{{-}\infty}^{{+}\infty} \frac{1+v^{2}}{2(1-v^{2})} H^{'}_{0,1}\left(\frac{x-y(t)}{\sqrt{1-v^{2}}}\right)^{2}+U\left(H_{0,1}\left(\frac{x-y(t)}{\sqrt{1-v^{2}}}\right)\right)\,dx\\
    &\frac{{-}1}{\sqrt{1-v^{2}}}\int_{{-}\infty}^{{+}\infty} v H^{'}_{0,1}\left(\frac{x-y(t)}{\sqrt{1-v^{2}}}\right)\psi_{2}(t,x)- H^{'}_{0,1}\left(\frac{x-y(t)}{\sqrt{1-v^{2}}}\right)\partial_{x}\psi_{1}(t,x)\\&{+}\int_{{-}\infty}^{{+}\infty} U^{'}\left(H_{0,1}\left(\frac{x-y(t)}{\sqrt{1-v^{2}}}\right)\right)\psi_{1}(t.x)\,dx \\&{+}\frac{1}{2}\left[\int_{0}^{+\infty}\partial_{x}\psi_{1}(t,x)^{2}+ U^{''}\left(H_{0,1}\left(\frac{x-y(t)}{\sqrt{1-v^{2}}}\right)\right)\psi_{1}(t,x)^{2}+\psi_{2}(t,x)^{2}\,dx\right]\\&{+}O\left(\left(1-v^{2}\right)^{-\frac{1}{2}}y(t)e^{-2\sqrt{2}y(t)}\right)\\&{+}O\left(\norm{\overrightarrow{\psi}(t)}_{H^{1}_{x}\times L^{2}_{x}}e^{-\sqrt{2}y(t)}+\norm{\psi_{1}(t,x)}_{H^{1}_{x}(\mathbb{R})}^{3}\right),
\end{aligned}
\end{multline}
\par We also recall the Bogomolny identity
$
    H^{'}_{0,1}(x)=\sqrt{2 U(H_{0,1}(x))},
$
from which we deduce with change of variables that
\begin{equation}
    \frac{1}{2}\int_{\mathbb{R}} H^{'}_{0,1}\left(\frac{x}{\sqrt{1-v^{2}}}\right)^{2}\,dx=\int_{\mathbb{R}}U\left(H_{0,1}\left(\frac{x}{\sqrt{1-v^{2}}}\right)\right)\,dx=\sqrt{1-v^{2}}\frac{\norm{ H^{'}_{0,1}}_{L^{2}_{x}}^{2}}{2}.
\end{equation}
\textbf{Step 3.}\big(Conclusion of the estimate of $E_{+}(t).$ \big)\\
Since $\overrightarrow{H_{0,1}}((v,y(t)),x))$ is defined by 
\begin{equation*}
   \overrightarrow{H_{0,1}}((v,y(t)),x)=
   \begin{bmatrix}
   H_{0,1}\left(\frac{x-y(t)}{\sqrt{1-v(t)^{2}}}\right)\\{-}\frac{v}{\sqrt{1-v^{2}}}
H^{'}_{0,1}\left(\frac{x-y(t)}{\sqrt{1-v^{2}}}\right)
   \end{bmatrix},
\end{equation*}
and we can verify by similar reasoning to \eqref{taylorE} the identity
\begin{equation*}
E\left(\overrightarrow{H_{0,1}}((v,y(t)),x)\right)=\int_{-\infty}^{+\infty} \frac{1+v^{2}}{2(1-v^{2})}  H^{'}_{0,1}\left(\frac{x-y(t)}{\sqrt{1-v^{2}}}\right)^{2}+U\left(H_{0,1}\left(\frac{x-y(t)}{\sqrt{1-v^{2}}}\right)\right)\,dx,
\end{equation*}
we conclude that
$
E\left(\overrightarrow{H_{0,1}}((v,y(t)),x)\right)=\frac{1}{\sqrt{1-v^{2}}}\norm{ H^{'}_{0,1}}_{L^{2}_{x}}^{2}.
$
In conclusion, using \eqref{E+2}, we obtain that
\begin{align*}
E_{+}\left(\phi(t),\partial_{t}\phi(t)\right)=&\frac{1}{\sqrt{1-v^{2}}}\norm{ H^{'}_{0,1}}_{L^{2}_{x}}^{2}
    -\int_{-\infty}^{+\infty} \frac{v}{\sqrt{1-v^{2}}} H^{'}_{0,1}\left(\frac{x-y(t)}{\sqrt{1-v^{2}}}\right)\psi_{2}(t,x)\,dx\\&{+}\int_{{-}\infty}^{{+}\infty}\frac{1}{\sqrt{1-v^{2}}} H^{'}_{0,1}\left(\frac{x-y(t)}{\sqrt{1-v^{2}}}\right)\partial_{x}\psi_{1}(t,x)\\&{+}\int_{{-}\infty}^{{+}\infty} U^{'}\left(H_{0,1}\left(\frac{x-y(t)}{\sqrt{1-v^{2}}}\right)\right)\psi_{1}(t,x)\,dx \\&{+}\frac{1}{2}\left[\int_{0}^{+\infty}\partial_{x}\psi_{1}(t,x)^{2}+ U^{''}\left(H_{0,1}\left(\frac{x-y(t)}{\sqrt{1-v^{2}}}\right)\right)\psi_{1}(t,x)^{2}+\psi_{2}(t,x)^{2}\right]\\
   &{+}O\left(\left(1-v^{2}\right)^{-\frac{1}{2}}y(t)e^{-2\sqrt{2}y(t)}+\norm{(\psi_{1}(t),\psi_{2}(t))}_{H^{1}_{x}\times L^{2}_{x}}e^{-\sqrt{2}y(t)}\right)\\&{+}O\left(\norm{\psi_{1}(t)}_{H^{1}_{x}(\mathbb{R})}^{3}\right),
\end{align*}
from this using integration by parts we conclude that
\begin{multline}\label{E+comp1}
\begin{aligned}
E_{+}\left(\phi(t),\partial_{t}\phi(t)\right)=&\frac{1}{\sqrt{1-v^{2}}}\norm{ H^{'}_{0,1}}_{L^{2}_{x}}^{2}+v\left\langle J\circ C_{v,y(t)},\overrightarrow{\psi(t)}\right\rangle\\ 
&{+}\frac{1}{2}\left[\int_{0}^{+\infty}\psi_{2}(t,x)^{2}+\partial_{x}\psi_{1}(t,x)^{2}+ U^{''}\left(H_{0,1}\left(\frac{x-y(t)}{\sqrt{1-v^{2}}}\right)\right)\psi_{1}(t,x)^{2}\right]\\
&{+}O\left(\left(1-v^{2}\right)^{-\frac{1}{2}}y(t)e^{-2\sqrt{2}y(t)}\right)\\&{+}O\left(\norm{(\psi_{1}(t),\psi_{2}(t))}_{H^{1}_{x}\times L^{2}_{x}}e^{-\sqrt{2}y(t)}+\norm{\psi_{1}(t)}_{H^{1}_{x}}^{3}\right),
\end{aligned}
\end{multline}
where the function $C_{v,y}(x)$ is defined in \eqref{Cc4}.\\
\textbf{Step 4.}\big(Estimate of $-vP_{+}\left(\phi(t),\partial_{t}\phi(t)\right).$\big)\\
First, we recall from \eqref{P+} that $P_{+}\left(\phi(t),\partial_{t}\phi(t)\right)$ is given by
\begin{equation*}
 P_{+}\left(\phi(t),\partial_{t}\phi(t)\right)={-}\int_{0}^{+\infty}\partial_{t}\phi(t,x)\partial_{x}\phi(t,x)\,dx.   
\end{equation*}
Then, while $(\phi(t,x),\partial_{t}\phi(t,x))$ satisfies the formula
\begin{equation*}
(\phi(t,x),\partial_{t}\phi(t,x))=\overrightarrow{H_{-1,0}}((v,y(t)),x)+\overrightarrow{H_{0,1}}((v,y(t)),x)+(\psi_{1}(t,x),\psi_{2}(t,x)),    
\end{equation*}
using the estimates \eqref{lo1} and \eqref{lo2}, we obtain by similar reasoning to the estimate of $(2.12)$ of Lemma $2.3$ in \cite{asympt} that  
\begin{multline}\label{P+v}
\begin{aligned}
{-}vP_{+}\left(\phi(t),\partial_{t}\phi(t)\right)=&{-}\frac{v^{2}}{\sqrt{1-v^{2}}}\norm{ H^{'}_{0,1}}_{L^{2}_{x}}^{2}
    -v\left\langle J\circ C_{v,y(t)},\overrightarrow{\psi}(t)\right\rangle\\ &{+}v\int_{0}^{+\infty}\partial_{x}\psi_{1}(t,x)\psi_{2}(t,x)\,dx+O\left(\frac{v^{2}}{\left(1-v^{2}\right)}y(t)e^{-2\sqrt{2}y(t)}\right)\\ &{+}O\left(\frac{v}{\sqrt{1-v^{2}}}e^{-\sqrt{2}y(t)}\norm{(\psi_{1}(t),\psi_{2}(t))}_{H^{1}_{x}\times L^{2}_{x}}\right),
\end{aligned}
\end{multline}
more precisely the errors in the estimate \eqref{P+v} above come from estimate \eqref{lo1} and Cauchy-Schwarz inequality applied in
\begin{equation*}
    \int_{0}^{+\infty}\left\vert  H^{'}_{-1,0}\left(\frac{x+y(t)}{\sqrt{1-v^{2}}}\right)\right\vert\Big[\vert \partial_{x}\psi_{1}(t,x) \vert+\vert \psi_{2}(t,x) \vert\Big]\,dx,
\end{equation*}
from Lemma \ref{interactt} applied in the following integral 
\begin{equation*}
    \int_{0}^{+\infty} H^{'}_{0,1}\left(\frac{x-y(t)}{\sqrt{1-v^{2}}}\right)H^{'}_{-1,0}\left(\frac{x+y(t)}{\sqrt{1-v^{2}}}\right)\,dx,
\end{equation*}
and from the elementary estimate 
\begin{equation*}
   \int_{{-}\infty}^{0} H^{'}_{0,1}\left(\frac{x-y(t)}{\sqrt{1-v^{2}}}\right)^{2}\,dx+ \int_{0}^{+\infty} H^{'}_{-1,0}\left(\frac{x+y(t)}{\sqrt{1-v^{2}}}\right)^{2}\,dx\lesssim e^{-2\sqrt{2}y(t)},
\end{equation*}
which can be obtained from \eqref{lo1}.\\ 
\textbf{Step 5.}(Estimate and monotonicity of $M(\phi(t),\partial_{t}\phi(t)).$) 
From estimates \eqref{E+comp1} and \eqref{P+v}, we deduce 
\begin{multline}\label{Mfirst}
  \begin{aligned}
  M\left(\phi(t),\partial_{t}\phi(t)\right)=&E_{+}\left(\phi(t),\partial_{t}\phi(t)\right)-vP_{+}\left(\phi(t),\partial_{t}\phi(t)\right)\\=&\sqrt{1-v^{2}}\norm{H^{'}_{0,1}}_{L^{2}_{x}}^{2}\\
    &{+}\frac{1}{2}\left[\int_{0}^{+\infty}\psi_{2}(t,x)^{2}+\partial_{x}\psi_{1}(t,x)^{2}+ U^{''}\left(H_{0,1}\left(\frac{x-y(t)}{\sqrt{1-v(t)^{2}}}\right)\right)\psi_{1}(t,x)^{2}\,dx\right]\\
&{+}O\left(v\norm{(\psi_{1}(t),\psi_{2}(t))}_{H^{1}_{x}\times L^{2}_{x}}^{2}+\norm{(\psi_{1}(t),\psi_{2}(t))}_{H^{1}_{x}\times L^{2}_{x}}e^{-\sqrt{2}y(t)}\right)\\&{+}O\left(\norm{\psi_{1}(t)}_{H^{1}_{x}}^{3}++y(t)e^{-2\sqrt{2}y(t)}\right).
\end{aligned}
\end{multline}
Furthermore, using estimate \eqref{E+comp1} and Lemma \ref{interactt}, we can also verify the following estimates
 \begin{align*}\label{M0}
E_{+}\left(\overrightarrow{H_{0,1}}(v,y(t))+\overrightarrow{H_{-1,0}}(v,y(t))\right)=&\frac{1}{\sqrt{1-v^{2}}}\norm{H^{'}_{0,1}}_{L^{2}_{x}}^{2}+O\left(y(t)e^{{-}2\sqrt{2}y(t)}\right),\\
P_{+}\left(\overrightarrow{H_{0,1}}(v,y(t))+\overrightarrow{H_{-1,0}}(v,y(t))\right)=&\frac{v}{\sqrt{1-v^{2}}}\norm{H^{'}_{0,1}}_{L^{2}_{x}}^{2}+O\left(y(t)e^{{-}2\sqrt{2}y(t)}\right).
\end{align*}
Therefore, we obtain that
\begin{equation}\label{Msolitons}
M\left(\overrightarrow{H_{0,1}}(v,y(t))+\overrightarrow{H_{-1,0}}(v,y(t))\right)=\sqrt{1-v^{2}}\norm{ H^{'}_{0,1}}_{L^{2}_{x}}^{2}+O\left(y(t)e^{{-}2\sqrt{2}y(t)}\right),
\end{equation}
from which we deduce
\begin{multline*}
\begin{aligned}
M\left(\phi(t),\partial_{t}\phi(t)\right)=&M\left(\overrightarrow{H_{0,1}}(v,y(0))+\overrightarrow{H_{-1,0}}(v,y(0))\right)\\&{+}\frac{1}{2}\left[\int_{0}^{+\infty}\psi_{2}(t,x)^{2}+\partial_{x}\psi_{1}(t,x)^{2}+ U^{''}\left(H_{0,1}\left(\frac{x-y(t)}{\sqrt{1-v^{2}}}\right)\right)\psi_{1}(t,x)^{2}\,dx\right]   
\\ &{+}O\left(\max\left\{y(t)e^{{-}2\sqrt{2}y(t)},y(0)e^{{-}2\sqrt{2}y(0)}\right\}\right)\\&{+}O\left(v\norm{(\psi_{1}(t),\psi_{2}(t))}_{H^{1}_{x}\times L^{2}_{x}}^{2}+\norm{(\psi_{1}(t),\psi_{2}(t))}_{H^{1}_{x}\times L^{2}_{x}}^{3}\right).
\end{aligned}
\end{multline*}
\par Consequently, since $M\left(\phi(0),\partial_{t}\phi(0)\right)\geq M(\phi(t),\partial_{t}\phi(t))$ for all $t\geq 0$ and 
\begin{equation*}
\left(\phi(0),\partial_{t}\phi(0)\right)=\overrightarrow{H_{0,1}}(v,y(0))+\overrightarrow{H_{{-}1,0}}(v,y(0))+(\psi_{1}(0),\psi_{2}(0)),
\end{equation*}
we have for every $t\geq 0$ the following estimate
\begin{multline*}
       \int_{0}^{{+}\infty}\psi_{2}(t,x)^{2}+\partial_{x}\psi_{1}(t,x)^{2}+ U^{''}\left(H_{0,1}\left(\frac{x-y(t)}{\sqrt{1-v^{2}}}\right)\right)\psi_{1}(t,x)^{2}\,dx \\\lesssim  y(t)e^{{-}2\sqrt{2}y(t)}+y(0)e^{{-}2\sqrt{2}y(0)}+v\norm{(\psi_{1}(t),\psi_{2}(t))}_{H^{1}_{x}\times L^{2}_{x}}^{2}+\norm{(\psi_{1}(t),\psi_{2}(t))}_{H^{1}_{x}\times L^{2}_{x}}^{3}\\{+}\norm{(\psi_{1}(0),\psi_{2}(0))}_{H^{1}_{x}\times L^{2}_{x}},
\end{multline*}
from which with Lemma \ref{coccer} we deduce for all $t\geq 0$ that
\begin{equation}\label{remest}
    \norm{(\psi_{1}(t),\psi_{2}(t))}_{H^{1}_{x}\times L^{2}_{x}}^{2}\lesssim y(t)e^{{-}2\sqrt{2}y(t)}+y(0)e^{{-}2\sqrt{2}y(0)}+\norm{(\psi_{1}(0),\psi_{2}(0))}_{H^{1}_{x}\times L^{2}_{x}},
\end{equation}
if $v\ll 1.$\\
\textbf{Step 6.}(Final Argument.)\\
The last argument is to prove that the set denoted by
\begin{equation}\label{BO}
    BO=\left\{t\in\mathbb{R}_{\geq 0} \Big\vert  \, \norm{(\psi_{1}(t),\psi_{2}(t))}_{H^{1}_{x}\times L^{2}_{x}}\leq v^{1+\frac{\theta}{4}},\,y(t)\geq y(0) \text{ and \eqref{formm} is true.}\right\},
\end{equation}
is the proper $\mathbb{R}_{\geq 0}.$  From the hypotheses of Theorem \ref{orbittheo} and Step $1,$ we can verify that $0 \in BO.$ 
\par Furthermore, from Step $1,$ we have obtained that there exists $\epsilon>0$ such that if $0\leq t\leq \epsilon,$ then
\begin{equation*}
    (\phi(t,x),\partial_{t}\phi(t,x))=\overrightarrow{H}_{-1,0}\left((v,y(t)),x\right)+\overrightarrow{H}_{0,1}\left((v,y(t)),x\right)+(\psi_{1}(t,x),\psi_{2}(t,x))
\end{equation*}
and 
\begin{equation}\label{lclc}
   \vert y(t)-y_{0} \vert+ \norm{(\psi_{1}(t),\psi_{2}(t))}_{H^{1}_{x}\times L^{2}_{x}}\leq 2C\norm{(u_{1},u_{2})}_{H^{1}_{x}\times L^{2}_{x}}.
\end{equation}
Since $\norm{(u_{1},u_{2})}_{H^{1}_{x}\times L^{2}_{x}}\leq v^{2+\theta}$
and Lemma \ref{newmodulation} implies the estimate $\norm{(\psi_{1}(0),\psi_{2}(0))}_{H^{1}_{x}\times L^{2}_{x}}\lesssim \norm{(u_{1},u_{2})}_{H^{1}_{x}\times L^{2}_{x}},
$ from \eqref{lclc} and Lemma \ref{dependentmod}, we deduce the existence of a constant $0<K$ independent of $\epsilon$ and $v$ such that $y(t)$ is a function of class $C^{1}$ in $[0,\epsilon]$ and for any $t\in[0,\epsilon],$ the inequality
\begin{equation}\label{yyest}
    \vert \dot y(t)-v\vert  \leq K\left[ \norm{(\psi_{1}(t),\psi_{2}(t))}_{H^{1}_{x}\times L^{2}_{x}}+e^{-2\sqrt{2}y(t)}\right]
\end{equation}
is true. Therefore,
\begin{equation}\label{dycontrol}
    \dot y(t)\geq v-K\left[ \norm{(\psi_{1}(t),\psi_{2}(t))}_{H^{1}_{x}\times L^{2}_{x}}+e^{-2\sqrt{2}y(t)}\right],
\end{equation}
while $t\in [0,\epsilon].$ Moreover, from inequality \eqref{lclc} and the observations done before, to prove that $[0,\epsilon]\subset BO$ it is only needed to verify that $y(t)\geq y(0)$ for all $t\in[0,\epsilon].$  
 \par First, since $y(t)$ is continuous for $t \in [0,\epsilon],$ there exists $\epsilon_{2}\in(0,\epsilon)$ such that if $0\leq t\leq \epsilon_{2},$ then
 \begin{equation*}
     y(t)\geq \frac{3y(0)}{4},
 \end{equation*}
so \eqref{lclc}, \eqref{dycontrol} and the estimate
$
    \norm{(\psi_{1}(0),\psi_{2}(0))}_{H^{1}_{x}\times L^{2}_{x}}\lesssim \norm{(u_{1},u_{2})}_{H^{1}_{x}\times L^{2}_{x}}\leq v^{2+\theta}
$
imply that if $0\leq t\leq \epsilon_{2}$ and $0<v\ll 1,$ then
\begin{equation}\label{goodyy}
    \dot y(t)\geq v-v^{2}-K e^{{-}\frac{3\sqrt{2}y(0)}{2}}\geq \frac{4 v}{5}. 
\end{equation}
In conclusion, estimate \eqref{lclc}, the hypothesis of $y_{0}\geq 4\ln{\frac{1}{v}}$ and inequality \eqref{goodyy} imply for $v\ll 1$ that if $0\leq t\leq \epsilon_{2},$ then $y(t)\geq y(0)+\frac{4v}{5}t$ and $[0,\epsilon_{2}]\subset BO.$ 
\par If $t\in [\epsilon_{2},\epsilon],$ it is not difficult to verify that $y(t)\geq y(0)$ in this region. Indeed, the continuity of the function $y$ would imply otherwise the existence of $t_{i}$ satisfying $\epsilon_{2}<t_{i}\leq \epsilon,$ $y(t_{i})=y(0)$ and $y(s)>y(0)$ for any $\epsilon_{2}\leq s< t_{i},$ which implies that estimate \eqref{goodyy} is true for $t\in\left[\epsilon_{2},t_{1}\right]$. But, repeating the argument above, we would conclude that  $y(t_{i})\geq y(0)+\frac{4v}{5}t_{i},$ which is a contradiction. In conclusion, the interval $[0,\epsilon]$ is contained in the set $ BO.$
\par Similarly, from Lemma \ref{dependentmod}, we can use inequality \eqref{dycontrol}
to verify that
$
    y(t)\geq y(0)+\frac{4v}{5}t
$
always when $[0,t] \subset BO.$ Therefore, estimate \eqref{remest} implies
\begin{equation}\label{phiestimate}
\norm{(\psi_{1}(t),\psi_{2}(t))}_{H^{1}_{x}\times L^{2}_{x}(x)}\lesssim \norm{(u_{1},u_{2})}_{H^{1}_{x}\times L^{2}_{x}}^{\frac{1}{2}} +y(0)^{\frac{1}{2}}e^{{-}\sqrt{2}y(0)}\ll v^{1+\frac{\theta}{4}},
\end{equation}
if $[0,t]\in BO.$
\par In conclusion, $BO=\mathbb{R}_{\geq 0}$ and estimates \eqref{yyest}, \eqref{phiestimate} imply the result of Theorem \ref{orbittheo} for all $t\geq 0.$
\end{proof}

\section{Proof of Theorem \ref{maintheo}}\label{colsec}
\par First, from Theorem $1.3$ in the article \cite{multison}, we know for any $0<v<1$ that there exist $\delta(v)>0,\,T(v)>0$ and a solution $\phi(t,x)$ of \eqref{nlww} with finite energy satisfying the identity
\begin{equation}
    \phi(t,x)=H_{0,1}\left(\frac{x-vt}{\left(1-v^{2}\right)^{\frac{1}{2}}}\right)+H_{-1,0}\left(\frac{-x-vt}{\left(1-v^{2}\right)^{\frac{1}{2}}}\right)+\psi(t,x),
\end{equation}
and the following decay estimate
\begin{equation}\label{uniqq}
    \sup_{t\geq T}\norm{(\psi(t,x),\partial_{t}\psi(t,x))}_{H^{1}_{x}\times L^{2}_{x}}e^{\delta t}<+\infty,
\end{equation}
for any  $T\geq T(v)$ and $\delta\leq \delta(v).$ Moreover, we can find $\delta(v),\,T(v)>0$ such that
\begin{equation}\label{contraction1}
    \sup_{t\geq T(v)}\norm{(\psi(t,x),\partial_{t}\psi(t,x))}_{H^{1}_{x}\times L^{2}_{x}}e^{\delta(v)t}<1,
\end{equation}
indeed, in \cite {multison} it was proved using Fixed point theorem that for any $0<v<1$ that there is a unique solution of \eqref{nlww} that satisfies \eqref{uniqq} for some $T,\,\delta>0.$
\par Next, if we restrict the argument of the proof of Proposition $3.6$ of \cite{multison} to the traveling kink-kink of the $\phi^{6}$ model, we can find explicitly the values of $\delta(v)$ and $T(v).$ More precisely, we have:
\begin{theorem}\label{purepure}
There is $\delta_{0}>0$ such that if $0<v<\delta_{0},$ then there exists a unique solution $\phi(t,x)$ of \eqref{nlww} with
\begin{equation*}
    h(t,x)=\phi(t,x)-H_{0,1}\left(\frac{x-vt}{\left(1-v^{2}\right)^{\frac{1}{2}}}\right)-H_{-1,0}\left(\frac{x+vt}{\left(1-v^{2}\right)^{\frac{1}{2}}}\right),
\end{equation*}
satisfying \eqref{uniqq} for some $0<\delta<1$ and $T>0.$ Furthermore, we have if $t\geq \frac{4\ln{\left(\frac{1}{v}\right)}}{v}$ that
\begin{equation}
    \norm{(h(t,x),\partial_{t}h(t,x))}_{H^{1}_{x}\times L^{2}_{x}}\leq e^{-v t}.
\end{equation}
This solution is also an odd function on $x.$
\end{theorem}
\begin{proof}
See the Appendix section $B.$\\
\end{proof}
\par Finally, we have obtained all the framework necessary to start the demonstration of Theorem \ref{maintheo}.
\begin{proof}[Proof of Theorem \ref{maintheo}]
\par First, from Theorem \ref{purepure}, for any $k\in\mathbb{N}$ bigger than $2$ and $0<v\leq \delta_{0},$ we have that the travelling kink-kink with speed $v$ satisfies for $T_{0,k}=\frac{32k\ln{\left(\frac{1}{v^{2}}\right)}}{2\sqrt{2}v}$ the following estimate:
\begin{equation}\label{boundary}
    \norm{(h(T_{0,k}),\partial_{t}h(T_{0,k}))}_{H^{1}_{x}\times L^{2}_{x}}\leq v^{16\sqrt{2}k},
\end{equation}
for $h(t,x)$ the function denoted in Theorem \ref{purepure}. Now, we start the proof of the second item of Theorem \ref{maintheo}.\\
\textbf{Step 1.}(Proof of second item of Theorem \ref{maintheo}.)
\par First, in notation of Theorem \ref{toobig}, we consider
\begin{equation*}
    \phi_{k}(v,t,x)=\varphi_{k,v}(t,x+\tau_{k,v}).
\end{equation*}
For the $T_{0,k}$ given before, we can verify using Theorems \ref{approximated theorem},  \ref{toobig} that 
\begin{gather*}
    \norm{\phi_{k}(v,T_{0,k},x)-H_{0,1}\left(\frac{x-vT_{0,k}}{\sqrt{1-v^{2}}}\right)-H_{-1,0}\left(\frac{x+vT_{0,k}}{\sqrt{1-v^{2}}}\right)}_{H^{1}_{x}}\\
    +\norm{\partial_{t}\phi_{k}(v,T_{0,k},x)+\frac{v}{\sqrt{1-v^{2}}} H^{'}_{0,1}\left(\frac{x-vT_{0,k}}{\sqrt{1-v^{2}}}\right)-\frac{v}{\sqrt{1-v^{2}}} H^{'}_{-1,0}\left(\frac{x+vT_{0,k}}{\sqrt{1-v^{2}}}\right)}_{H^{1}_{x}}\leq v^{15 k}.
\end{gather*}
In conclusion, Theorem \ref{energyE} and Remark \ref{Important} imply that there is $\Delta_{k,\theta}>0$ such that if also $v<\Delta_{k,\theta},$ then
\begin{equation*}
     \norm{(\phi(t,x),\partial_{t}\phi(t,x))-(\phi_{k}(v,t,x),\partial_{t}\phi_{k}(v,t,x))}_{H^{1}_{x}\times L^{2}_{x}}<v^{2k-\frac{1}{2}},
\end{equation*}
while 
\begin{equation*}
    \vert t-T_{0,k}\vert<\frac{\ln{\left(\frac{1}{v}\right)}^{2-\frac{\theta}{2}}}{v}.
\end{equation*}
Also, Theorem \ref{toobig} and Theorem \ref{approximated theorem} implies that if $v\ll1$ and  
\begin{equation*}
  - 4\frac{\ln{\left(\frac{1}{v}\right)}^{2-\theta}}{v}\leq t\leq-\frac{\ln{\left(\frac{1}{v}\right)}^{2-\theta}}{v},  
\end{equation*}
then there exist $e_{k,v}$ satisfying $\left\vert e_{v,k}-\frac{1}{\sqrt{2}}\ln{\left(\frac{8}{v^{2}}\right)}\right\vert\ll 1$ such that 
\begin{multline}\label{lastestt}
     \norm{\phi_{k}(v,t,x)-H_{0,1}\left(\frac{x-e_{k,v}+vt}{\sqrt{1-v^{2}}}\right)-H_{-1,0}\left(\frac{x+e_{k,v}-vt}{\sqrt{1-v^{2}}}\right)}_{H^{1}_{x}}\\
    +\norm{\partial_{t}\phi_{k}(v,t,x)-\frac{v}{\sqrt{1-v^{2}}} H^{'}_{0,1}\left(\frac{x-e_{k,v}+vt}{\sqrt{1-v^{2}}}\right)+\frac{v}{\sqrt{1-v^{2}}} H^{'}_{-1,0}\left(\frac{x+e_{k,v}-vt}{\sqrt{1-v^{2}}}\right)}_{H^{1}_{x}}\ll v^{2k-\frac{1}{2}}.
\end{multline}
In conclusion, the second item of Theorem \ref{maintheo} follows from the observation above and Remark \ref{Important}.\\
\textbf{Step 2.}(Proof of first item of Theorem \ref{maintheo}.)
\par From Step $1,$ for $t_{0}=-\frac{\ln{\left(\frac{1}{v}\right)}^{2-\theta}}{v},$ it was obtained that  $\phi(t_{0},x)$ satisfies 
\eqref{lastestt}. Next, we are going to study the behavior of $\phi(t,x)$ for $t\leq t_{0},$ which is equivalent to study the function $\phi_{1}(t,x)=\phi(-(t+t_{0}),x)$ for $t\geq 0.$  
\par However, from the estimate \eqref{lastestt}, we can verify that $(\phi_{1}(0,x),\partial_{t}\phi_{1}(0,x))$ satisfies the hypotheses of Theorem \ref{orbittheo}, if we consider $y_{0}=e_{k,v}-vt_{0}$ and  $0<v<\ll1.$ Therefore, using the result of Theorem \ref{orbittheo} and the identity $\phi_{1}(t,x)=\phi(-(t+t_{0}),x),$ we obtain the first item of Theorem \ref{maintheo}. 
\end{proof}
\section*{Acknowledgement(s)}

The author acknowledges the support of his supervisors
Thomas Duyckaerts and Jacek Jendrej for providing helpful comments and
orientation, which were essential to conclude this paper. The author
is also grateful to the math department LAGA of the University Sorbonne Paris
Nord and all the referees for providing remarks and suggestions in the writing of
this manuscript.

\section*{Funding}

The author acknowledges the support of the French State Program ”Investisse-
ment d’Avenir”, managed by the ”Angence Nationale de la Recherche” under
the grant ANR-18-EURE-0024. The author is a Ph.D. candidate at the University
Sorbonne Paris Nord.
\bibliographystyle{plain}
\bibliography{interacttfqsample}
\bigskip

\appendix
\section{Auxiliary Estimates}\label{aux}
In this Appendix section, we complement our article by demonstrating complementary estimates.
\begin{lemma}\label{est1}
For 
\begin{gather*}
     \mathcal{G}(x)=e^{-\sqrt{2}x}-\frac{e^{-\sqrt{2}x}}{(1+e^{2\sqrt{2}x})^{\frac{3}{2}}}+x\frac{e^{\sqrt{2}x}}{(1+e^{2\sqrt{2}x})^{\frac{3}{2}}}+k_{1}\frac{e^{\sqrt{2}x}}{(1+e^{2\sqrt{2}x})^{\frac{3}{2}}},
\end{gather*}
we have that
\begin{align*}
  \int_{\mathbb{R}} U^{(3)}(H_{0,1}(x)) H^{'}_{0,1}(x)^{2}\mathcal{G}(x)\,dx=& \int_{\mathbb{R}}U^{(3)}(H_{0,1}(x)) H^{'}_{0,1}(x)^{2}e^{-\sqrt{2}x}\,dx\\&{-}\sqrt{2}\int_{\mathbb{R}}\left[U^{''}(H_{0,1}(x))-2\right]H^{'}_{0,1}(x)e^{-\sqrt{2}x}\,dx.
  \end{align*}
\end{lemma}
\begin{remark}
Indeed, the value $k_{1}$ in Lemma \ref{est1} can be replaced by zero, since
\begin{equation*}
    \int_{\mathbb{R}}U^{(3)}(H_{0,1}(x))H^{'}_{0,1}(x)^{3}\,dx=0.
\end{equation*}
\end{remark}
\begin{proof}[proof of Lemma \ref{est1}.]
\par First, from identity $ H^{''}_{0,1}(x)= U^{'}(H_{0,1}(x))$ and integration by parts, we can verify the following identity
\begin{align*}
\int_{\mathbb{R}}U^{(3)}\left(H_{0,1}(x)\right)H^{'}_{0,1}(x)^{2}\mathcal{G}(x)\,dx
=&\int_{\mathbb{R}} U^{'}(H_{0,1}(x))\left[ \mathcal{G}^{''}(x)- U^{''}(H_{0,1})\mathcal{G}(x)\right]\,dx,
\end{align*}
also, since $-\mathcal{G}^{''}(x)+ U^{''}(H_{0,1}(x))\mathcal{G}(x)=\left[ U^{''}(H_{0,1}(x))-2\right]e^{-\sqrt{2}x}+8\sqrt{2} H^{'}_{0,1}(x)$ and $\left\langle  H^{'}_{0,1},\,U^{'}(H_{0,1})\right\rangle=0,$ we conclude using integration by parts that
\begin{align*}
\int_{\mathbb{R}}U^{(3)}\left(H_{0,1}(x)\right)H^{'}_{0,1}(x)^{2}\mathcal{G}(x)\,dx=&{-}\int_{\mathbb{R}} U^{'}\left(H_{0,1}(x)\right)\left[ U^{''}(H_{0,1}(x))-2\right]e^{{-}\sqrt{2}x}\,dx\\
=&{-}\int_{\mathbb{R}} H^{''}_{0,1}(x)\left[U^{''}\left(H_{0,1}(x)\right)-2\right]e^{-\sqrt{2}x}\,dx,\\
=&\int_{\mathbb{R}}U^{(3)}\left(H_{0,1}(x)\right)H^{'}_{0,1}(x)^{2}e^{-\sqrt{2}x}\,dx\\&{-}\sqrt{2}\int_{\mathbb{R}}\left[ U^{''}\left(H_{0,1}(x)\right)-2\right]H^{'}_{0,1}(x)e^{-\sqrt{2}x}\,dx.
\end{align*}
\end{proof}
\par Now, using integration by parts and identity $(27)$ of \cite{first}, we have that
\begin{equation}\label{IDD2}
    -\int_{\mathbb{R}}\left[ U^{''}\left(H_{0,1}(x)\right)-2\right]e^{-\sqrt{2}x} H^{'}_{0,1}(x)\,dx=-\sqrt{2}\int_{\mathbb{R}}\left[6 H_{0,1}(x)^{5}-8H_{0,1}(x)^{3}\right]e^{-\sqrt{2}x}\,dx=4,
\end{equation}
from which we deduce the following Lemma.
\begin{lemma}\label{est2}
\begin{equation*}
\int_{\mathbb{R}} U^{(3)}(H_{0,1}(x)) H^{'}_{0,1}(x)^{2}\mathcal{G}(x)\,dx-\int_{\mathbb{R}}U^{(3)}(H_{0,1}(x)) H^{'}_{0,1}(x)^{2}e^{-\sqrt{2}x}\,dx=4\sqrt{2}. 
\end{equation*}
\end{lemma}
\begin{lemma}\label{coercir}
There is $\delta>0,\,c>0$ such that if $0<v<\delta,\,d(t)=\frac{1}{\sqrt{2}}\ln{\left(\frac{8}{v^{2}}\cosh{\left(\sqrt{2}vt\right)}^{2}\right)},$ then for
\begin{align*}
   H^{+}_{0,1}(x,t)=H_{0,1}\left(\frac{x-\frac{d(t)}{2}}{\sqrt{1-\frac{\dot d(t)^{2]}}{4}}}\right),\\
   H^{-}_{0,1}(x,t)=H_{-1,0}\left(\frac{x+\frac{d(t)}{2}}{\sqrt{1-\frac{\dot d(t)^{2]}}{4}}}\right),
\end{align*}
and any $g \in H^{1}_{x}(\mathbb{R})$ such that
\begin{equation*}
    \left\langle g(x),\partial_{x}H^{+}_{0,1}(x,t)\right\rangle=0,\,\left\langle g(x),\partial_{x}H^{-}_{0,1}(x,t)\right\rangle=0,
\end{equation*}
we have
\begin{equation}
c\norm{g}_{H^{1}_{x}}^{2}\leq\left\langle -\partial^{2}_{x}g(x)+ U^{''}\left(H^{+}_{0,1}(x,t)+H^{-}_{0,1}(x,t)\right)g(x),\,g(x)\right\rangle.
\end{equation}
\begin{proof}[Proof of Lemma \ref{coercir}]
First, to simplify our computations we denote
\begin{equation*}
    \gamma_{d(t)}=\frac{1}{\sqrt{1-\frac{\dot d(t)^{2}}{4}}}.
\end{equation*}
Next, we can verify using a change of variables that
\begin{equation*}
    \left\langle U^{''}\left(H^{+}_{0,1}(x,t)\right)g(x),\,g(x)\right\rangle=\sqrt{1-\frac{\dot d(t)^{2}}{4}}\int_{\mathbb{R}} U^{''}\left(H_{0,1}(y)\right)\left[g\left(\left(y+\frac{d(t)}{2}\gamma_{d(t)}\right)\gamma_{d(t)}^{-1}\right)\right]^{2}\,dy,
\end{equation*}
and
\begin{equation}\label{did}
    \int_{\mathbb{R}} \frac{dg(x)}{dx}^{2}\,dx=\frac{1}{\sqrt{1-\frac{\dot d(t)^{2}}{4}}}\displaystyle\int\limits_{\mathbb{R}}\left[\frac{d}{dy}\left[g\left(y\gamma_{d(t)}^{-1}\right)\right]\right]^{2}\,dy.
\end{equation}
We denote now
\begin{equation*}
    g_{1}(t,y)=g\left(y\sqrt{1-\frac{\dot d(t)^{2}}{4}}\right)=g(y \gamma_{d(t)}^{-1}).
\end{equation*}
Moreover, $L=-\partial^{2}_{x}+ U^{''}(H_{0,1}(x))$ is a positive operator in $L^{2}(\mathbb{R})$ when it is restricted to the orthogonal complement of $H^{'}_{0,1}(x)$ in $L^{2}_{x}(\mathbb{R}),$ see \cite{jkl} or \cite{first} for the proof.
In conclusion, we deduce that there is a constant $C>0$ independent of $v>0$ such that
\begin{equation}
    \left\langle -\frac{d^{2}}{dx^{2}}g(x)+ U^{''}\left(H^{+}_{0,1}(x,t)\right)g(x),\,g(x)\right\rangle \geq C \sqrt{1-\frac{\dot d(t)^{2}}{4}}\norm{g_{1}(t,y)}_{H^{1}_{y}(\mathbb{R})}^{2},
\end{equation}
so, from $\dot d(t)=v\tanh{(\sqrt{2}vt)}$ and identity \eqref{did}, we deduce that there is a constant $C_{1}>0$ such that if $v\ll 1,$ then
\begin{equation}
     \left\langle -\frac{d^{2}}{dx^{2}}g(x)+ U^{''}\left(H^{+}_{0,1}(x,t)\right)g(x),\,g(x)\right\rangle\geq C_{1}\norm{g(x)}_{H^{1}(\mathbb{R})}^{2}.
\end{equation}
Similarly, we can verify for the same constant $C_{1}>0$ that if $\left\langle g(x),\,\partial_{x}H^{-}_{-1,0}(x,t)\right\rangle_{L^{2}_{x}}=0$ and $v\ll 1,$ then  
\begin{equation}
     \left\langle {-}\frac{d^{2}}{dx^{2}}g(x)+U^{''}\left(H^{-}_{0,1}(x,t)\right)g(x),\,g(x)\right\rangle \geq C_{1}\norm{g(x)}_{H^{1}(\mathbb{R})}^{2}.
\end{equation}
The remaining part of the proof proceeds exactly as the proof of Lemma $2.6$ of \cite{first}.
\end{proof}
\end{lemma}
\begin{lemma}\label{coccer}
     There exist $C>1,\,c>0\,\delta>0$ such that if $0<v<\delta,$ 
then we have for any $(\varphi_{1},\varphi_{2})\in H^{1}_{x}(\mathbb{R})\times L^{2}_{x}(\mathbb{R})$ that
\begin{equation*}
    \int_{\mathbb{R}}\varphi_{2}^{2}+\partial_{x}\varphi_{1}^{2}+U^{''}\left(H_{0,1}\left(\frac{x}{\sqrt{1-v^{2}}}\right)\right)\varphi_{1}(x)^{2}\,dx\geq c\norm{(\varphi_{1},\varphi_{2})}_{H^{1}_{x}\times L^{2}_{x}}^{2}-C\left\langle (\varphi_{1},\varphi_{2}),J D_{v,0}(x)\right\rangle^{2}.
\end{equation*}
\end{lemma}
\begin{proof}
    The proof is completely analogous to the proof of property $(2)$ of Lemma $2.8$ in the article \cite{asympt}.
\end{proof}
\section{Proof of Theorem \ref{purepure}}
We start by denoting
\begin{equation*}
    J=\begin{bmatrix}
     0 & 1\\
     -1 & 0
    \end{bmatrix},
\end{equation*}
and we consider for $x\in\mathbb{R}$ and $-1<v<1$ the following functions
\begin{align}\label{psi-1,0,0}
    \psi^{0}_{-1,0}(x,v)=
    J
    \begin{bmatrix}
      H^{'}_{-1,0}(\frac{x}{\sqrt{1-v^{2}}})\\
     \frac{v}{1-v^{2}}H^{(2)}_{-1,0}\left(\frac{x}{\sqrt{1-v^{2}}}\right)
    \end{bmatrix},\\ \label{psi-1,0,1}
    \psi^{1}_{-1,0}(x,v)=J
    \begin{bmatrix}
    v x  H^{'}_{-1,0}\left(\frac{x}{\sqrt{1-v^{2}}}\right)\\
   \frac{1}{\sqrt{1-v^{2}}} H^{'}_{-1,0}\left(\frac{x}{\sqrt{1-v^{2}}}\right)+ \frac{v^{2} x}{1-v^{2}} H^{(2)}_{-1,0}\left(\frac{x}{\sqrt{1-v^{2}}}\right)
\end{bmatrix},
\end{align}
and we denote, for $j\in \{0,1\},$ $\psi^{j}_{0,1}(x,v)=\psi^{j}_{-1,0}(-x,-v).$
\par Next, we will use Lemma $2.6$ of \cite{multison}.
\begin{lemma}
The functions
\begin{align}
 Y^{0}_{-1,0}(v;x,t)={-}J \psi^{0}_{-1,0}(x+vt,v),\\ Y^{1}_{-1,0}(v;x,t)={-}J\psi^{1}_{-1,0}(x+vt,v)+t\sqrt{1-v^{2}}  Y^{0}_{-1,0}(v;x+vt,t)
\end{align}
are solutions of the linear differential system
\begin{equation}
    \frac{d}{dt}\begin{bmatrix}
    w_{1}(t)\\
    w_{2}(t)
    \end{bmatrix}=
    J\begin{bmatrix}
    -\frac{\partial^{2}}{\partial x^{2}}+ U^{''}\left(H_{-1,0}\left(\frac{x+vt}{\sqrt{1-v^{2}}}\right)\right) & 0\\
    0 & 1
    \end{bmatrix}
    \begin{bmatrix}
    w_{1}(t)\\
    w_{2}(t)
    \end{bmatrix},
\end{equation}
and the functions
\begin{align}
 Y^{0}_{0,1}(v;x,t)=-J \psi^{0}_{0,1}(x-vt,v),\\ Y^{1}_{0,1}(v;x,t)={-}J\psi^{1}_{0,1}(x-vt,v)+t \sqrt{1-v^{2}}  Y^{0}_{0,1}(v;x-vt,t)
\end{align}
are solutions of the linear differential system
\begin{equation}
    \frac{d}{dt}\begin{bmatrix}
    w_{1}(t)\\
    w_{2}(t)
    \end{bmatrix}=
    J\begin{bmatrix}
    -\frac{\partial^{2}}{\partial x^{2}}+ U^{''}\left(H_{0,1}\left(\frac{x-vt}{\sqrt{1-v^{2}}}\right)\right) & 0\\
    0 & 1
    \end{bmatrix}
    \begin{bmatrix}
    w_{1}(t)\\
    w_{2}(t)
    \end{bmatrix}.
\end{equation}
\end{lemma}
Now, similarly to \cite{multison}, we consider the linear operator $L_{+,-}(v,t)$ defined by
\begin{equation}
    L_{+,-}(v,t)=
    \begin{bmatrix}
     -\frac{\partial^{2}}{\partial x^{2}}+U^{''}\left(H_{0,1}\left(\frac{x-vt}{\sqrt{1-v^{2}}}\right)+H_{-1,0}\left(\frac{x+vt}{\sqrt{1-v^{2}}}\right)\right) & 0\\
     0 & 1
    \end{bmatrix}.
\end{equation}
We recall that
\begin{equation*}
    H_{0,1}(x)=\frac{e^{\sqrt{2}x}}{\sqrt{1+e^{2\sqrt{2}x}}},
\end{equation*}
and 
\begin{equation*}
    \left\vert\frac{d^{l}}{dx^{l}}H_{0,1}(x)\right\vert\lesssim \min\left(e^{\sqrt{2}x},e^{-2\sqrt{2}x}\right),
\end{equation*}
for any $l\in\mathbb{N}.$
\par From now on, we denote $\psi^{j}_{-1,0}(v;t,x)=\psi^{j}_{-1,0}(x+vt,v)$ and $\psi^{j}_{0,1}(v;t,x)=\psi^{j}_{-1,0}(x-vt,v)$ for any $j\in\{0,1\}.$
Furthermore, using Lemma \ref{interactt}, we can verify similarly to the proof of Proposition $2.8$ of \cite{multison} the following result.
\begin{lemma}\label{projest}
There exists $C>0$, such that for any $0<v<1,$ we have for all $t\in \mathbb{R}_{\geq 1}$ that 
\begin{gather*}
  \norm{\frac{\partial}{\partial t}\psi^{0}_{0,1}(v;t,x)- L_{+,-}J\psi^{0}_{0,1}(v;t,x)}_{L^{2}_{x}}\leq C \exp\left(\frac{{-}2\sqrt{2}v\vert t\vert}{\sqrt{1-v^{2}}} \right),\\ \norm{\frac{\partial}{\partial t}\psi^{0}_{-1,0}(v;t,x)- L_{+,-}J\psi^{0}_{-1,0}(v;t,x)}_{L^{2}_{x}}\leq C \exp\left(\frac{{-}2\sqrt{2}v\vert t\vert}{\sqrt{1-v^{2}}} \right),\\
     \norm{\frac{\partial}{\partial t}\psi^{1}_{0,1}(v;t,x)- L_{+,-}J\psi^{1}_{0,1}(v;t,x)+\sqrt{1-v^{2}}\psi^{0}_{0,1}(v;t,x)}_{L^{2}_{x}}\\\leq C(\vert t\vert v+1)v\exp\left(\frac{{-}2\sqrt{2}v\vert t\vert }{\sqrt{1-v^{2}}}\right),\\
     \norm{\frac{\partial}{\partial t}\psi^{1}_{-1,0}(v;t,x)- L_{+,-}J\psi^{1}_{-1,0}(v;t,x)+\sqrt{1-v^{2}}\psi^{0}_{-1,0}(v;t,x)}_{L^{2}_{x}}\\\leq C(\vert t\vert v+1)v\exp\left(\frac{{-}2\sqrt{2}v\vert t\vert }{\sqrt{1-v^{2}}}\right).
\end{gather*}
\end{lemma}
\par Next, we consider a smooth cut function $0\leq\chi(x)\leq 1$ that satisfies 
\begin{equation*}
   \chi(x)=
   \begin{cases}
        1, \text{if $ x\leq 2(1-10^{-3}),$}\\ 
        0, \text{if $x\geq 2.$}
    \end{cases}
\end{equation*}
From now on, for each $0<v<1,$ we consider $p(v)=\frac{v}{2}\left(1-10^{-3}\right)$ and we also denote
\begin{equation*}
    \chi_{1}(v;t,x)=\chi\left(\frac{x+vt}{p(v) t}\right),\,\chi_{2}(v;t,x)=1-\chi\left(\frac{x+vt}{p(v) t}\right).
\end{equation*}
\begin{lemma}\label{coerQ}
There is $c,\delta_{0}>0$ such that if $0<v<\delta_{0},$ then 
\begin{gather*}
    Q(t,r)=\frac{1}{2}\left[\int_{\mathbb{R}} \partial_{t}r(t,x)^{2}+\partial_{x}r(t,x)^{2}+ U^{''}\left(H_{0,1}\left(\frac{x-vt}{\sqrt{1-v^{2}}}\right)+H_{-1,0}\left(\frac{x+vt}{\sqrt{1-v^{2}}}\right)\right)r(t,x)^{2}\,dx\right]\\
    +\sum_{j=1}^{2}v\int_{\mathbb{R}}\chi_{j}(v;t,x)(-1)^{j}\partial_{t}r(t,x)\partial_{x}r(t,x)\,dx,
\end{gather*}
satisfies for any $t\geq \frac{\ln\left(\frac{1}{v}\right)}{v}$
\begin{equation*}
    Q(t,r)\geq c\norm{\overrightarrow{r}(t)}_{H^{1}_{x}\times L^{2}_{x}}^{2}-\frac{1}{c}\left[\sum_{j=0}^{1} \left\langle \overrightarrow{r}(t),\psi^{j}_{-1,0}(v;t)\right\rangle^{2}+\left\langle \overrightarrow{r}(t),\psi^{j}_{0,1}(v;t)\right\rangle^{2}\right].
\end{equation*}
\end{lemma}
\begin{proof}
From definition of $\psi^{1}_{-1,0}$ and $\psi^{1}_{0,1},$ we can verify that there is a constant $C>0$ such that if $v\ll 1,$ then
\begin{align}\label{ll1}
    \left\vert\left\langle r(t),H^{'}_{0,1}\left(\frac{x-vt}{\sqrt{1-v^{2}}}\right) \right\rangle^{2}\right\vert\leq C\left[\langle (r(t),\partial_{t}r(t)),\psi^{1}_{0,1}(v;t)\rangle^{2}+v^{2}\norm{(r(t),\partial_{t}r(t))}_{H^{1}_{x}\times L^{2}_{x}}^{2}\right],\\ \label{ll2}
    \left\vert\left\langle r(t),\dot H_{-1,0}\left(\frac{x+vt}{\sqrt{1-v^{2}}}\right) \right\rangle^{2}\right\vert\leq C\left[\langle (r(t),\partial_{t}r(t)),\psi^{1}_{-1,0}(v;t)\rangle^{2}+v^{2}\norm{(r(t),\partial_{t}r(t))}_{H^{1}_{x}\times L^{2}_{x}}^{2}\right].
\end{align}
Then, using the estimates \eqref{ll1} and \eqref{ll2}, the proof of Lemma \ref{coerQ} is analogous to the demonstration of Lemma $2.3$ of \cite{jkl} or the proof of Lemma $2.5$ in \cite{first} or the demonstration of Lemma \ref{coercir} in the section Appendix $A.$ 
\end{proof}
\begin{remark}
Indeed, Proposition $2.10$ of \cite{multison} implies that for any $0<v<1,$ there is $T_{v}$ and $c_{v},$ such that Lemma \ref{coerQ} holds with $c_{v}$ in the place of $c$ for all $t\geq T_{v}.$
\end{remark}
\begin{lemma}\label{solutionp}
There is $C>0,$ such that, for any $0<v<1,$ if $f(t,x)\in L^{\infty}_{t}(\mathbb{R};H^{1}_{x}(\mathbb{R}))$ and $h(t,x) \in L^{\infty}_{t}\left(\mathbb{R}_{\geq 1};H^{1}_{x}(\mathbb{R})\right)\cap C^{1}_{t}\left(\mathbb{R}_{\geq 1};L^{2}_{x}(\mathbb{R})\right)$ is a solution of the integral equation associated to the following partial differential equation
\begin{equation*}
    \partial^{2}_{t}h(t,x)-\partial^{2}_{x}h(t,x)+ U^{''}\left(H_{0,1}\left(\frac{x-vt}{\sqrt{1-v^{2}}}\right)+H_{-1,0}\left(\frac{x+vt}{\sqrt{1-v^{2}}}\right)\right)h(t,x)=f(t,x),
\end{equation*}
for some boundary condition $(h(t_{0}),\partial_{t}h(t_{0}))\in H^{1}_{x}(\mathbb{R})\times L^{2}_{x}(\mathbb{R}),$ then
\begin{gather*}
    Q(t,h)=\frac{1}{2}\left[\int_{\mathbb{R}} \partial_{t}h(t,x)^{2}+\partial_{x}h(t,x)^{2}+U^{''}\left(H_{0,1}\left(\frac{x-vt}{\sqrt{1-v^{2}}}\right)+H_{-1,0}\left(\frac{x+vt}{\sqrt{1-v^{2}}}\right)\right)h(t,x)^{2}\,dx\right]\\
    +\sum_{j=1}^{2}v\int_{\mathbb{R}}\chi_{j}(v;t,x)(-1)^{j}\partial_{t}h(t,x)\partial_{x}h(t,x)\,dx,
\end{gather*}
satisfies
\begin{multline*}
    \left \vert \frac{\partial}{\partial t}Q(t,h)\right\vert\leq C\Bigg[\norm{f(t)}_{L^{2}_{x}}\norm{(h(t),\partial_{t}h(t))}_{H^{1}_{x}\times L^{2}_{x}}\\+\norm{(h(t),\partial_{t}h(t))}_{H^{1}_{x}\times L^{2}_{x}}^{2}\left(v\exp\left(\frac{-\sqrt{2}v t\left(1-10^{-3}\right)^{2}}{\sqrt{1-v^{2}}}\right)+\frac{1}{t}\right)\Bigg]
\end{multline*}
for all $t\geq 1.$
\end{lemma}
\begin{proof}
First, from the equation satisfied by $h(t,x),$ we obtain that
\begin{multline}\label{ff}
    \int_{\mathbb{R}}\left[\partial^{2}_{t}h(t,x)-\partial^{2}_{x}h(t,x)+U^{''}\left(H_{0,1}\left(\frac{x-vt}{\sqrt{1-v^{2}}}\right)+H_{-1,0}\left(\frac{x+vt}{\sqrt{1-v^{2}}}\right)\right)h(t,x)^{2}\right]\partial_{t}h(t,x)\,dx\\
    =\int_{\mathbb{R}}f(t,x)\partial_{t}h(t,x)\,dx.
\end{multline}
As a consequence, we deduce by integration by parts that
\begin{multline}\label{este1}
    \frac{d}{ dt}\left[\int_{\mathbb{R}} \partial_{t}h(t)^{2}+\partial_{x}h(t)^{2}+U^{''}\left(H_{0,1}\left(\frac{x-vt}{\sqrt{1-v^{2}}}\right)+H_{-1,0}\left(\frac{x+vt}{\sqrt{1-v^{2}}}\right)\right)h(t)^{2}\,dx\right]
    \\
    \begin{aligned}
    =&{-}\frac{v}{\sqrt{1-v^{2}}}\int_{\mathbb{R}}U^{(3)}\left(H_{0,1}\left(\frac{x-vt}{\sqrt{1-v^{2}}}\right)+H_{-1,0}\left(\frac{x+vt}{\sqrt{1-v^{2}}}\right)\right) H^{'}_{0,1}\left(\frac{x-vt}{\sqrt{1-v^{2}}}\right)h(t)^{2}\,dx\\ &{+}\frac{v}{\sqrt{1-v^{2}}}\int_{\mathbb{R}}U^{(3)}\left(H_{0,1}\left(\frac{x-vt}{\sqrt{1-v^{2}}}\right)+H_{-1,0}\left(\frac{x+vt}{\sqrt{1-v^{2}}}\right)\right) H^{'}_{-1,0}\left(\frac{x+vt}{\sqrt{1-v^{2}}}\right)h(t)^{2}\,dx\\
&{+}2\int_{\mathbb{R}}f(t,x)h(t,x)\,dx.
\end{aligned}
\end{multline}
\par Next, from the definition of $\chi_{1}(v;t,x)$ and $\chi_{2}(v;t,x),$ we can verify for each $j\in \{1,2\}$ that
\begin{multline*}
   \begin{aligned}
    \frac{d}{dt}\left[v\int_{\mathbb{R}}\chi_{j}(v;t,x)(-1)^{j}\partial_{t}h(t,x)\partial_{x}h(t,x)\,dx\right]=&v\int_{\mathbb{R}}\chi_{j}(v;t,x)(-1)^{j}\partial^{2}_{t}h(t,x)\partial_{x}h(t,x)\,dx\\
&{+}v\int_{\mathbb{R}}\chi_{j}(v; t,x)(-1)^{j}\partial_{t}h(t,x)\partial^{2}_{t,x}h(t,x)\,dx
    \\&{+}
    O\left(\norm{\dot \chi}_{L^{\infty}_{x}(\mathbb{R})}\frac{v}{t }\norm{(h(t),\partial_{t}h(t))}_{H^{1}_{x}\times L^{2}_{x}}^{2}\right),
\end{aligned}
\end{multline*}
from which we deduce using integration by parts that
\begin{multline}\label{este2}
    \begin{aligned}
    \frac{d}{dt}\left[v\int_{\mathbb{R}}\chi_{j}(v;t,x)(-1)^{j}\partial_{t}h(t,x)\partial_{x}h(t,x)\,dx\right]= &v\int_{\mathbb{R}}\chi_{j}(v;t,x)(-1)^{j}\partial^{2}_{t}h(t,x)\partial_{x}r(t,x)\,dx\\
  &{+}O\left(\norm{\dot \chi}_{L^{\infty}_{x}(\mathbb{R})}\frac{1}{t }\norm{(h(t),\partial_{t}h(t))}_{H^{1}_{x}\times L^{2}_{x}}^{2}\right).
\end{aligned}
\end{multline}
From the equation satisfied by $h(t,x),$ we have that
\begin{multline*}
v\int_{\mathbb{R}}\chi_{j}(v;t,x)(-1)^{j}\partial^{2}_{t}h(t,x)\partial_{x}h(t,x)\,dx\\
 \begin{aligned}
=&v\int_{\mathbb{R}}\chi_{j}(v;t,x)(-1)^{j}f(t,x)\partial_{x}h(t,x)\,dx\\&{+}v\int_{\mathbb{R}}\chi_{j}(v;t,x)(-1)^{j}\partial^{2}_{x}h(t,x)\partial_{x}h(t,x)\,dx\\&{-}v\int_{\mathbb{R}}\chi_{j}(v;t,x)(-1)^{j} U^{''}\left(H_{0,1}\left(\frac{x-vt}{\sqrt{1-v^{2}}}\right)+H_{-1,0}\left(\frac{x+vt}{\sqrt{1-v^{2}}}\right)\right)h(t,x)\partial_{x}h(t,x)\,dx.\end{aligned}
\end{multline*}
So, using integration by parts, we obtain for any $j\in\{1,\,2\}$ that
\begin{multline*}
     2\sqrt{1-v^{2}}\int_{\mathbb{R}}\chi_{j}(v;t,x)\partial^{2}_{t}h(t,x)\partial_{x}h(t,x)\,dx\\
     \begin{aligned}
     =&\int_{\mathbb{R}}\chi_{j}(v;t,x) U^{(3)}\left(H_{0,1}\left(\frac{x-vt}{\sqrt{1-v^{2}}}\right)+H_{-1,0}\left(\frac{x+vt}{\sqrt{1-v^{2}}}\right)\right) H^{'}_{0,1}\left(\frac{x-vt}{\sqrt{1-v^{2}}}\right)h(t,x)^{2}\,dx\\ &{+}\int_{\mathbb{R}}\chi_{j}(v;t,x) U^{(3)}\left(H_{0,1}\left(\frac{x-vt}{\sqrt{1-v^{2}}}\right)+H_{-1,0}\left(\frac{x+vt}{\sqrt{1-v^{2}}}\right)\right) H^{'}_{-1,0}\left(\frac{x+vt}{\sqrt{1-v^{2}}}\right)h(t,x)^{2}\,dx\\
&{+}O\left(\norm{ \chi^{'}}_{L^{\infty}_{x}(\mathbb{R})}\frac{1}{vt }\norm{(h(t),\partial_{t}h(t))}_{H^{1}_{x}\times L^{2}_{x}}^{2}+\norm{f(t)}_{L^{2}_{x}}\norm{(h(t),\partial_{t}h(t))}_{H^{1}_{x}\times L^{2}_{x}}\right).
\end{aligned}
\end{multline*}
From the definitions of $\chi_{1}(v;t,x)$ and $\chi_{2}(v;t,x),$ we can verify for all $t>1$ that 
\begin{align*}
H^{'}_{0,1}\left(\frac{x-vt}{\sqrt{1-v^{2}}}\right)\chi_{1}(v;t,x)<&\sqrt{2}\exp\left(-\frac{\sqrt{2}vt(1+2\times 10^{-3})}{\sqrt{1-v^{2}}}\right),\\
    \dot H_{-1,0}\left(\frac{x+vt}{\sqrt{1-v^{2}}}\right)\chi_{2}(v;t,x)<&\sqrt{2}\exp\left(-\frac{\sqrt{2}vt(1-10^{-3})^{2}}{\sqrt{1-v^{2}}}\right),
\end{align*}
In conclusion, we obtain that
\begin{multline}\label{ultimo}
     \sum_{j=1}^{2}v\int_{\mathbb{R}}\chi_{j}(v;t,x)(-1)^{j}\partial^{2}_{t}h(t,x)\partial_{x}h(t,x)\,dx\\
     \begin{aligned}
     =&\frac{v}{2\sqrt{1-v^{2}}}\int_{\mathbb{R}} U^{(3)}\left(H_{0,1}\left(\frac{x-vt}{\sqrt{1-v^{2}}}\right)+H_{-1,0}\left(\frac{x+vt}{\sqrt{1-v^{2}}}\right)\right) H^{'}_{0,1}\left(\frac{x-vt}{\sqrt{1-v^{2}}}\right)h(t,x)^{2}\,dx\\
     &{-}\frac{v}{2\sqrt{1-v^{2}}}\int_{\mathbb{R}} U^{(3)}\left(H_{0,1}\left(\frac{x-vt}{\sqrt{1-v^{2}}}\right)+H_{-1,0}\left(\frac{x+vt}{\sqrt{1-v^{2}}}\right)\right) H^{'}_{-1,0}\left(\frac{x+vt}{\sqrt{1-v^{2}}}\right)h(t,x)^{2}\,dx\\
     &{+}O\left(\norm{\dot \chi}_{L^{\infty}_{x}(\mathbb{R})}\frac{1}{t }\norm{(h(t),\partial_{t}h(t))}_{H^{1}_{x}\times L^{2}_{x}}^{2}+v\norm{f(t)}_{L^{2}_{x}}\norm{(h(t),\partial_{t}h(t))}_{H^{1}_{x}\times L^{2}_{x}}\right)\\&{+}O\left(v \exp\left(-\frac{\sqrt{2}vt(1-10^{-3})^{2}}{\left(1-v^{2}\right)^{\frac{1}{2}}}\right)\norm{h(t,x)}_{H^{1}_{x}(\mathbb{R})}^{2}\right).
\end{aligned}
\end{multline}
So, using estimate \eqref{ultimo}, Lemma \ref{solutionp} will follow from the sum of \eqref{este1} and \eqref{este2}.
\end{proof}
\begin{lemma}\label{projl}
There is $C>0,$ such that, for any $0<v<1,$ if $f(t,x)\in L^{\infty}_{t}(\mathbb{R};H^{1}_{x}(\mathbb{R}))$ and $h(t,x) \in L^{\infty}_{t}\left(\mathbb{R}_{\geq 1};H^{1}_{x}(\mathbb{R})\right)\cap C^{1}_{t}\left(\mathbb{R}_{\geq 1};L^{2}_{x}(\mathbb{R})\right)$ is a solution of the integral equation associated to the following partial differential equation
\begin{equation*}
    \partial^{2}_{t}h(t,x)-\partial^{2}_{x}h(t,x)+U^{''}\left(H_{0,1}\left(\frac{x-vt}{\sqrt{1-v^{2}}}\right)+H_{-1,0}\left(\frac{x+vt}{\sqrt{1-v^{2}}}\right)\right)h(t,x)=f(t,x),
\end{equation*}
for some boundary condition $(h(t_{0}),\partial_{t}h(t_{0}))\in H^{1}_{x}(\mathbb{R})\times L^{2}_{x}(\mathbb{R}),$ then for
$\overrightarrow{h}(t)=(h(t,x),\partial_{t}h(t,x))$ we have
\begin{align*}
    \left \vert\frac{d}{dt}\left\langle\overrightarrow{h}(t),\,\psi^{0}_{-1,0}(v;t)\right\rangle \right \vert\leq &C\left[\norm{f(t)}_{L^{2}_{x}(\mathbb{R})}+\norm{\overrightarrow{h(t)}}_{H^{1}_{x}(\mathbb{R})\times L^{2}_{x}(\mathbb{R})}\exp\left(\frac{-2\sqrt{2}v t}{(1-v^{2})^{\frac{1}{2}}}\right)\right],\\
    \left \vert\frac{d}{dt}\left\langle\overrightarrow{h}(t),\,\psi^{0}_{0,1}(v;t)\right\rangle \right \vert\leq &  C\left[\norm{f(t)}_{L^{2}_{x}(\mathbb{R})}+\norm{\overrightarrow{h(t)}}_{H^{1}_{x}(\mathbb{R})\times L^{2}_{x}(\mathbb{R})}\exp\left(\frac{-2\sqrt{2}v t}{(1-v^{2})^{\frac{1}{2}}}\right)\right],
\end{align*}
and,
\begin{gather*}    
    \left\vert \frac{d}{dt}\left\langle\overrightarrow{h}(t),\,\psi^{1}_{-1,0}(v;t)\right\rangle+(1-v^{2})^{\frac{1}{2}}\left\langle\overrightarrow{h}(t),\,\psi^{0}_{-1,0}(v;t)\right\rangle\right\vert\leq  C\Bigg[\norm{f(t)}_{L^{2}_{x}}\\{+}\norm{\overrightarrow{h(t)}}_{H^{1}_{x}\times L^{2}_{x}}(\vert t\vert v+1)\exp\left(\frac{-2\sqrt{2}v t}{(1-v^{2})^{\frac{1}{2}}}\right)\Bigg],\\
    \left\vert \frac{d}{dt}\left\langle\overrightarrow{h}(t),\,\psi^{1}_{0,1}(v;t)\right\rangle+(1-v^{2})^{\frac{1}{2}}\left\langle\overrightarrow{h}(t),\,\psi^{0}_{0,1}(v;t)\right\rangle\right\vert\leq C\Bigg[\norm{f(t)}_{L^{2}_{x}}\\{+}\norm{\overrightarrow{h(t)}}_{H^{1}_{x}\times L^{2}_{x}}(\vert t\vert v+1)\exp\left(\frac{-2\sqrt{2}v t}{(1-v^{2})^{\frac{1}{2}}}\right)\Bigg],
\end{gather*}
\end{lemma}
\begin{proof}[Proof of Lemma \ref{projl}]
It follows directly from the identity
\begin{equation}
    \frac{d}{dt}\overrightarrow{h}(t)=J  L_{+,-} \overrightarrow{h}(t)+
    \begin{bmatrix}
     0\\
     f(t,x)
    \end{bmatrix},
\end{equation}
and from Lemma \ref{projest}.
\end{proof}
\begin{proof}[Proof of Theorem \ref{purepure}]
For $T_{0}\geq \frac{4\ln{\left(\frac{1}{v}\right)}}{v},$ we consider similarly to \cite{multison} the following norms denoted by
\begin{equation*}
    \norm{u}_{L^{2}_{v,T_{0}}}=\sup_{t\geq T_{0}}e^{vt}\norm{u(t,x)}_{L^{2}_{x}(\mathbb{R})},\, \norm{u}_{H^{1}_{v,T_{0}}}=\sup_{t\geq T_{0}}e^{vt}\left[\norm{u(t,x)}_{H^{1}_{x}(\mathbb{R})}^{2}+\norm{\partial_{t}u(t,x)}_{L^{2}_{x}(\mathbb{R})}^{2}\right]^{\frac{1}{2}}.
\end{equation*}
Next, from Lemma \ref{projl}, we can verify using the Fundamental Theorem of Calculus that there is a constant $C>1$ such that if $v\ll 1,$ then for any $t\geq T_{0}$ we have that
\begin{align}\label{ee1}
    \left\vert\left\langle\overrightarrow{h}(t),\,\psi^{0}_{-1,0}(v;t)\right\rangle\right\vert\leq & C\left[\norm{f}_{L^{2}_{v,T_{0}}}\frac{e^{-vt}}{v}+\norm{h}_{H^{1}_{v,T_{0}}}\frac{e^{-(2\sqrt{2}+1)vt}}{v}\right],\\ \label{ee2}
    \left\vert\left\langle\overrightarrow{h}(t),\,\psi^{1}_{-1,0}(v;t)\right\rangle\right\vert\leq &C\left[\norm{f}_{L^{2}_{v,T_{0}}}\frac{e^{-vt}}{v^{2}}+\norm{h}_{H^{1}_{v,T_{0}}}te^{{-}\left(2\sqrt{2}+1\right)vt}+\norm{h}_{H^{1}_{v,T_{0}}}\frac{e^{{-}\left(2\sqrt{2}+1\right)vt}}{v^{2}}\right],
\end{align}
and that
\begin{align}\label{eee1}
 \left\vert\left\langle\overrightarrow{h}(t),\,\psi^{0}_{0,1}(v;t)\right\rangle\right\vert\leq &C\left[\norm{f}_{L^{2}_{v,T_{0}}}\frac{e^{-vt}}{v}+\norm{h}_{H^{1}_{v,T_{0}}}\frac{e^{-(2\sqrt{2}+1)vt}}{v}\right],\\ \label{eee2}
   \left\vert \left\langle\overrightarrow{h}(t),\,\psi^{1}_{0,1}(v;t)\right\rangle\right\vert\leq & C\left[\norm{f}_{L^{2}_{v,T_{0}}}\frac{e^{-vt}}{v^{2}}+\norm{h}_{H^{1}_{v,T_{0}}}te^{{-}\left(2\sqrt{2}+1\right)vt}+\norm{h}_{H^{1}_{v,T_{0}}}\frac{e^{{-}\left(2\sqrt{2}+1\right)vt}}{v^{2}}\right].
\end{align}
\par Also, from Lemma \ref{solutionp}, we can verify using the Fundamental Theorem of Calculus for any $t\geq T_{0}$ that there is a constant $K\geq 1$ such that if $v\ll1,$ then
\begin{multline}\label{integest}
    \int_{t}^{+\infty}\left\vert \frac{d}{ds}Q(s,h) \right\vert\,ds\leq K\Bigg[\frac{e^{-2vt}}{v}\norm{f}_{L^{2}_{v,T_{0}}}\norm{h}_{H^{1}_{v,T_{0}}}\\{+}\norm{h}_{H^{1}_{v,T_{0}}}^{2}\left(\frac{e^{-2vt}}{vt}+e^{-t(2v+\sqrt{2}v(1-10^{-3})^{2})}\right)\Bigg]
\end{multline}
In conclusion, similarly Step $1$ in the proof of Lemma $3.1$ of \cite{multison}, we deduce using the estimates \eqref{ee1}, \eqref{eee1}, \eqref{ee2}, \eqref{eee2} with Lemma \ref{coerQ} and the estimate above \eqref{integest} that there exists a new constant $C>1$ such that for any $t\geq T_{0}$ and $v\ll 1$ we have
\begin{equation}\label{globb}
    \norm{h}_{H^{1}_{v,T_{0}}}^{2}\leq \frac{C}{v^{4}} \norm{f}_{L^{2}_{v,T_{0}}}^{2}.
\end{equation}
The fact that the constant $C$ in \eqref{globb} is independent of $v$ follows from $T_{0}\geq \frac{4\ln{\left(\frac{1}{v}\right)}}{v},$ which implies that
\begin{equation*}
    \frac{e^{-2vt}}{v^{4}}+\frac{e^{-2vt}}{vt}\ll v^{4}.
\end{equation*}
\par We also observe that if $(g_{1}(t,x),\partial_{t}g_{1}(t,x))$ and $(g_{2}(t,x),\partial_{t}g_{2}(t,x))$ are in the space $(g(t),\partial_{t}g(t))\in H^{1}_{x}(\mathbb{R})\times L^{2}_{x}(\mathbb{R})$ such that
\begin{equation}
    \norm{(g(t),\partial_{t}g(t))}_{L^{\infty}\left([T_{0},+\infty],H^{1}_{x}\times L^{2}_{x}\right)}\leq 1,
\end{equation}
then, since $U\in C^{\infty},$ we can verify that the following function
\begin{multline}
    N(v,\overrightarrow{g})(t,x)\\
    \begin{aligned}
    =&U^{'}\left(H_{-1,0}\left(\frac{x+vt}{\sqrt{1-v^{2}}}\right)+H_{0,1}\left(\frac{x-vt}{\sqrt{1-v^{2}}}\right)+g(t,x)\right) - U^{'}\left(H_{-1,0}\left(\frac{x+vt}{\sqrt{1-v^{2}}}\right)\right)\\&{-} U^{'}\left(H_{0,1}\left(\frac{x-vt}{\sqrt{1-v^{2}}}\right)\right)
    -U^{''}\left(H_{-1,0}\left(\frac{x+vt}{\sqrt{1-v^{2}}}\right)+H_{0,1}\left(\frac{x-vt}{\sqrt{1-v^{2}}}\right)\right)g(t,x)
\end{aligned}
\end{multline}
satisfies for some new constant $C\geq 1$ and any $v\ll 1$
\begin{equation*}
    \norm{ N(v,\overrightarrow{g_{1}(t)})-N(v,\overrightarrow{g_{2}(t)})}_{H^{1}_{x}}\leq C\left[\norm{g_{1}(t)}_{H^{1}_{x}}+\norm{g_{2}(t)}_{H^{1}_{x}}\right]\norm{g_{1}(t)-g_{2}(t)}_{H^{1}_{x}},
\end{equation*}
which implies the following estimate given by
\begin{equation}\label{nonlinearity}
    \norm{N(v,\overrightarrow{g_{1}(t)})-N(v,\overrightarrow{g_{2}(t)})}_{H^{1}_{v,T_{0}}}\leq C e^{-vt}\left[\norm{g_{1}}_{H^{1}_{v,T_{0}}}+\norm{g_{2}}_{H^{1}_{v,T_{0}}}\right]\norm{g_{1}-g_{2}}_{H^{1}_{v,T_{0}}}.
\end{equation}
In conclusion, by repeating the argument of the proof of proposition $3.6$ of \cite{multison}, we can verify using the Lipschtiz estimate of \eqref{nonlinearity} and estimate \eqref{globb} that if $T_{0}\geq \frac{4\ln{\left(\frac{1}{v}\right)}}{v}$ and $v\ll 1,$ then there exists a map
\begin{equation}\label{SS}
    S:\{u\in H^{1}_{v,T_{0}}\vert \norm{u}_{H^{1}_{v,T_{0}}}\leq 1\}\to \{u\in H^{1}_{v,T_{0}}\vert \norm{u}_{H^{1}_{v,T_{0}}}\leq 1\}
\end{equation}
such that $\mu(t,x)=S(u)(t,x)$ is the unique solution of the equation
\begin{equation}\label{pppp}
    \partial^{2}_{t}\mu(t,x)-\partial^{2}_{x}\mu(t,x)+ U^{''}\left(H_{-1,0}\left(\frac{x+vt}{\sqrt{1-v^{2}}}\right)+H_{0,1}\left(\frac{x-vt}{\sqrt{1-v^{2}}}\right)\right)\mu(t,x)=N(v,\overrightarrow{\mu})(t,x),
\end{equation}
such that $\mu \in H^{1}_{v,T_{0}}.$ Indeed, the uniqueness is guaranteed by estimate \eqref{globb} and from estimates \eqref{globb} and \eqref{nonlinearity} we have that the map $S$ is a contraction in the set
\begin{equation*}
B=\{u\in H^{1}_{v,T_{0}}\vert \norm{u}_{H^{1}_{v,T_{0}}}\leq 1\},    
\end{equation*}
and so, Theorem \ref{purepure} follows similarly to the proof of Proposition $3.6$ of \cite{multison} by using the Banach's fixed point theorem.
\end{proof}

\end{document}